\documentclass[a4paper, 10pt, american]{amsart}

\usepackage{times,latexsym,amssymb}
\usepackage{amsmath,amsthm,bm}
\usepackage{color}
\usepackage{bm}
\usepackage[colorlinks,pdfpagelabels,pdfstartview = FitH,bookmarksopen
= true,bookmarksnumbered = true,linkcolor = blue,plainpages =
false,hypertexnames = false,citecolor = red,pagebackref=false]{hyperref}

\usepackage{amsbsy}
\usepackage{amstext}
\usepackage{amssymb}
\usepackage{esint}
\usepackage{stmaryrd}
\usepackage[pdftex]{graphicx}
\usepackage{floatflt}

\allowdisplaybreaks
\newtheorem{theorem}{Theorem}
\newtheorem{lemma}[theorem]{Lemma}
\newtheorem{definition}[theorem]{Definition}
\newtheorem{proposition}[theorem]{Proposition}

\newtheorem{remark}[theorem]{Remark}
\numberwithin{theorem}{section}
\numberwithin{equation}{section}

\newcommand{\mint}{- \mskip-19,5mu \int}

\def\N{\mathbb{N}}
\def\R{\mathbb{R}}

\def\bH{\boldsymbol{\mathcal H}}
\def\bC{{\mathcal C}}
\def\bQ{{\mathcal Q}}

\newcommand{\dHn}{\mathrm{d}\mathcal{H}^{n-1}}

\renewcommand{\d}{\mathrm{d}}
\newcommand{\dx}{\mathrm{d}x}
\newcommand{\dy}{\mathrm{d}y}

\newcommand{\dt}{\mathrm{d}t}
\newcommand{\ds}{\mathrm{d}s}
\newcommand{\dtau}{\mathrm{d}\tau}

\newcommand{\dr}{\mathrm{d}r}

\renewcommand{\epsilon}{\varepsilon}

\DeclareMathOperator{\spt}{spt}
\DeclareMathOperator{\trace}{trace}
\DeclareMathOperator{\Div}{div}
\DeclareMathOperator{\diam}{diam}
\DeclareMathOperator{\dist}{dist}
\DeclareMathOperator{\loc}{loc}
\renewcommand{\epsilon}{\varepsilon}
\newcommand{\eps}{\varepsilon}
\renewcommand{\rho}{\varrho}

\def\eqn#1$$#2$${\begin{equation}\label#1#2\end{equation}}

\newcommand{\wto}{\rightharpoondown}
\newcommand{\wsto}{\overset{\raisebox{-1ex}{\scriptsize $*$}}{\rightharpoondown}}

\newcommand{\w}{\hat u}

\def\Xint#1{\mathchoice
    {\XXint\displaystyle\textstyle{#1}}%
    {\XXint\textstyle\scriptstyle{#1}}%
    {\XXint\scriptstyle\scriptscriptstyle{#1}}%
    {\XXint\scriptscriptstyle\scriptscriptstyle{#1}}%
    \!\int}
\def\XXint#1#2#3{\setbox0=\hbox{$#1{#2#3}{\int}$}
    \vcenter{\hbox{$#2#3$}}\kern-0.5\wd0}
\def\bint{\Xint-}
\def\dashint{\Xint{\raise4pt\hbox to7pt{\hrulefill}}}

\def\Xiint#1{\mathchoice
    {\XXiint\displaystyle\textstyle{#1}}%
    {\XXiint\textstyle\scriptstyle{#1}}%
    {\XXiint\scriptstyle\scriptscriptstyle{#1}}%
    {\XXiint\scriptscriptstyle\scriptscriptstyle{#1}}%
    \!\iint}
\def\XXiint#1#2#3{\setbox0=\hbox{$#1{#2#3}{\iint}$}
    \vcenter{\hbox{$#2#3$}}\kern-0.5\wd0}
\def\biint{\Xiint{-\!-}}

\subjclass[2010]{35B65, 35K65, 35K40, 35K55}

\author[V. B\"ogelein]{Verena B\"{o}gelein}
\address{Verena B\"ogelein\\
Fachbereich Mathematik, Universit\"at Salzburg\\
Hellbrunner Str. 34, 5020 Salzburg, Austria}
\email{verena.boegelein@sbg.ac.at}

\author[F. Duzaar]{Frank Duzaar}
\address{Frank Duzaar\\
Department Mathematik, Universit\"at Erlangen--N\"urnberg\\
Cauerstrasse 11, 91058 Erlangen, Germany}
\email{frank.duzaar@fau.de}

\author[N. Liao]{Naian Liao}
\address{Naian Liao\\
Fachbereich Mathematik, Universit\"at Salzburg\\
Hellbrunner Str. 34, 5020 Salzburg, Austria}
\email{naian.liao@sbg.ac.at}

\author[C. Scheven]{Christoph Scheven}
\address{Christoph Scheven\\ Fakult\"at f\"ur Mathematik, 
Universit\"at Duisburg-Essen\\45117 Essen, Germany}
\email{christoph.scheven@uni-due.de}

\begin{document}
\newcommand{\cl}{\centerline}
\newcommand{\sms}{\smallskip}
\newcommand{\ms}{\medskip}
\newcommand{\bs}{\bigskip}
\newcommand{\noi}{\noindent}
\newcommand{\itl}[1]{\textit{#1}}
\newcommand{\blf}[1]{\textbf{#1}}
\newcommand{\dsty}{\displaystyle}
\newcommand{\txty}{\textstyle}
\newcommand{\ssty}{\scriptstyle}
\newcommand{\tty}{\texttt}


\newcommand\Par{\mathhexbox278\,}


\newcommand{\al}{\alpha}
\newcommand{\Al}{\Alpha}
\newcommand{\be}{\beta}
\newcommand{\Be}{\Beta}
\newcommand{\Gm}{\Gamma}
\newcommand{\gm}{\gamma}
\newcommand{\dl}{\delta}
\newcommand{\Dl}{\Delta}
\newcommand{\lm}{\lambda}
\newcommand{\Lm}{\Lambda}
\newcommand{\kp}{\kappa}
\newcommand{\varep}{\varepsilon}
\newcommand{\vp}{\varphi}
\newcommand{\sig}{\sigma}
\newcommand{\Sig}{\Sigma}
\newcommand{\om}{\omega}
\newcommand{\Om}{\Omega}
\newcommand{\uom}{\mbox{\boldmath$\omega$}}
\newcommand{\btau}{\mbox{\boldmath$\tau$}}
\newcommand{\bnu}{\mbox{\boldmath$\nu$}}
\newcommand{\up}{\upsilon}
\newcommand{\z}{\zeta}


\newcommand{\df}[1]{\buildrel\mbox{\small def}\over{#1}}
\newcommand{\op}[1]{\buildrel\mbox{\tiny o}\over{#1}}
\newcommand{\db}{\prime\prime}
\newcommand{\bsl}{\backslash}
\newcommand{\lb}{\lbrack\!\lbrack}
\newcommand{\rb}{\rbrack\!\rbrack}
\newcommand\la{\langle}
\newcommand\ra{\rangle}
\newcommand{\ev}{\equiv}
\newcommand{\nev}{\not\equiv}
\newcommand{\qq}{\mathbb{Q}}
\newcommand{\zz}{\mathbb{Z}}
\newcommand{\rr}{\mathbb{R}}
\newcommand{\rn}{\rr^N}
\newcommand{\cc}{\mathbb{C}}
\newcommand{\id}{\mathbb{I}}
\newcommand{\bo}{\mathbb{O}}

\newcommand{\amsb}[1]{\mathbb{#1}}
\newcommand{\mcl}[1]{\mathcal{#1}}
\newcommand{\bl}[1]{\mathbf{#1}}
\newcommand{\ov}[1]{\overline{#1}}
\newcommand{\wt}[1]{\widetilde{#1}}
\newcommand{\wh}[1]{\widehat{#1}}

\newcommand{\llra}{\leftrightarrow}
\newcommand{\lra}{\longrightarrow}
\newcommand{\LLR}{\Longleftrightarrow}
\newcommand{\LRA}{\Longrightarrow}
\newcommand{\LLA}{\Longleftarrow}


\newcommand{\bbox}{\vrule height.6em width.6em 
depth0em} 
\newcommand{\os}{\vbox{\hrule \hbox{\vrule 
height.6em depth0pt 
\hskip.6em \vrule height.6em depth0em}
\hrule}} 


\newcommand{\dvg}{\operatorname{div}}
\newcommand{\curl}{\operatorname{curl}}
\newcommand{\supp}{\operatorname{supp}}
\newcommand{\essup}{\operatornamewithlimits{ess\,sup}}
\newcommand{\essinf}{\operatornamewithlimits{ess\,inf}}
\newcommand{\essosc}{\operatornamewithlimits{ess\,osc}}
\newcommand{\osc}{\operatornamewithlimits{osc}}
\newcommand{\sign}{\operatorname{sign}}
%


\newcommand{\overlim}{\mathop{\overline{\lim}}\limits}
\newcommand{\underlim}{\mathop{\underline{\lim}}\limits}
\newcommand{\ttop}[2]{\genfrac{}{}{0pt}{}{#1}{#2}}
\newcommand{\bcu}{\mathop{\txty{\bigcup}}\limits}
\newcommand{\bca}{\mathop{\txty{\bigcap}}\limits}
\newcommand{\bsu}{\mathop{\txty{\sum}}\limits}
\newcommand{\pro}{\mathop{\txty{\prod}}\limits}


\newcommand{\pl}{\partial}
\newcommand{\ptt}{\frac{\pl}{\pl t}}
\newcommand{\ppx}{\frac\pl{\pl x}}
\newcommand{\dds}{\frac d{ds}}
\newcommand{\ddt}{\frac d{dt}}

\newcommand{\intl}{\int\limits}
\newcommand{\iintl}{\iint\limits}
\def\Xint#1{\mathchoice
    {\XXint\displaystyle\textstyle{#1}}%
    {\XXint\textstyle\scriptstyle{#1}}%
    {\XXint\scriptstyle\scriptscriptstyle{#1}}%
    {\XXint\scriptscriptstyle\scriptscriptstyle{#1}}%
    \!\int}
\def\XXint#1#2#3{\setbox0=\hbox{$#1{#2#3}{\int}$}
    \vcenter{\hbox{$#2#3$}}\kern-0.5\wd0}
\def\bint{\Xint-}
\def\dashint{\Xint{\raise4pt\hbox to7pt{\hrulefill}}}
\def\dashiint{\bint\kern-0.15cm\bint}

\newcommand{\ovl}[3]{\int_{#1}^{#2}\kern-#3pt\raise4pt\hbox to7pt{\hrulefill}\ }

\newcommand{\ovll}[3]{\intl_{#1}^{#2}\kern-#3pt\raise4pt\hbox to7pt{\hrulefill}\ }

\newcommand{\tvl}[2]{\iint_{#1}\kern-#2pt\raise4pt\hbox to7pt{\hrulefill}\ }



\newcommand{\omt}{\Om_T}
\newcommand{\plo}{\partial\Omega}
\newcommand{\ovo}{\bar{\Om} }

%
\newcommand{\ci}[1]{C^\infty\!\left({#1}\right)}
\newcommand{\cio}[1]{C_o^\infty\!\left({#1}\right)}
\newcommand{\lloc}[1]{L_{\loc}\!\left({#1}\right)}
\newcommand{\xy}{|x-y|}


\newcommand{\intom}{\intl_{\Om}}
\newcommand{\intbo}{\intl_{\plo}}
\newcommand{\inom}{\int_{\Om}}
\newcommand{\inbo}{\int_{\plo}}
\newcommand{\intrn}{\intl_{\rn}}


\newcommand{\bye}{\end{document}}



%
%
%

%
%
%
%
%

%
%

\title[Boundary regularity for parabolic systems]{
Boundary regularity  for parabolic systems\\ in convex domains}
\date{\today}

\begin{abstract}
In a cylindrical space-time domain with a convex, spatial base,
we establish a local Lipschitz estimate for weak solutions to parabolic systems with Uhlenbeck structure up to the lateral boundary,
provided homogeneous Dirichlet data are assumed on that part of the lateral boundary.
\end{abstract}
\maketitle
\tableofcontents

\section{Introduction}
This paper studies boundary regularity of  weak solutions $u\colon \Omega_T\to\R^N$, $N\ge 1$, to  nonlinear parabolic systems of the type
\begin{equation}\label{system}
    \partial_t u^i - \sum_{\alpha=1}^{n}
    \big[ \mathbf a\big(|Du|\big) u^i_{x_{\alpha}}\big] _{x_{\alpha}}=b^i \quad
    \mbox{for $i=1,\dots, N$,}
\end{equation} 
in a space-time cylinder $\Omega_T=\Omega\times (0,T)$, where
$\Omega\subset\R^n$ is a bounded open convex
set, $n\ge 2$ and $T>0$.
We assume that $u$ satisfies a homogeneous Dirichlet boundary condition
on some part of the lateral boundary $(\partial\Omega)_T=\partial\Omega\times (0,T)$.
The nonlinearity
$\mathbf a\colon\R_{\ge 0}\to\R_{\ge 0}$   fulfills a growth condition of the type
$\mathbf a (r)+r\mathbf a'(r)\approx r^{p-2}$ for some growth
exponent $p>\frac{2n}{n+2}$. As such the diffusion part in \eqref{system}
is said to have the Uhlenbeck structure. For the  right-hand side we require $b\in L^\sigma (\Omega_T,\R^N)$
for some $\sigma >n+2$.

The primary purpose of this paper is to establish 
$$
	Du\in L^\infty \big( \Omega_T\cap Q_{\rho}(z_o),\R^{Nn}\big)
$$
whenever $u$ is a weak solution to the system \eqref{system}
satisfying $u\equiv 0$ on the subset $(\partial\Omega)_T\cap Q_{2\rho}(z_o)$ of
the lateral boundary. Here $Q_\varrho(z_o):=B_{\varrho}(x_o)\times(t_o-\rho^2,t_o)$
for some $z_o=(x_o,t_o)\in(\pl\Om)_T$. We only require that  $\Omega$ is a bounded open convex set.
 No further regularity of $\partial\Omega$ is assumed.
The qualitative assertion is confirmed by a quantitative
$L^\infty$-estimate for the spatial gradient $Du$.

Regularity problems for nonlinear equations or systems of the parabolic $p$-Laplacian type and their stationary counterparts
were very difficult to access in the past. The interior $C^{1,\lambda}$-regularity had been  longstanding open problems. 
The first major breakthrough was achieved  by Uraltseva \cite{Uraltseva} in 1968. She showed that solutions to $p$-Laplacian equations, whose model is given by
\begin{equation}\label{i:p-Laplace}
	-\Div\big( |Du|^{p-2}Du\big)=0\qquad\mbox{in $\Omega$,}
\end{equation}
are of class $C^{1,\lambda}$ in the interior of the domain $\Omega\subset\R^n$. 
This result was generalized  in 1977 by Uhlenbeck in her famous paper \cite{Uhlenbeck} to the $p$-Laplacian type systems
(i.e. elliptic version of \eqref{system})
\begin{equation}\label{i:p-Laplace-type}
	-
	\sum_{\alpha=1}^{n}
    	\big[ \mathbf a\big(|Du|\big) u^i_{x_{\alpha}}\big] _{x_{\alpha}}=0
	\quad
    	\mbox{for $i=1,\dots, N$.}
\end{equation}
 More general structures, replacing $|Du|^2$ by some quadratic expression
$Q(Du,Du)$ and including a sufficiently regular dependence on $x\in\Omega$,  have been considered by Tolksdorf \cite{Tolksdorf:1983}. Roughly speaking, in the weak formulation of  \eqref{i:p-Laplace-type} the nonlinear term $\mathbf a\big(|Du|\big)Du$ is replaced by $\mathbf a\big(x, Q(Du, Du)^\frac12\big)Q(Du,\,\cdot \,)$.
Similar $C^{1,\lambda}$-regularity results have been shown for
minimizers of  integral functionals with $p$-growth. The degenerate
case with growth exponent  $p\ge 2$ goes back to Giaquinta \& Modica
\cite{GiaquintaModica:1986-a}, while the singular case $1<p<2$ was
treated by Acerbi \& Fusco \cite{AcerbiFusco}.  For systems of
$p$-Laplacian type as in \eqref{i:p-Laplace-type}, sharp pointwise interior gradient  bounds in terms of  a nonlinear potential of the right-hand side $b$ have been established in \cite{DuzaarMingione:2010}.

Regarding the boundary regularity for $p$-Laplacian type systems the picture is  less complete.
Global $C^{1,\lambda}$-regularity is known only for homogeneous Dirichlet and Neumann boundary data; see Hamburger \cite{Hamburger}.  For general boundary data, it is still an open problem. However, local $L^\infty$-gradient bounds (Lipschitz estimates at the boundary)
have been established  by Foss \cite{Foss} for minimizers to
asymptotically regular integral functionals on domains with $C^{1,\lambda}$-boundary; see also
\cite{Foss-Passarelli-Verde, Lieberman}.  Again for homogeneous
Dirichlet or Neumann data, 
global Lipschitz estimates (in terms of the right-hand side $b$ of \eqref{i:p-Laplace-type} and under minimal assumptions on the regularity
of $\partial \Omega$ and $b$) have been proved by Cianchi \& Maz'ya in \cite{ Cianchi-Mazya-2, Cianchi-Mazya-1}. These results are global in nature and only valid if $u$ or its outer normal derivative $\partial_\nu u$ vanishes on the whole boundary $\partial\Omega$.
It is noteworthy that their results hold for  convex domains in particular. In contrast to
these global results, Banerjee \& Lewis \cite{BanerjeeLewis:2014} established   
local boundary Lipschitz estimates with homogeneous data
for convex domains.  Their result is of local nature as
 they only require the homogeneous boundary condition 
 on a part 
 of the boundary. 
 Inspired by the technique introduced in \cite{BanerjeeLewis:2014}, Marcellini, 
the first two, and the last author were able to establish the first local boundary Lipschitz estimate for integral functionals with
non-standard $p,q$-growth; see \cite{BDMS}. 

The interior $C^{1,\lambda}$ regularity theory for the parabolic $p$-Laplacian type systems \eqref{system} is a fundamental achievement by  DiBenedetto \& Friedman; see
\cite{DiBenedetto-Friedman, DiBenedetto-Friedman2,
  DiBenedetto-Friedman3} and the monograph \cite[Chapters~VIII, IX, X]{DB};
see also Chen \cite{Chen} and Wiegner \cite{Wiegner}. 
For systems without Uhlenbeck structure of the type
\begin{equation}
	\partial_t u^i - \sum_{\alpha=1}^{n}
    \big[ \mathbf  a_\alpha^i(Du) \big] _{x_{\alpha}}=0 \quad
    \mbox{for $i=1,\dots, N$,}
\end{equation}
with a nonlinear diffusion  $\mathbf a$ that behaves asymptotically like the $p$-Laplacian
at the
origin, i.e.~$s^{1-p}\mathbf a(s\xi)\to |\xi|^{p-2}\xi$ in the limit $s\downarrow 0$ for any $\xi$, partial $C^{1,\lambda}$-regularity has been established by B\"ogelein \& Duzaar \& Mingione \cite{BoeDuMin}. 

In contrast to the interior regularity  theory, the boundary regularity is largely an open problem. 
There were two results achieved by Chen \& DiBenedetto in \cite{Chen-DiBenedetto} for the parabolic systems with the Uhlenbeck structure
in $C^{1,\lambda}$-domains.
The first was about the H\"older continuity of a solution $u$ up to the lateral boundary with any H\"older exponent in $(0,1)$,
given sufficiently regular boundary data; see also \cite[Chapter X, Theorem 1.1]{DB}.
The second dealt with the H\"older continuity of $Du$ up to the lateral boundary, given homogeneous boundary data;
see  \cite[Chapter X, Theorem 1.2]{DB}. 
The results have been achieved by a  boundary flattening  procedure. This allows us, after freezing  the coefficients, to reduce the problem to the interior case via reflection along  the flat boundary. At this stage it is important that the transformed coefficients admit certain quantitative H\"older-regularity. In the course of the proof the authors established gradient sup-estimates 
for the model case of $p$-Laplacian systems with homogeneous Dirichlet data when the boundary is flat; see  \cite[Propositions 3.1, 3.1']{Chen-DiBenedetto}. These estimates serve as reference inequalities when comparing the solution with the one to the frozen system.
This is why $\partial\Omega$ and $\mathbf g$ are assumed to be $C^{1,\lambda}$. This approach fails in the case of Lipschitz domains.  

Boundary regularity for more general parabolic systems has been considered by the first author in \cite{B-boundary}. The main result ensures the boundedness up to the lateral boundary of the spatial derivative of weak solutions to  asymptotically regular parabolic systems. Roughly speaking this means that the $C^1$-coefficients of the diffusion part behave like the $p$-Laplacian when $|Du|$ becomes large.  The result holds true for  inhomogeneous boundary values. As in \cite{Chen-DiBenedetto}, the proof relies  on a  boundary flattening  procedure and comparison arguments. Therefore, $\partial\Omega$ and $\mathbf g$ have to be of class $C^{1,\lambda}$.

\subsection{Statement of the result}

We assume that the nonlinearity  $\mathbf a\colon\R_{\ge0}\to\R_{\ge0}$ is 
 of class  $C^1(\R_{>0},\R_{>0})$ and satisfies 
\begin{equation}
  \label{assumption:a(0)}
  \lim_{r\downarrow0}r\mathbf{a}(r)=0.
\end{equation}
Moreover, $\mathbf a$ fulfills a standard monotonicity and $p$-growth condition 
 \begin{equation}\label{assumption:a'}
  m(\mu^2+r^{2}) ^{\frac{p-2}{2}}
  \leq 
  \mathbf{a}(r)+r\mathbf{a}^{\prime }(r) 
  \leq 
  M(\mu^2+r^{2}) ^{\frac{p-2}{2}}
  \qquad\mbox{for all }r>0,
\end{equation}
with positive constants $0<m\le M$, some parameter $\mu\in[0,1]$, and  some growth exponent  
$\frac{2n}{n+2}<p<\infty$. Note that in the case $\mu>0$ the parabolic system 
\eqref{system} is non-degenerate, while for $\mu=0$ the diffusion part becomes either degenerate or singular at points with
$|Du|=0$.  For the inhomogeneity  $b=(b^1,\dots,b^N)\colon\Omega_T\to\R^N$ we assume the integrability condition
\begin{equation}
  \label{assumption:b}
 \mbox{$ b\in L^\sigma(\Omega_T,\R^N)\;$ for some $\sigma>n+2$.}
\end{equation}

\begin{definition}\label{def:weak-sol}\upshape
Assume that the nonlinearity $\mathbf a$  and the inhomogeneity $b$ satisfy the assumptions  \eqref{assumption:a(0)} -- \eqref{assumption:b}.  
A map $u \colon \Omega_T\to \R^N$ with
$$
	u\in C^0\big( [0, T ]; L^2 (\Omega ,\R^N)\big)\cap L^p\big(0,T;W^{1,p}(\Omega ,\R^N)\big)
$$	
is called a weak solution to the nonlinear parabolic system  \eqref{system}
if and only if 
\begin{equation}\label{weak-solution}
	\iint_{\Omega_T}\big[ u\cdot\varphi_t -\mathbf a(|Du|)Du\cdot D\varphi\big]\, \dx\dt
	=
	\iint_{\Omega_T}b\cdot\varphi\, \dx\dt
\end{equation}
for any test function $\varphi\in C^\infty_0(\Omega ,\R^N)$.\hfill$\Box$
\end{definition}

Throughout this article $d$ denotes the scaling deficit given by 
\begin{align}\label{def:deficit}
  d:=\left\{
    \begin{array}{cl}
      \frac p2,&\mbox{if $p\ge2$,}\\[1.2ex]
      \frac{2p}{p(n+2)-2n},& \mbox{if $\frac{2n}{n+2}<p<2$.}
    \end{array}
\right.
\end{align}
Now we can state our main result.

\begin{theorem}[$L^\infty$-gradient bound at the boundary]\label{thm:main}
Let $\Omega\subset\R^n$ be an open bounded convex set, and assume that the structural assumptions \eqref{assumption:a(0)} -- \eqref{assumption:b} are in force and let $u\in C^0([0,T];L^2(\Omega,\R^N))\cap
 L^p(0,T;W^{1,p}(\Omega,\R^N))$ be a weak solution to the parabolic
 system \eqref{system} in the sense of Definition~\ref{def:weak-sol}
 satisfying the homogeneous Dirichlet boundary condition
\begin{equation*}
    \mbox{$u\equiv0\;$
    on $(\partial\Omega)_T\cap Q_{2\rho}(z_o)$ in the sense of traces,}
\end{equation*}
where $z_o=(x_o,t_o)$ is a point with space center
$x_o\in\partial\Omega$ and time $t_o\in(0,T)$, and $\rho\in (0,\frac12\sqrt{t_o})$.
Then, we have 
$$
	Du\in L^\infty \big(\Omega_T\cap Q_{\rho/2}(z_o)\big).
$$
Moreover, the following quantitative $L^\infty$-gradient bound 
\begin{align*}
    \sup_{\Omega_T\cap Q_{\rho/2}(z_o)}|Du|
    \le
  	C\bigg[ \Big(1 + \rho^{n+2}\|b\|_{
  	L^\sigma(\Omega_T\cap Q_\rho(z_o))}^{\frac{(n+2)\sigma}{\sigma-n-2}}\Big)
  	\biint_{\Omega_T \cap Q_{\rho}(z_o)}\big(1+|Du|^p\big)
  	\,\dx\dt\bigg]^{\frac{d}{p}}
\end{align*}
holds with a constant $C$ depending on $n,N,p,\sigma,m,M$, and
the geometry of the boundary.
\end{theorem}

\begin{remark}
The dependence of the constant $C$ on the geometry of the boundary can be quantified in terms of the expression $\Theta_{\rho/2}(x_o)$ defined in Section
\ref{sec:convex-domains}.
\end{remark}

We point out that the gradient bound from the preceding theorem is the
exact analogue of the interior gradient bounds in 
\cite[Chapter VIII, Thms. 5.1, 5.2']{DB} for the case
$b=0$.

\subsection{Strategy of the proof}
The usual boundary flattening
 procedure via a local Lipschitz representation of $\partial\Omega$ 
leads to a nonlinearity depending on the gradient of the Lipschitz graph. Due to the limited regularity of $\partial \Omega$,
the  transformed nonlinearity admits only a measurable dependence on the spatial variables.
This prevents the reduction of the problem by freezing, comparing and reflection arguments
to the interior.  Therefore we pursue a different strategy, which is inspired by ideas from Banerjee \& Lewis 
\cite{BanerjeeLewis:2014}; see also \cite{BDMS} for the corresponding 
boundary estimate for minimizers to integral functionals with non-standard $p,q$ growth.
The present paper represents in some sense the parabolic counterpart of \cite{BanerjeeLewis:2014}.

We establish the 
sup-estimate of $Du$ in Theorem~\ref{thm:main}
as the limit of similar estimates for more regular approximating problems.
More precisely, we approximate the convex domain $\Omega$
in Hausdorff distance from outside by smooth convex domains $\Omega_\eps$,
regularize the  nonlinearity $\mathbf{a}$ into $\mathbf{a}_\eps$, extend $u$ and $b$ by zero outside of $\Omega_T$,
and mollify them properly into $g_\eps$ and $b_\eps$. 
Then we solve in $(\Omega_\eps)_T\cap
Q_\rho(x_o,t_o)$ the Cauchy-Dirichlet problem associated to $\mathbf{a}_\eps$ and $b_\eps$, with boundary values $g_\eps$
on the parabolic boundary of $(\Omega_\eps)_T\cap Q_\rho(x_o,t_o)$.
The unique solution $u_\eps$ -- which exists by standard methods  --
 fulfills the Dirichlet  condition $u_\eps=0$ on
$(\partial\Omega_\eps)_T \cap Q_\rho(x_o,t_o)$ by construction. Since the domain of $u_\epsilon$ is smooth we may use a reflection argument together with the interior  $C^{1,\lambda}$-regularity results and the Schauder estimates
for linear parabolic systems to show that these solutions are smooth up to the boundary; see
Appendix~\ref{app:smooth}. 

Next, we prove a quantitative sup-estimate for $Du_{\eps}$, which is uniform
in the parameter $\eps$. Its proof consists of two steps. In the first  step
we derive an energy estimate for the second order derivatives; see Proposition \ref{prop:energy}. 
The key ingredient is a differential geometric
identity from \cite{Grisvard:1985}; see Lemma~\ref{lem:Grisvard}. This identity allows us to exploit the convexity of $\partial\Omega_\eps$ in the sense that
the boundary integral, which cannot be controlled by integrals over $(\Omega_\eps)_T \cap Q_\rho(x_o,t_o)$,
admits a sign and can be discarded in the estimate. 
Based on the energy estimate, we then perform
 a Moser iteration, which leads to the sup-estimate for $Du_{\eps}$ in
Proposition~\ref{prop:apriori}.

Finally we pass to the limit $\eps\downarrow0$, which can be achieved by certain compactness
arguments. 
Decisive
for this argument are the uniform (in $\eps$) energy estimates for the solutions to the regularized problems.
The main obstruction at this stage is that testing the
original parabolic system by the difference $u-u_\eps$ is not allowed,
since $u_\epsilon$ does not admit zero boundary values on
$(\partial\Omega)_T\cap Q_\rho(x_o,t_o)$. Moreover, $u$ is not
sufficiently regular in time, i.e.  the
extension of $u$ by zero outside of $\Omega_T$ does not necessarily
admit a distributional time derivative on $(\Omega_\eps)_T\cap
Q_\rho(x_o,t_o)$. This is  why we will not choose the zero
extension of $u$
as boundary values for $u_\eps$, but the modified version
$g_\eps := \eta_\eps (x)u$. The cut-off function
$\eta_\eps$ is chosen to vanish on the set
$\{ x\in\Omega: {\rm dist}(x,\partial\Omega)<\eps\}$. With the aid of
Hardy's inequality, one checks that $g_\eps$ admits a time
derivative in the dual space on the domain $(\Omega_\eps)_T\cap
Q_\rho(x_o,t_o)$. 
Note that this choice of the boundary values does not affect the sup-estimate for $Du_{\eps}$.  On the other hand the choice allows us to derive  an appropriate uniform energy estimate for $u_\eps$. Thus, we can pass to a weakly convergent subsequence with the weak limit $\tilde u$.
Since the sup-estimate for $Du_{\eps}$ is uniform in $\eps$, it can be transferred to
 $D\tilde u$. 
 To conclude, it is left to show $\tilde u=u$. This however follows from the uniqueness.

\medskip
\noindent
{\it Acknowledgments.} V.~B\"ogelein  and N.~Liao have been supported by the FWF-Project P31956-N32
``Doubly nonlinear evolution equations".

\section{Preliminaries}\label{Sec:Preliminaries}

\subsection{A remark on convex domains}\label{sec:convex-domains}
The dependence of the constant from
  Theorem~\ref{thm:main} on the domain is given by the 
  quantity
\begin{equation}\label{def-Theta}
  \Theta_{\rho}(x_o):=\frac{\rho^n}{|\Omega\cap B_\rho(x_o)|},
  \qquad\mbox{for $x_o\in\partial\Omega$ and $\rho>0$.}
\end{equation}
Since every bounded convex
set satisfies a uniform cone condition, $\Theta_{\rho}(x_o)$ can be bounded
independently of $x_o\in\partial\Omega$ and $\rho>0$
by a constant only depending on the
domain $\Omega$. For a more detailed discussion, we
refer to \cite[Section 2.1]{BDMS}.



\subsection{A differential geometric identity}

For a $C^2$-domain $\Omega\subset\R^n$, the second fundamental
form of $\partial\Omega$ is defined by 
\begin{equation*}
  \boldsymbol{B}_x(\xi,\eta):=-\partial_\xi\nu(x)\cdot \eta
\end{equation*}
for any $x\in\partial\Omega$ and all tangential vectors $\xi,\eta\in
\mathrm{T}_x(\partial\Omega)$,
where $\nu\in C^1(\partial\Omega,\R^n)$ denotes the outer
unit normal vector field on $\partial\Omega$.

We will use the following differential geometric identity due to Grisvard
\cite[Eqn. 3,1,1,8]{Grisvard:1985}. 

\begin{lemma}\label{lem:Grisvard}
  Let $\Omega\subset\R^n$ be a bounded $C^2$-domain and $w\in
  C^1(\overline\Omega,\R^n)$ a vector field. Then we have the
  following identity on $\partial\Omega$:
  \begin{align*}
    &(w\cdot\nu)\Div w
    -
    \partial_w w\cdot \nu\\
    &\qquad=
    \Div_{\mathrm{T}}((w\cdot\nu)w_{\mathrm{T}})
    -
    (\trace\,\boldsymbol B)(w\cdot\nu)^2
    -
    \boldsymbol B(w_{\mathrm{T}},w_{\mathrm{T}})
    -
    2w_{\mathrm{T}}\cdot\nabla_{\mathrm{T}}(w\cdot\nu),
  \end{align*}
  where $w_{\mathrm{T}}:=w-(w\cdot\nu)\nu$ denotes the tangential component
  of $w$
  and $\nabla_{\mathrm{T}},$ $\Div_{\mathrm{T}}$ the gradient and the
  divergence, respectively, with regard to the
  tangential directions.
\end{lemma}
Note that for a convex domain $\Omega$, our sign convention for the second
fundamental form implies
\begin{equation*}
  \boldsymbol{B}_x(\eta,\eta)\le0
  \mbox{\qquad for any }\eta\in\mathrm{T}_x(\partial\Omega).
\end{equation*}

\subsection{Properties of the coefficients $\mathbf{a}(r)$}

Keeping in mind assumption~\eqref{assumption:a(0)}, we observe
\begin{equation*}
  \mathbf{a}(r)
  =
  \frac1r\int_0^r\frac{\d}{\ds}[s\mathbf{a}(s)]\ds
  =
  \int_0^1\big[\mathbf{a}(r\sigma)+r\sigma\mathbf{a}'(r\sigma)\big]\d\sigma
  \quad\mbox{for any $r>0$.}
\end{equation*}
Therefore, assumption \eqref{assumption:a'} and 
standard estimates, cf. \cite[Lemma~2.1]{GiaquintaModica:1986}, \cite[Lemma~2.1]{AcerbiFusco}, imply
\begin{equation}\label{bounds-a}
	c^{-1}m\big( \mu^2+r^{2}\big)^{\frac{p-2}{2}}
	\leq \mathbf{a}(r)
	\leq
	cM\big( \mu^2+r^{2}\big)^{\frac{p-2}{2}}
\end{equation}
for all $r>0$ and a constant $c=c(p)$. For the derivatives of the coefficients
$\mathbf{a}(|\xi|)\xi_\alpha^i$ we compute
\begin{equation*}
  \partial_{\xi_\beta^j}\big[\mathbf{a}(|\xi|)\xi_\alpha^i\big] 
  =
  \mathbf{a}(|\xi|)\delta_{\alpha\beta}\delta^{ij}
  +\frac{\mathbf{a}^{\prime }( |\xi|) }{|\xi|}\xi_\alpha^i\xi_\beta^j
\end{equation*}
for any $\xi\in\R^{Nn}$ with $\xi\not=0$. This implies the monotonicity and growth property
\begin{equation}\label{monoton-coefficients}
	\boldsymbol h( |\xi| ) |\lambda | ^{2}
	\leq
	\sum_{i,j=1}^{N}\sum_{\alpha,\beta=1}^n \partial_{\xi_\beta^j}\big[\mathbf{a}(|\xi|)\xi_\alpha^i\big] 
        \lambda_{\alpha}^i\lambda _{\beta}^j
	\leq 
	\boldsymbol H( |\xi| ) |\lambda | ^{2},
\end{equation}
for any $\xi,\lambda\in\R^{Nn}$ with $\xi\neq0$, where we abbreviated
\begin{equation}\label{bounds-h-H}
  \left\{
  \begin{aligned}
	\boldsymbol{h}(r)
        &:=\min\{\mathbf{a}(r), \mathbf{a}(r)+r\mathbf{a}'(r)\}\ge c^{-1}m(\mu^2+r^2)^{\frac{p-2}2},\\
	\boldsymbol H(r)&:=\max\{\mathbf{a}(r), \mathbf{a}(r)+r\mathbf{a}'(r)\}
        \le
        cM(\mu^2+r^2)^{\frac{p-2}2}.
      \end{aligned}\right.
  \end{equation}
  The estimates follow from~\eqref{bounds-a}
  and~\eqref{assumption:a'} by distinguishing the cases $\mathbf{a}'(r)\ge 0$ and $\mathbf{a}'(r)<0$.

\subsection{Sobolev's constant on convex domains}

In order to determine the dependencies of the constants in the
Moser iteration, we rely on the following version of
Sobolev's embedding  valid for convex domains, cf.~\cite[Chapter~10, Thm.~8.1]{DiBenedetto:RealAnalysis} and \cite[Lemma~2.3]{BDMS}.

\begin{lemma}\label{lem:sobolev}
Let $K\subset\R^n$ be a bounded open convex set and $1\le p<n$. Then, for any
$w\in W^{1,p}(K)$ we have
\begin{equation*}
  \bigg[\mint_K| w|^{p^\ast}\,\dx\bigg]^\frac{1}{p^\ast}
  \le
  c(n,p)\frac{(\diam K)^n}{|K|}|K|^{\frac1n}\bigg[\mint_K|Dw|^{p}\,\dx\bigg]^\frac{1}{p}
  +
  \bigg[\mint_K| w|^{p}\,\dx\bigg]^\frac{1}{p},
\end{equation*}
with the Sobolev exponent $p^\ast=\frac{np}{n-p}$.
\end{lemma}

\subsection{Auxiliary lemmata} The following elementary assertions will be used in
the Moser iteration. 

\begin{lemma}\label{lem:A}
Let $A>1$, $\theta>1$, $\gamma>0$ and $k\in\N$. Then, we have 
\begin{align}\label{A1}
	\prod_{j=1}^k
	A^{\frac{\theta^{k-j+1}}{\gamma(\theta^k-1)}}
	=
	A^\frac{\theta}{\gamma(\theta-1)}
	\qquad\mbox{for $A,\theta>0$.}
\end{align}	
and 
\begin{align}\label{A2}
	\prod_{j=1}^k
	A^{\frac{j\theta^{k-j+1}}{\gamma(\theta^k-1)}}
	\le
	A^{\frac{\theta^2}{\gamma(\theta-1)^2}} 
	\qquad\mbox{for $A,\theta>1$.}
\end{align}

\end{lemma}

\begin{proof}
For the first product we compute
\begin{align*}
	\prod_{j=1}^k
	A^{\frac{\theta^{k-j+1}}{\gamma(\theta^k-1)}}
	&=
	\exp\bigg[
	\log A\sum_{j=1}^k\frac{\theta^{k-j+1}}{\gamma(\theta^k-1)}\bigg]\\\nonumber
	&=
	\exp\bigg[
	\frac{\log A}{\gamma(1-\theta^{-k})}
	\sum_{j=1}^k\theta^{-j+1}	\bigg]\\\nonumber
	&=
	\exp\bigg[
	\frac{\log A}{\gamma (1-\theta^{-1})} 
	\bigg] 
	=
	A^\frac{\theta}{\gamma(\theta-1)}.
\end{align*}	
Similarly, we re-write the second product in the form
\begin{align*}
  	\prod_{j=1}^k 
  	A^{\frac{j\theta^{k-j+1}}{\gamma(\theta^k-1)}}
  	=
  	\exp\bigg[\log A
	\sum_{j=1}^k \frac{j\theta^{k-j+1}}{\gamma(\theta^k-1)}\bigg]
  	=
  	\exp\bigg[\frac{\log A }{\gamma(1-\theta^{-k})}
	\sum_{j=1}^k j\theta^{-j+1}\bigg].
\end{align*}
To estimate the right-hand side further, we observe that for any $t\in(0,1)$ we have
\begin{align*}
	\frac{1}{1-t^k}\sum_{j=1}^kj t^{j-1}
	=
	\frac{1}{1-t^k}\frac{\d}{\d t}\sum_{j=0}^k t^{j}
        =\frac{1}{1-t^k}
	\frac{\d}{\d t}\frac{1-t^{k+1}}{1-t}
        \le
        \frac{1}{(1-t)^2}.
\end{align*}
We use this estimate with the choice $t=\theta^{-1}\in(0,1)$ and obtain 
\begin{align*}
  	\prod_{j=1}^k
	A^{\frac{j\theta^{k-j+1}}{\gamma(\theta^k-1)}}
  	&\le
  	\exp\bigg[\frac{\log A}{\gamma(1-\theta^{-1})^2}\bigg]
  	=
  	A^{\frac{\theta^2}{\gamma(\theta-1)^2}} .
\end{align*}
This proves the claim.
\end{proof}

\section{A priori estimates for smooth solutions}\label{sec:apriori}
We begin by proving the desired gradient bound in the case of 
regular data. More precisely, we additionally assume that the boundary
$\partial\Omega$ is of class $C^2$ and that the solution is of class
$C^3$. Moreover, we consider parabolic systems that are
non-degenerate, i.e. $\mu>0$, and inhomogeneities with $\spt b\Subset\Omega\times\R$.
The precise statement reads as follows.

\begin{proposition}\label{prop:apriori}
    Let $\Omega\subset\R^n$ be a bounded convex domain with
    $C^2$-boundary, $B_\rho(x_o)$ a ball with
    $\frac{\rho^n}{2^n|\Omega\cap B_{\rho/2}(x_o)|}\le \Theta$ 
    for some constant $\Theta>0$, and $(t_o-\rho^2,t_o)\subset(0,T)$. Moreover, we assume that
    $u\in C^3(\overline\Omega_T\cap
    Q_\rho(z_o),\R^N)$ is a solution to
the parabolic system 
\begin{equation}\label{reg: parabolic system}
  \partial_tu^i
  -
	\sum_{\alpha=1}^{n}
	\big[\mathbf{a}(|Du|) u^i_{x_\alpha}\big]_{x_\alpha}
	=
	b^i 
	\qquad\mbox{in $\Omega_T\cap Q_\rho(z_o)$, }
\end{equation}
for $i=1,\ldots,N$, where $\mathbf a$ and $b$ satisfy assumptions \eqref{assumption:a(0)} -- \eqref{assumption:b} with $\mu\in(0,1]$ and
  \begin{equation*}
    \spt b\Subset\Omega\times\R.
  \end{equation*}
 Moreover, $u\equiv0$ on $(\pl\Om)_T\cap Q_{2\rho}(z_o)$.
  Then we have the gradient sup-estimate
\begin{align}\label{Lipschitz-estimate}\nonumber
    &\sup_{\Omega_T\cap Q_{\rho/2}(z_o)}|Du|\\
    &\qquad\le
        C\bigg[\Big(1+ \rho^{n+2}\| b\|_{
  L^\sigma(\Omega_T\cap Q_\rho(z_o))}^{\frac{(n+2)\sigma}{\sigma-n-2}}\Big) \biint_{\Omega_T \cap Q_{3\rho/4}(z_o)}
	\big( 1+|Du|^p\big)
	\,\dx\dt\bigg]^{\frac{d}p},
\end{align}
for some constant $C$ that depends at most on $n,N,p,\sigma,m,M$,
and $\Theta$ and where $d$ denotes the scaling deficit defined in \eqref{def:deficit}.
\end{proposition} 

The proof is given in the following subsections.

\subsection{Energy estimates for second order derivatives}\label{sec:energy-est}
The first step in the proof of Proposition~\ref{prop:apriori} is the
derivation of an energy  estimate
for smooth solutions to the parabolic system
\eqref{reg: parabolic system}.

\begin{proposition}[Energy estimate for second derivatives] \label{prop:energy}
 Suppose the hypotheses in Proposition~\ref{prop:apriori} hold.
  Then, for  any non-negative increasing
 $C^1$-function $\boldsymbol\Phi \colon \R_{\ge 0} \to  \R_{\ge 0} $,
 any cut-off function
 $\phi\in C_0^\infty (B_\rho(x_o),\R_{\ge0})$  and any non-negative Lipschitz continuous function $\chi\colon [t_o-\rho^2,t_o]\to 
 \R_{\ge 0}$ we have the estimate
 \begin{align}\label{eq:energy-est-sys1}\nonumber
       \iint_{\Omega_T\cap
         Q_\rho}&
       \chi\phi^2\bigg[\partial_t\big[\boldsymbol\Psi(|Du|)\big]
       +
	\tfrac12 
	\boldsymbol \Phi(|Du|)
	 \sum_{\alpha,\beta,\gamma=1}^n \sum_{i,j=1}^N
	\mathbf{b}_{\alpha\beta}^{ij}
	u^i_{x_\alpha x_\gamma }u^j_{x_\beta x_\gamma}
	\bigg]\dx\dt\\\nonumber
	&\le
	2\iint_{\Omega_T\cap Q_\rho} \chi \boldsymbol \Phi(|Du|)
	 \sum_{\alpha,\beta,\gamma=1}^n \sum_{i,j=1}^N
	 \mathbf{b}_{\alpha\beta}^{ij}
         \phi_{x_\alpha}u_{x_\gamma}^i \phi_{x_\beta}
         u_{x_\gamma}^j\dx\dt\\[4pt]
         &\phantom{\le\,}
	\,+\iint_{\Omega_T\cap Q_\rho}
	\chi\phi^2 \boldsymbol \Phi(|Du|)Du\cdot Db\,\dx\dt,
\end{align}
where we abbreviated
\begin{equation}\label{def-chi}
        \boldsymbol\Psi(s):=\int_0^{s}\boldsymbol\Phi(\tau)\tau\d\tau
\end{equation}
and
\begin{equation}\label{def:b}
  \mathbf{b}_{\alpha\beta}^{ij}
  :=
	\mathbf{a}(|Du|)\delta_{\alpha\beta}\delta^{ij}
	+
	\frac{\mathbf{a}'(|Du|)}{|Du|}
        u_{x_\alpha}^i u_{x_{\beta}}^j
      \end{equation}
      for $\alpha,\beta=1,\ldots, n$ and $i,j=1,\ldots, N$. 

\end{proposition}

\begin{remark}
  \upshape
  We note that the monotonicity conditions~\eqref{assumption:a'}
  and~\eqref{bounds-h-H} imply the ellipticity and growth estimates
  \begin{equation}\label{b-elliptic}
    c^{-1}m\big(\mu^2+|Du|^2\big)^{\frac{p-2}2}|\lambda|^2
    \le
    \sum_{\alpha,\beta=1}^n\sum_{i,j=1}^N
    \mathbf{b}_{\alpha\beta}^{ij}\lambda_\alpha^i\lambda_\beta^j
    \le
    cM\big(\mu^2+|Du|^2\big)^{\frac{p-2}2}|\lambda|^2
  \end{equation}  
  for any $\lambda\in\R^{Nn}$, cf.~\eqref{monoton-coefficients}. Therefore, the preceding proposition
  yields an energy estimate of the form
   \begin{align}\label{eq:energy-est-sys2}
   \nonumber
       \iint_{\Omega_T\cap
         Q_\rho}
       \chi\phi^2\bigg[\partial_t\big[\boldsymbol\Psi(|Du|)\big]
       &+
	\tfrac{m}{C}\boldsymbol \Phi(|Du|)\big(\mu^2+|Du|^2\big)^{\frac{p-2}2}|D^2u|^2
	\bigg]\dx\dt\\\nonumber
	&\le
	CM\iint_{\Omega_T\cap Q_\rho}\chi \boldsymbol \Phi(|Du|)
	 \big(\mu^2+|Du|^2\big)^{\frac p2}|D\phi|^2\dx\dt\\[4pt]
         &\phantom{\le\,}
	+\iint_{\Omega_T\cap Q_\rho}\chi\phi^2\boldsymbol \Phi(|Du|)Du\cdot Db\,\dx\dt,
      \end{align}
with a constant $C=C(p)\ge 1$.
This  will be the starting point for the Moser iteration. \hfill $\Box$
\end{remark}

\begin{remark}\upshape
  It is crucial that Proposition~\ref{prop:energy} holds for cylinders
  with arbitrary centers $z_o$, not only for points in the lateral boundary. This allows us to apply it on regularized
  domains $\Omega_\eps\supset\Omega$, with the choice $z_o\in\partial \Omega\times (0,T)$
  independently of $\eps>0$.  \hfill $\Box$
\end{remark}


\begin{proof}[{\rm\bfseries Proof of Proposition~\ref{prop:energy}.}]
  For the sake of convenience,
  we omit the reference to the center $z_o$ in our notation.
  We write $v_e$ for the directional derivative of a function $v$ in
 the direction $e\in\R^n$.
  We start by differentiating \eqref{reg: parabolic system}
in the direction $e$. In view of the identities
\begin{equation}\label{eq:lambda-derivative}
	(|Du|)_e 
	= 
	\sum_{j=1}^N\sum_{\beta =1}^n\frac{u_{x_\beta}^j u_{x_\beta e}^j}{|Du|}
	\quad
	\mbox{and}
	\quad
	\big[ \mathbf{a}(|Du|)\big]_{e}
	=
	\frac{\mathbf{a}'(|Du|)}{|Du|}
	\sum_{j=1}^N\sum_{\beta =1}^n  u_{x_\beta}^ju_{x_\beta e}^j
\end{equation}
we obtain for $i=1,\ldots,N$ that
\begin{align}\label{eq:diff-system}
  (b^i)_e&
  =
  \partial_t u_e^i -
  \sum_{\alpha=1}^n \Big[\mathbf{a}(|Du|) u_{x_{\alpha}}^i\Big] _{x_{\alpha}e}
	=
        \partial_t u_e^i
        -
	\sum_{\alpha,\beta=1}^n\sum_{j=1}^N
	\big[\mathbf{b}_{\alpha\beta}^{ij}(x,t)u_{x_\beta e}^j\big]_{x_\alpha}.
\end{align}
In the last line, we used the abbreviation introduced in \eqref{def:b}.     
Next, we compute 
\begin{align*}
	\mathbf I
	&:=
	\sum_{\alpha,\beta,\gamma=1}^n\, \sum_{i,j=1}^N
	\bigg[
	\frac{\mathbf{a}'(|Du|)}{|Du|}
	u_{x_\alpha}^iu_{x_\gamma}^iu_{x_\beta}^ju_{x_\beta x_\gamma}^j
	\bigg]_{x_\alpha} \\
	&\,=
	\sum_{\alpha,\beta,\gamma=1}^n\, \sum_{i,j=1}^N
	\frac{\mathbf{a}'(|Du|)}{|Du|}
	u_{x_\alpha}^iu_{x_{\alpha}x_\gamma}^iu_{x_\beta}^ju_{x_\beta x_\gamma}^j\\
	&\,\phantom{=\,}
	+
	\sum_{\alpha,\beta,\gamma=1}^n\, \sum_{i,j=1}^N
	u_{x_\gamma}^i
	\bigg[
	\frac{\mathbf{a}'(|Du|)}{|Du|}
	u_{x_\alpha}^iu_{x_\beta}^ju_{x_\beta x_\gamma}^j
	\bigg]_{x_\alpha} \\
	&\,=
	\sum_{\alpha,\beta,\gamma=1}^n\, \sum_{i,j=1}^N
	\Big[\mathbf{b}_{\alpha\beta}^{ij}
	u_{x_{\alpha}x_\gamma}^iu_{x_\beta x_\gamma}^j +
	u_{x_\gamma}^i \big[\mathbf{b}_{\alpha\beta}^{ij}
	u_{x_\beta x_\gamma}^j\big]_{x_\alpha}\Big] \\
	&\,\phantom{=\,} -
	\sum_{\alpha,\gamma=1}^n\, \sum_{i=1}^N
	\Big[
	\mathbf{a}(|Du|) u_{x_{\alpha}x_\gamma}^iu_{x_\alpha x_\gamma}^i +
	u_{x_\gamma}^i 
	\big[
	\mathbf{a}(|Du|) u_{x_\alpha x_\gamma}^i\big]_{x_\alpha} \Big]. 
\end{align*}
 We note that the last term on the right-hand side of the preceding identity is equal to $\Delta [\mathbf g(|Du|)]$, where $\Delta$ stands for the Laplacian, and $\mathbf{g}$ is defined by
$$
	\mathbf g(s):=\int_0^sr\mathbf a(r)\dr\quad \mbox{for any $s>0$.}
$$
Using the differentiated system \eqref{eq:diff-system}
with $e =e_\gamma$ we thus obtain
\begin{align}\label{bfI}
	\mathbf I
	=
	\sum_{\alpha,\beta,\gamma=1}^n\, \sum_{i,j=1}^N
	\mathbf{b}_{\alpha\beta}^{ij}
	u^i_{x_\alpha x_\gamma }u^j_{x_\beta x_\gamma} -
	\Delta \big[\mathbf g(|Du|)\big] +
    \tfrac12\partial_t|Du|^2 -
	Du\cdot Db.
\end{align}
On the other hand, a direct calculation gives
\begin{align}\label{LG}
	\mathbf I + \Delta\big[\mathbf g(|Du|)\big] 
	=
	\mathrm L\big[ \mathbf g(|Du|)\big] ,
\end{align}
where $\mathrm L$ denotes the second order elliptic differential operator  defined by
\begin{equation}\label{def:op-L-system}
	\mathrm L[v] 
	:=
	\sum_{\alpha,\gamma =1}^n
	\big[ \mathbf c_{\alpha\gamma}v_{x_\gamma}\big]_{x_\alpha},
\end{equation}
with coefficients
\begin{align}\label{def:coeff-c-system}
  	\mathbf c_{\alpha\gamma}(x)
  	&:=
 	\frac{1}{\mathbf{a}(|Du|)}
	\bigg[
	\mathbf{a}(|Du|) \delta_{\alpha\gamma}
	+
        \frac{\mathbf{a}'(|Du|)}{|Du|}
        \sum_{i=1}^N
	u_{x_\alpha}^i u_{x_{\gamma}}^i\bigg]\\
	&\,=
	\delta_{\alpha\gamma}
	+
	\frac{\mathbf{a}'(|Du|)}{|Du|\,\mathbf{a}(|Du|)}
        \sum_{i=1}^Nu_{x_\alpha}^i u_{x_{\gamma}}^i\nonumber .
\end{align}
Joining \eqref{bfI} and \eqref{LG}, we find
\begin{align}\label{eq:sub-1-system}
	\tfrac12\partial_t|Du|^2 +
	\sum_{\alpha,\beta,\gamma=1}^n\,\sum_{i,j=1}^N
	\mathbf{b}_{\alpha\beta}^{ij}
	u^i_{x_\alpha x_\gamma }u^j_{x_\beta x_\gamma}
	=
	\mathrm L\big[ \mathbf g(|Du|)\big] +
    Du\cdot Db.
\end{align}
We now multiply this identity by $\phi^2(x)$, where $\phi\in C^\infty_0(B_\rho,\R_{\ge 0})$ is a smooth cut-off function. In the resulting equation we examine the diffusion term on the right-hand side. 
We start by noting that
\begin{align}\label{replace-c-by-D^2f}
  \sum_{\gamma=1}^n\mathbf c_{\alpha\gamma} \mathbf g(|Du|)_{x_\gamma}
  &=
  \sum_{\beta,\gamma=1}^n\sum_{j=1}^N
  \mathbf{a}(|Du|)
  \mathbf c_{\alpha\gamma} u_{x_\beta}^j u_{x_\beta
    x_\gamma}^j\\\nonumber
  &=
  \sum_{\beta,\gamma=1}^n\sum_{i,j=1}^N
  \mathbf{b}_{\alpha\beta}^{ij}
  u_{x_\gamma}^i u_{x_\beta x_\gamma}^j
\end{align}
for $\alpha=1,\ldots,n$, where we used first 
\eqref{eq:lambda-derivative}$_1$ with $e=e_\beta$ and then 
the definition \eqref{def:coeff-c-system}
for the coefficients  $\mathbf c_{\alpha\beta}$ and~\eqref{def:b}. This allows us to compute 
\begin{align*}
	\phi^2 \mathrm L\big[ \mathbf g(|Du|)\big]
  	&=
  	\sum_{\alpha,\gamma =1}^n\Big[ \phi^2
  	\mathbf c_{\alpha\gamma}\mathbf g(|Du|)_{x_\gamma}\Big]_{x_\alpha} -
  	2\phi 
  	\sum_{\alpha,\gamma=1}^n  \phi_{x_\alpha}\mathbf c_{\alpha\gamma}
  	\mathbf g(|Du|)_{x_\gamma} \\
  	&=
  	\mathbf{II} -
	2\phi 
	\sum_{\alpha,\beta,\gamma=1}^n  \sum_{i,j=1}^N
    \mathbf{b}_{\alpha\beta}^{ij}
    \phi_{x_\alpha}u_{x_\gamma}^iu_{x_\beta x_\gamma}^j ,
\end{align*}
with the obvious meaning of $\mathbf{II}$.      
Inserting this identity into \eqref{eq:sub-1-system} multiplied by
$\phi^2$ as described above, we deduce
\begin{align*}
  	\phi^2 & 
  	\bigg[ \tfrac12 \partial_t|Du|^2 +
    \sum_{\alpha,\beta,\gamma=1}^n \sum_{i,j=1}^N
    \mathbf{b}_{\alpha\beta}^{ij} 
  	u^i_{x_\alpha x_\gamma }u^j_{x_\beta x_\gamma}\bigg]\\
  	&=
  	\mathbf{II} -
	2\phi \sum_{\alpha,\beta,\gamma=1}^n \sum_{i,j=1}^N
	\mathbf{b}_{\alpha\beta}^{ij}
    \phi_{x_\alpha}u_{x_\gamma}^iu_{x_\beta x_\gamma}^j +
	\phi^2Du\cdot Db .
\end{align*}
Next, we
note that due to~\eqref{b-elliptic}, the matrix
$(\mathbf{b}_{\alpha\beta}^{ij})$ defines a
positive definite bilinear form on $\R^{Nn}$, which grants a Young type inequality for quadratic forms. That is,
\begin{align*}
  2\phi & \bigg|\sum_{\alpha,\beta=1}^n\,\sum_{i,j=1}^N 
   \mathbf{b}_{\alpha\beta}^{ij}
        \phi_{x_\alpha}u_{x_\gamma}^iu_{x_\beta x_\gamma}^j\bigg|\\
   &\le
   \tfrac12\phi^2\sum_{\alpha,\beta=1}^n\,\sum_{i,j=1}^N  
   \mathbf{b}_{\alpha\beta}^{ij}u_{x_\alpha
     x_\gamma}^iu_{x_\beta x_\gamma}^j
   +2\sum_{\alpha,\beta=1}^n\,\sum_{i,j=1}^N 
   \mathbf{b}_{\alpha\beta}^{ij}\phi_{x_\alpha}u_{x_\gamma}^i\phi_{x_\beta}u_{x_\gamma}^j,
\end{align*}
for any $\gamma=1,\ldots,n$. 
Using this estimate in the identity above and re-absorbing the term containing the second
derivatives of $u$ into the left-hand side yields
\begin{align*}
 	\tfrac12\phi^2 & 
	\bigg[
   	\partial_t|Du|^2 +
	\sum_{\alpha,\beta,\gamma=1}^n \sum_{i,j=1}^N
	\mathbf{b}_{\alpha\beta}^{ij}
	u^i_{x_\alpha x_\gamma }u^j_{x_\beta x_\gamma}\bigg]\\
	&\le
	\mathbf{II} +
	2\sum_{\alpha,\beta,\gamma=1}^n \sum_{i,j=1}^N
	\mathbf{b}_{\alpha\beta}^{ij}
	\phi_{x_\alpha}u_\gamma^i \phi_{x_\beta} u_{x_\gamma}^j +
	\phi^2Du\cdot Db.
\end{align*}
Next, we multiply this identity by $\chi(t)\boldsymbol\Phi(|Du|)$, where $\chi\colon [t_o-\rho^2,t_o]\to\R_{\ge 0}$ is a non-negative Lipschitz continuous function and $\boldsymbol\Phi\in C^1(\R_{\ge 0},\R_{\ge 0})$ is
increasing. 
For the term involving the time derivative we compute 
\begin{align*}
   	\tfrac12\boldsymbol\Phi (|Du|)\partial_t|Du|^2
   	=
   	\boldsymbol\Phi (|Du|)|Du|\partial_t|Du|
   	=
   	\partial_t\bigg[\int_0^{|Du|}\boldsymbol\Phi(\tau)\tau\,d\tau\bigg]
   	=
   	\partial_t\big[\boldsymbol\Psi(|Du|)\big],   
\end{align*}
with the function $\boldsymbol\Psi$ defined in~\eqref{def-chi} and obtain 
\begin{align}\label{eq:sub-2-system}\nonumber
 	\chi\phi^2 & 
	\bigg[
   	\partial_t\big[\boldsymbol\Psi(|Du|)\big] +
	\tfrac12\boldsymbol\Phi (|Du|)\sum_{\alpha,\beta,\gamma=1}^n \sum_{i,j=1}^N
	\mathbf{b}_{\alpha\beta}^{ij}
	u^i_{x_\alpha x_\gamma }u^j_{x_\beta x_\gamma}\bigg]\\
	&\le
	\chi \boldsymbol\Phi(|Du|) \bigg[
	\mathbf{II} +
	2\sum_{\alpha,\beta,\gamma=1}^n \sum_{i,j=1}^N
	\mathbf{b}_{\alpha\beta}^{ij}
	\phi_{x_\alpha}u_\gamma^i \phi_{x_\beta} u_{x_\gamma}^j +
	\phi^2Du\cdot Db \bigg].
\end{align}
Next, we analyze the term containing $\mathbf{II}$, which will result in a boundary term. Indeed,
\begin{align*}
	\boldsymbol\Phi (|Du|)\mathbf{II}
	&=
	\boldsymbol\Phi (|Du|)\sum_{\alpha,\gamma=1}^n\Big[ \phi^2 
	\mathbf c_{\alpha\gamma} \mathbf g(|Du|)_{x_\gamma}\Big]_{x_\alpha}\\
	&=
	\sum_{\alpha,\gamma=1}^n\Big[ \phi^2 \boldsymbol\Phi (|Du|)
	\mathbf c_{\alpha\gamma}\mathbf g(|Du|)_{x_\gamma}\Big]_{x_\alpha} -
	\phi^2 \sum_{\alpha,\gamma=1}^n \mathbf c_{\alpha\gamma}\boldsymbol\Phi (|Du|)_{x_\alpha}
	\mathbf g(|Du|)_{x_\gamma}\\
	&\le
	\sum_{\alpha,\gamma=1}^n\Big[ \phi^2 \boldsymbol\Phi (|Du|)
	\mathbf c_{\alpha\gamma}\mathbf g(|Du|)_{x_\gamma}\Big]_{x_\alpha}.
\end{align*}
In the last line we used 
\begin{align*}
 	\sum_{\alpha,\gamma=1}^n \mathbf c_{\alpha\gamma}
	\boldsymbol\Phi (|Du|)_{x_\alpha}
	\mathbf g(|Du|)_{x_\gamma}
	=
	\boldsymbol\Phi^\prime (|Du|) \mathbf g^\prime(|Du|)
	\sum_{\alpha,\gamma=1}^n
	\mathbf c_{\alpha\gamma}|Du|_{x_\alpha}|Du|_{x_\gamma}
	\ge
	0,
\end{align*}
since $\boldsymbol\Phi$ and $\mathbf g$ are both increasing and that the coefficients
$\mathbf c_{\alpha\gamma}$ are positive definite. 

We now integrate  $\chi \boldsymbol\Phi (|Du|)\mathbf{II}$ over $\Omega_T\cap Q_\rho$ and perform an integration
by parts. This leads to a boundary integral. More
 precisely, denoting by $\nu$ the outward  unit normal vector on $\partial\Omega$, we have
 \begin{align}\label{integral-of-II}
 	&\iint_{\Omega_T\cap Q_\rho}\chi\boldsymbol\Phi (|Du|)\mathbf{II}\,\dx\dt\\
	\nonumber
	&\qquad\le
	\iint_{(\partial\Omega)_T\cap Q_\rho}\chi\phi^2\boldsymbol\Phi (|Du|)
	\sum_{\alpha,\gamma=1}^n \mathbf c_{\alpha\gamma}\mathbf g(|Du|)_{x_\gamma}\nu_\alpha\,\dHn\dt\\\nonumber
	&\qquad=
	\iint_{(\partial\Omega)_T\cap Q_\rho}
	\chi\phi^2\boldsymbol\Phi (|Du|)
	\sum_{\alpha,\beta,\gamma=1}^n \sum_{i,j=1}^N\mathbf{b}_{\alpha\beta}^{ij}
	u_{x_\gamma}^iu_{x_\beta x_\gamma}^j\nu_\alpha\,\dHn\dt.
 \end{align}
 The last step follows from~\eqref{replace-c-by-D^2f}.
 Now, we analyze the integrand on the right-hand side by recalling the explicit form
of the coefficients. In view of~\eqref{def:b}, we obtain
\begin{align*}
  \sum_{\alpha,\beta,\gamma=1}^n & \sum_{i,j=1}^N
  \mathbf{b}_{\alpha\beta}^{ij}
	u_{x_\gamma}^iu_{x_\beta x_\gamma}^j\nu_\alpha\\
	&=
	\sum_{\alpha,\beta,\gamma=1}^n \sum_{i,j=1}^N
	\bigg[
	\mathbf{a}(|Du|)\delta_{\alpha\beta}\delta^{ij}
        +\frac{\mathbf{a}'(|Du|)}{|Du|}
	u^i_{x_\alpha}u^j_{x_\beta}\bigg]u_{x_\gamma}^iu_{x_\beta x_\gamma}^j\nu_\alpha\\
	&=
	\sum_{\alpha,\gamma=1}^n\sum_{i=1}^N
	\mathbf{a}(|Du|)u_{x_\gamma}^iu_{x_\alpha x_\gamma}^i\nu_\alpha
	+
	\sum_{i=1}^N
	u_\nu^i
	\sum_{\beta,\gamma=1}^n\sum_{j=1}^N
	\frac{\mathbf{a}'(|Du|)}{|Du|}
	u^i_{x_\gamma}u^j_{x_\beta}u^j_{x_\beta x_\gamma}.
\end{align*}
At this point,
we apply the differential geometric identity in
Lemma~\ref{lem:Grisvard} to the vector fields $w:=\nabla u^i$, $i=1,\ldots,N$.
In our case, the tangential components
of $w$ 
vanish since $u\equiv0$ on $(\partial\Omega)_T\cap Q_\rho$.
Hence, Lemma~\ref{lem:Grisvard} yields the identity 
\begin{equation*}
	\Delta u^i u_\nu^i- 
	\sum_{\alpha,\gamma=1}^n
	u_{x_\gamma}^iu_{x_\alpha x_\gamma}^i\nu_\alpha
	=-(\trace\,\boldsymbol B) \big(u_\nu^i\big)^2\quad \mbox{on $(\partial\Omega)_T\cap Q_\rho $}
\end{equation*}
for $i=1, \dots ,N$.
Here, $\boldsymbol B$ denotes the second fundamental form of $\partial\Omega$. Note that
$\mathrm{trace}\,\boldsymbol B\le 0$, since $\Omega$ is convex. Therefore, in the above identity
the right-hand side is non-negative. This allows us to continue in estimating the boundary 
 integral above.
 More precisely, on $(\partial\Omega)_T\cap Q_\rho $ we obtain 
 \begin{align*}
   &\sum_{\alpha,\beta,\gamma=1}^n \sum_{i,j=1}^N
   \mathbf{b}_{\alpha\beta}^{ij}
	u_{x_\gamma}^iu_{x_\beta x_\gamma}^j\nu_\alpha\\
	&\quad\quad\le
	\sum_{i=1}^N
	u_\nu^i
	\sum_{\beta,\gamma=1}^n\sum_{j=1}^N\bigg[
	\mathbf{a}(|Du|)\delta_{\gamma\beta}\delta^{ij}+
	\frac{\mathbf{a}'(|Du|)}{|Du|}
	u^i_{x_\gamma}u^j_{x_\beta}\bigg]u^j_{x_\beta x_\gamma}\\
	&\quad\quad=
	\sum_{i=1}^N
	u_\nu^i
	\sum_{\beta,\gamma=1}^n\sum_{j=1}^N
	\mathbf{b}_{\gamma\beta}^{ij}u^j_{x_\beta x_\gamma}
	=
        \sum_{i=1}^N
	u_\nu^i
	\sum_{\gamma=1}^n
	\big[\mathbf{a}(|Du|) u_{x_\gamma}^i\big]_{x_\gamma} \\
        &\quad\quad=
        u_\nu\cdot\partial_tu-u_\nu\cdot b=0.
      \end{align*}
In the last line, we used the parabolic system      
\eqref{reg: parabolic system},
the fact $\partial_tu\equiv0$ on
$(\partial\Omega)_T\cap Q_\rho$ and $\spt b\Subset\Omega_T$.
Recalling~\eqref{integral-of-II}, we deduce that
\begin{align*}
 	\iint_{\Omega_T\cap Q_\rho}
	\chi\boldsymbol\Phi (|Du|)\mathbf{II}\,\dx\dt
	&\le 0.
\end{align*}
Therefore, integrating \eqref{eq:sub-2-system} over $\Omega_T\cap Q_\rho$ we obtain
\begin{align*}
 	& \iint_{\Omega_T\cap Q_\rho} \chi\phi^2 
	\bigg[
   	\partial_t\big[\boldsymbol\Psi(|Du|)\big] +
	\tfrac12\boldsymbol\Phi (|Du|)\sum_{\alpha,\beta,\gamma=1}^n \sum_{i,j=1}^N
	\mathbf{b}_{\alpha\beta}^{ij}
	u^i_{x_\alpha x_\gamma }u^j_{x_\beta x_\gamma}\bigg] \dx\dt\\
	&\qquad\le
	\iint_{\Omega_T\cap Q_\rho}
	\chi \boldsymbol\Phi(|Du|) \bigg[
	2\sum_{\alpha,\beta,\gamma=1}^n \sum_{i,j=1}^N
	\mathbf{b}_{\alpha\beta}^{ij}
	\phi_{x_\alpha}u_\gamma^i \phi_{x_\beta} u_{x_\gamma}^j +
	\phi^2Du\cdot Db \bigg] \dx\dt.
\end{align*}
This finishes the proof of the proposition.
\end{proof}

\subsection{A reverse H\"older type inequality}\label{sec:iteration}
Here, we work in the setting of Proposition~\ref{prop:apriori}. Again we omit the reference to the center $z_o$
in our notation.  By $\zeta\in
C^1(\R_{\ge0},[0,1])$ we denote a cut-off function with respect to the
time variable that satisfies $\zeta\equiv0$ on $[0,\frac12]$, $\zeta\equiv 1$
on $[1,\infty)$ and $0\le\zeta'\le 3$ on $[\frac12,1]$.
Moreover, we consider a cut-off function $\phi\in
C^\infty_0(B_\rho,[0,1])$ with respect to the spatial variables. 
In the energy estimate \eqref{eq:energy-est-sys2} we choose
the non-negative increasing function $\boldsymbol\Phi\colon\R_{\ge0}\to\R_{\ge0}$ in the form
\begin{equation*}
  \boldsymbol\Phi(s):=  \widetilde{\boldsymbol\Phi} \big(\sqrt{\mu^2+s^2}\big)
  \qquad\mbox{with}\qquad
  \widetilde{\boldsymbol\Phi}(\tau):=\zeta^2(\tau)\tau^{2\alpha}\quad\text{ for some }\al\ge0.
\end{equation*}
We could omit the cut-off function $\zeta$ in the case $\mu=1$. For the sake of a unified approach we proceed using $\zeta$ in any case. 
In the sequel we use the abbreviation
$$
  \boldsymbol{\mathcal H}(\xi):= \sqrt{\mu^2+|\xi|^2}
  \qquad\mbox{for $\xi\in\R^{Nn}$,}
$$
so that $\boldsymbol\Phi(|\xi|)=\widetilde{\boldsymbol\Phi}(\boldsymbol{\mathcal H}(\xi))$. 
With this notation, we have 
\begin{align}\label{Psi-upper-bound}
  \boldsymbol \Psi (|Du|)
	&:=
	\int_0^{|Du|} \boldsymbol\Phi(s)s \,\d s
    =
    \int_{\mu}^{\bH(Du)}\,\zeta^2(\tau)\tau^{2\alpha+1}\d\tau
	\le
	\tfrac{1}{2+2\alpha}\,\bH(Du)^{2+2\alpha}.
\end{align}
Since $\bH(0)\equiv \mu\le 1$ we deduce the lower bound
\begin{equation}
  \label{Psi-lower-bound}
  \boldsymbol \Psi (|Du|)
  \ge
  \tfrac{1}{2+2\alpha}\,\bH(Du)^{2+2\alpha} -
  \tfrac{1}{2+2\alpha}.
\end{equation}
Now, we start with estimating the second integral 
on the right-hand side of \eqref{eq:energy-est-sys2}. Since $\spt b (\cdot,t)\Subset\Omega$ for any $t\in [0,T]$ and 
since $\phi\in C^\infty_0 (B_\rho)$,  we are allowed to integrate by parts (with respect to $x$) in the integral containing $b$ and obtain
\begin{align*}
	&\sum_{i=1}^N\sum_{\alpha=1}^n\iint_{\Omega_T \cap Q_{\rho } }
	\chi\phi^2 \boldsymbol\Phi (|Du|)u^i_{x_\alpha} b^i_{x_\alpha}\,\dx\dt\\
	&=
	-\sum_{i=1}^N \sum_{\alpha=1}^n\iint_{\Omega_T \cap Q_{\rho } }\chi b^i
	\Big[\phi^2 \boldsymbol\Phi (|Du|) u^i_{x_\alpha}\Big]_{x_\alpha}\,\dx\dt\\
	&=
	-\sum_{i,j=1}^N\sum_{\alpha,\beta=1}^n\iint_{\Omega_T \cap Q_{\rho } }
	\chi\phi^2 b^i
	\bigg[\boldsymbol\Phi (|Du|) \delta_{\alpha\beta}\delta^{ij} +\boldsymbol\Phi' (|Du|)
	\frac{u^i_{x_\alpha}u^j_{x_\beta}}{|Du|}
	\bigg]u^j_{x_\alpha x_\beta}\,\dx\dt\\
	&\quad
	-2\sum_{i=1}^N\sum_{\alpha=1}^n\iint_{\Omega_T \cap Q_{\rho }}
	\chi\phi\,  b^i
	\boldsymbol\Phi (|Du|)\phi_{x_\alpha} u^i_{x_\alpha}\,\dx\dt\\
	&=:\,\mathbf I +\,\mathbf{II},
\end{align*}
with the obvious meaning of $\mathbf I$ and $\mathbf{II}$.  The first integral can be estimated as
\begin{align*}
	\mathbf I
	&\le
	c(n,N)\iint_{\Omega_T\cap Q_\rho}\chi\phi^2 |b|
    \big[\boldsymbol\Phi(|Du|)+|Du|\boldsymbol\Phi'(|Du|)\big]|D^2u|\,
    \dx\dt.
\end{align*}
To estimate the term in brackets we note that $|Du|\le\bH(Du)$, $\zeta\le 1$ and $|Du|\le\bH(Du)\le1$ whenever $\zeta'(\bH(Du))\neq0$ and finally $\bH(Du)\ge\tfrac12$ whenever $\zeta(\bH(Du))\neq0$. Therefore, we obtain 
\begin{align*}
    &\boldsymbol\Phi(|Du|)+|Du|\boldsymbol\Phi'(|Du|) \\
    &\qquad\le
    \zeta(\bH(Du))\bH(Du)^{2\alpha} 
    \big[(1+2\alpha)\zeta(\bH(Du))  +
    2 \zeta'(\bH(Du))|Du|\big] \\
    &\qquad\le
    c\,(1+2\alpha)\zeta(\bH(Du))\bH(Du)^{2\alpha} \\
    &\qquad\le
    c\,(1+2\alpha)\zeta(\bH(Du))\bH(Du)^{2\alpha+\frac{p-2+(2-p)_+}2}. 
\end{align*}
Inserting this above yields
\begin{align*}
  	\mathbf{I}
  	&\le
    c\,(1+2\alpha) \iint_{\Omega_T \cap Q_{\rho }}
    \chi\phi^2
	|b|\zeta(\bH(Du))\bH(Du)^{2\alpha+\frac{p-2+(2-p)_+}2}
	|D^2u|\,\dx\dt\\\nonumber
	&\le
	\frac{m}{2C} \iint_{\Omega_T \cap Q_{\rho } } \chi\phi^2
	\zeta^2(\bH(Du))\bH(Du)^{p+2\alpha -2}|D^2u|^2\dx\dt\\\nonumber
	&\phantom{\le\,}
	+\frac{c\,(1+2\alpha)^2}{m}\iint_{\Omega_T \cap Q_{\rho } }
	\chi\phi^2
	\bH(Du)^{2\alpha+(2-p)_+}|b|^2\dx\dt\\
    &=:
    \frac{m}{2C}\mathbf{I}_1+\frac{c\,(1+2\alpha)^2}{m}\mathbf{I}_2,
\end{align*}
for a constant $c=c(n,N,p)$.
In the last estimate we used Young's inequality.
The constant $C$ is from the energy estimate \eqref{eq:energy-est-sys2} and depends only on $p$.
The second term is bounded by
\begin{align*}
   \mathbf{II}
   &\le
   2\iint_{\Omega_T \cap Q_{\rho } }
   \chi\phi\,  |b|
   \boldsymbol\Phi (|Du|)|D\phi| |Du|\,\dx\dt\\\nonumber
   &\le 
   2\iint_{\Omega_T \cap Q_{\rho } }
   \chi\phi\, |b|\zeta^2(\bH(Du))\bH (Du)^{2\alpha+1} 
   |D\phi| \,\dx\dt\\\nonumber
   &\le 
   c\iint_{\Omega_T \cap Q_{\rho } }\chi\phi\, |b|\bH(Du)^{2\alpha+\frac{p+(2-p)_+}2} 
   |D\phi| \,\dx\dt\\
   &\le
     c\,M\iint_{\Omega_T \cap Q_{\rho } }\chi \bH (Du)^{p+2\alpha}|D\phi|^2\,\dx\dt
     +\tfrac{1}{M}\mathbf{I}_2,
\end{align*}
where, in the second-to-last step, we again used the fact $\bH(Du)\ge\frac12$ on the support of $\zeta(\bH(Du))$, and in the last step
we applied Young's inequality. 
 Using the above estimates for $\mathbf{I}$ and $\mathbf{II}$  in~\eqref{eq:energy-est-sys2} and re-absorbing the term $\frac{m}{2C}\mathbf{I}_1$
 into the left-hand side, we arrive at
\begin{align}\label{pre-W22-est}
	\nonumber
   	\iint_{\Omega_T\cap Q_\rho} &
   	\chi\phi^2 
   	\Big[\partial_t\big[\boldsymbol\Psi(|Du|)\big] +
	\tfrac{m}{2C}
	\zeta^2(\bH(Du))\bH (Du)^{p+2\alpha-2} |D^2u|^2\Big]
	\,\dx\dt\\
	&\qquad\qquad\le
	c\iint_{\Omega_T\cap Q_\rho}\chi \bH (Du)^{p+2\alpha} |D\phi|^2\,\dx\dt +
	c(1+2\alpha)^2\mathbf{I}_2,
\end{align}
where $c=c(m,M,n,N,p)$.  
In \eqref{pre-W22-est} we choose $\chi$ in the form of a product of two functions $\chi$ and $\widetilde \chi$. We choose
the first function $\chi\in W^{1,\infty}\big(  [t_o-\rho^2, t_o]\big) $ to satisfy  $0\le\chi\le1$, $\chi(t_o-\rho^2)=0$, and 
$\partial_t\chi\ge 0$, while the second one  is defined by
\begin{equation*}
	\widetilde\chi (t)
	:=
	\left\{
	\begin{array}{cl}
		1,& t\in [t_o-\rho^2, \tau],\\[4pt]
		1-\frac{t-\tau}{\delta},& t\in (\tau,\tau+\delta),\\[4pt]
		0,& t\in [\tau+\delta ,t_o],
	\end{array}
	\right.
\end{equation*}
where $\delta>0$ and $t_o-\rho^2<\tau<\tau +\delta<t_o$.
With this specification of $\chi$
we consider the first integral on the left-hand side. We perform an integration by parts with respect to time and obtain (observe that no boundary terms occur due to the choice of $\chi$ and $\widetilde\chi$)
\begin{align*}
	&\iint_{\Omega_T\cap Q_\rho}
	\phi^2\chi\widetilde\chi
	\partial_t\big[\boldsymbol\Psi(|Du|)\big]\,\dx\dt\\
	&\qquad=
	-
	\iint_{\Omega_T\cap Q_\rho}
	\phi^2\partial_t\big[\chi(t)\widetilde \chi (t)\big]
	\boldsymbol\Psi(|Du|) \,\dx\dt\\
	&\qquad=
	-
	\iint_{\Omega_T\cap Q_\rho}
	\phi^2 \big[ \widetilde \chi\partial_t\chi+
	\chi\partial_t\widetilde \chi\big]
	\boldsymbol\Psi(|Du|)\,\dx\dt\\
	&\qquad=
	-\iint_{\Omega_T\cap Q_\rho}
	\widetilde\chi\phi^2\partial_t\chi
	\boldsymbol\Psi(|Du|) \,\dx\dt +
	\frac1{\delta}
	\iint_{\Omega_T\cap B_\rho\times (\tau,\tau+\delta)}
	\chi\phi^2 \boldsymbol\Psi(|Du|)\,\dx\dt.
\end{align*}
We insert this into \eqref{pre-W22-est} and pass to the limit
$\delta\downarrow 0$. For $\tau\in (t_o-\rho^2,t_o)$ we obtain
\begin{align*}
   	\nonumber
   	\int_{\Omega\cap B_\rho\times\{ \tau\}}&
    \chi\phi^2\boldsymbol\Psi(|Du|)\,\dx\\
    &\phantom{\le\,}+
	\iint_{\Omega_T\cap B_\rho\times(t_o-\rho^2,\tau)} \chi\phi^2
	\zeta^2(\bH(Du))\bH (Du)^{p+2\alpha-2}|D^2u|^2
	\,\dx\dt\nonumber\\
	&\le
	c\bigg[ \iint_{\Omega_T\cap Q_\rho}
	 \chi\bH (Du)^{p+2\alpha}|D\phi|^2\,\dx\dt+
	    (1+2\alpha)^2\mathbf{I}_2+\mathbf I_3\bigg],
\end{align*}
where 
$\mathbf I_3$ is defined by
\begin{equation*}
	\mathbf I_3
	:= 
	\frac{1}{2(1+\alpha)}
	\iint_{\Omega_T\cap Q_\rho}\phi^2
	\partial_t\chi
	 \bH (Du)^{2+2\alpha}\,\dx\dt.
\end{equation*}
In the estimate leading to  $\mathbf I_3$ we used \eqref{Psi-upper-bound} and the fact that
$\partial_t\chi\ge 0$.  Note that $\mathbf I_3$ is non-negative. Observe  also that the right-hand side of the preceding inequality
is independent of $\tau$.
Therefore we can pass  to the limit $\tau\uparrow t_o$ in the second integral on the left-hand side, while in the first integral we can take the supremum over $\tau\in (t_o-\rho^2,t_o)$. This implies
\begin{align*}
	\sup_{t_o-\rho^2<\tau<t_o}&
	\int_{\Omega\cap B_\rho\times\{ \tau\}}
    \chi\phi^2\boldsymbol\Psi(|Du|)\,\dx\nonumber\\
    &\phantom{\le\,} +
	\iint_{\Omega_T\cap Q_\rho} \chi\phi^2
	\zeta^2(\bH(Du))\bH (Du)^{p+2\alpha-2} |D^2u|^2
	\,\dx\dt\nonumber\\
	&\le
	c\bigg[ \iint_{\Omega_T\cap Q_\rho}\chi
	\bH (Du)^{p+2\alpha}|D\phi|^2\,\dx\dt+
	(1+2\alpha)^2\mathbf{I}_2+\mathbf I_3\bigg],
\end{align*}
with a constant $c=c(n,N,m,M,p)$.
In order to bound the sup-term from below, we 
use \eqref{Psi-lower-bound} and multiply the resulting inequality by $(p+2\alpha)$,  from which we deduce
\begin{align*}
       \sup_{t_o-\rho^2<\tau<t_o}&
       \int_{\Omega\cap B_\rho\times\{ \tau\}}
       \chi\phi^2 \bH (Du)^{2+2\alpha} \,\dx\\
       &\phantom{\le\,}+
       (p+2\alpha)\iint_{\Omega_T\cap Q_\rho} \chi\phi^2
       \zeta^2(\bH(Du))\bH (Du)^{p+2\alpha-2} |D^2u|^2
	\,\dx\dt\\
	&\le
	c\,(p{+}2\alpha)
	 \iint_{\Omega_T\cap Q_\rho}\!\! \Big[ \|D\phi\|_{L^\infty}^2\bH (Du)^{p+2\alpha}
	 +\|\partial_t\chi\|_{L^\infty}\bH (Du)^{2+2\alpha}\Big]\dx\dt\\
        &\phantom{\le\,} +
	c\,(p+2\alpha)^3\iint_{\Omega_T \cap Q_{\rho } }\chi\phi^2 \bH
        (Du)^{2\alpha+(2-p)_+}|b|^2\,\dx\dt +
    |\Omega\cap B_\rho|,
\end{align*}
for a constant $c$ depending on $n,N,m,M,$ and $p$.
 After taking means, the preceding estimate takes
the form 
\begin{align}\label{pre-W22-est-2}\nonumber
       \sup_{t_o-\rho^2<\tau<t_o}&
       \mint_{\Omega\cap B_\rho\times\{\tau\}}
       \chi\phi^2 \bH (Du)^{2+2\alpha} \,\dx\\\nonumber
       &\phantom{\le\,}+
	(p+2\alpha)\rho^2\biint_{\Omega_T\cap Q_\rho} \chi\phi^2
        \zeta^2(\bH(Du))\bH (Du)^{p+2\alpha-2} |D^2u|^2
	\,\dx\dt\\
	&\le
	c\,(p+2\alpha)\mathbf{R}_1
	+
	c\,(p+2\alpha)^3 \mathbf{R}_2 + 1
	=:
	\mathbf{R},
\end{align}
for a constant $c=c(n,N,m,M,p)\ge 1$ and with the abbreviations
\begin{align*}
 	\mathbf{R}_1
	&:=
	 \rho^2\biint_{\Omega_T\cap Q_\rho}\Big[
         \|D\phi\|_{L^\infty}^2\bH
         (Du)^{p+2\alpha}+\|\partial_t\chi\|_{L^\infty} \bH
         (Du)^{2+2\alpha}\Big]\dx\dt
\end{align*}
and
\begin{align*}
 	\mathbf{R}_2
	&:=
	\rho^2\biint_{\Omega_T \cap Q_{\rho }}
	\chi\phi^2\bH (Du)^{2\alpha+(2-p)_+}|b|^2 \,\dx\dt.
\end{align*}  
Observe that we kept the cut-off function $\chi\phi^2$ in the integrand
of the last integral. The reason for that will
become clear later.

The next step is to perform an interpolation argument of
Gagliardo-Nirenberg type. 
For the parameter $\delta:=\frac{2(1+\alpha)}{n}>0$, we compute
\begin{align}\label{est:D^2u}\nonumber
	\Big|D\Big[\phi^{1+\frac2n}\, &\zeta^2(\bH(Du))\bH(Du)^{\frac{p+ 2\alpha+2\delta }{2}}\Big]\Big|\\\nonumber
	&\le 
	\tfrac{p+2\alpha+2\delta}{2}
	\phi^{1+\frac2n}\zeta^2(\bH(Du))\bH (Du)^\frac{p+ 2\alpha-2+2\delta}{2}
	| D^2u |\\\nonumber
	&\phantom{\le\,}
        +2\phi^{1+\frac2n}\zeta(\bH(Du))\zeta'(\bH(Du))\bH(Du)^{\frac{p+ 2\alpha+2\delta }{2}}|D^2u|
        \\\nonumber
        &\phantom{\le\,}
        +\tfrac{n+2}{n}\phi^{\frac2n}|D\phi|\zeta^2(\bH(Du))\bH(Du)^{\frac{p+ 2\alpha+2\delta }{2}}\\\nonumber
      &\le 
	c\,(p+2\alpha)\phi^{1+\frac2n}
	\zeta(\bH(Du))\bH (Du)^\frac{p+ 2\alpha-2+2\delta}{2}
	| D^2u |\\
        &\phantom{\le\,}
        +c\,\|D\phi\|_{L^\infty}\phi^{\frac2n}\bH(Du)^{\frac{p+ 2\alpha+2\delta }{2}},
\end{align}
for a constant $c=c(n)$.
In the last step, we used the fact that $\bH(Du)\le 1$
on the support of $\zeta'(\bH(Du))$, as well as the bounds $\zeta\le1$
and $\zeta'\le3$.
 On a fixed time slice we apply Sobolev's embedding
 in Lemma~\ref{lem:sobolev} with $p=\frac{2n}{n+2}$ 
 and then inequality~\eqref{est:D^2u}, with the result
\begin{align*}
	&\mint_{\Omega\cap B_\rho}\phi^{2+\frac{4}{n}}\zeta^4(\bH(Du))\bH(Du)^{p+2\alpha+2\delta}\,\dx\\
	&\quad
	\le
	2C_{\mathrm{Sob}}^2\rho^2
        \bigg[
        \mint_{\Omega\cap B_\rho}\big|D\big[\phi^{1+\frac{2}{n}}\zeta^2(\bH(Du))
        \bH(Du)^{\frac{p+2\alpha+2\delta}{2}}\big]\big|^{\frac{2n}{n+2}}\,\dx
	\bigg]^\frac{n+2}{n}\\
        &\quad\phantom{\le\,}
        +2
        \bigg[
        \mint_{\Omega\cap B_\rho}\big|\phi^{1+\frac{2}{n}}\zeta^2(\bH(Du))\bH(Du)^{\frac{p+2\alpha+2\delta}{2}}\big|^{\frac{2n}{n+2}}\,\dx
	\bigg]^\frac{n+2}{n}
        \\
        &\quad\le
        c\,C_{\mathrm{Sob}}^2\rho^2(p+2\alpha)^2
        \bigg[
        \mint_{\Omega\cap B_\rho}\Big[\phi^{1+\frac2n}\zeta(\bH(Du))\bH (Du)^{\frac{p+ 2\alpha-2}{2}+\delta}
	| D^2u |\Big]^{\frac{2n}{n+2}}\,\dx
	\bigg]^\frac{n+2}{n}\\
        &\quad
        \phantom{\le\,}
        +c\,C_{\mathrm{Sob}}^2 \rho^2\|D\phi\|_{L^\infty}^2
        \bigg[
        \mint_{\Omega\cap B_\rho}\big|\phi^{\frac2n}\bH(Du)^{\frac{p+2\alpha}{2}+\delta}\big|^{\frac{2n}{n+2}}\,\dx
	\bigg]^\frac{n+2}{n}.
\end{align*}
In the last line, we used the estimate
$\|\phi\|_{L^\infty}\le \rho\|D\phi\|_{L^\infty}$, which holds true
since $\phi$ has compact support in $B_\rho$, and we assumed that
$C_{\mathrm{Sob}}\ge 1$.
According to Lemma~\ref{lem:sobolev}, we can choose the Sobolev constant
$C_{\mathrm{Sob}}$ only depending on $n$ and $\Theta$. 
 Next, we estimate both  integrals on the right-hand side by
 means of H\"older's inequality with exponents $\frac{n+2}{n}$ and
 $\frac{n+2}2$, which leads to 
\begin{align*}
  \mint_{\Omega\cap
    B_\rho}&\phi^{2+\frac{4}{n}}\zeta^4(\bH(Du))\bH(Du)^{p+2\alpha+2\delta}\,\dx
  \le
  c\,C_{\rm Sob}^2\rho^2\, 
  \mathbf{II}_1\cdot\mathbf{II}_2^\frac{2}{n},
\end{align*}
where
\begin{align*}
	\mathbf{II}_1
	&:=
	(p+2\alpha)^2
	\mint_{\Omega\cap B_\rho}\phi^2\zeta^2(\bH(Du))\bH (Du)^{p+2\alpha-2}|D^2u|^2\,\dx\\
	&\phantom{:=\,}
	+
	\|D\phi\|_{L^\infty}^2
        \mint_{\Omega\cap B_\rho}\bH (Du)^{p+2\alpha}\,\dx
\end{align*}
and
\begin{equation*}
	\mathbf{II}_2:=\mint_{\Omega\cap B_\rho}\phi^2\bH (Du)^{n\delta}\dx.
\end{equation*}
At this point, the reason for our choice of $\delta$ becomes clear.
In fact, we have chosen $\delta$ in such a way that
$n\delta=2+2\alpha$ coincides with the integrability exponent in the sup-term of the energy inequality \eqref{pre-W22-est-2}. We multiply the preceding inequality by
$\chi (t)^{1+\frac{2}{n}}$ and take the mean with respect to $t$ over the interval $(t_o-\rho^2,t_o)$. In this way we obtain
\begin{align}\label{hi-int-1}
	\biint_{\Omega_T\cap Q_\rho}&
	[\chi\phi^2]^{1+\frac{2}{n}}
    \zeta^4(\bH(Du))\bH(Du)^{p+2\alpha+2\delta} \,\dx\dt
	\nonumber\\
	&\le
	c\,C_{\rm Sob}^2\rho^2\,
	\mint_{(t_o-\rho^2,t_o)} \chi(t)\mathbf{II}_1\cdot \bigg[\mint_{\Omega\cap B_\rho}\chi\phi^2
	\bH (Du)^{2+2\alpha}\dx\bigg]^\frac{2}{n}\dt 
	\nonumber\\
	&\le 
	c\,C_{\rm Sob}^2\rho^2\,
	\mathbf R^\frac{2}{n} \mint_{(t_o-\rho^2,t_o)} \chi(t)\mathbf{II}_1\dt.
\end{align}
In the last line we used the energy inequality
\eqref{pre-W22-est-2}. Recall that $\mathbf R$ denotes the right-hand side 
of \eqref{pre-W22-est-2}. For the estimate of the last integral, we
again apply \eqref{pre-W22-est-2} and use the definition of $\mathbf R$, with the result
\begin{align*}
	 \mint_{(t_o-\rho^2,t_o)} \chi(t)\mathbf{II}_1\dt
	&\le 
	(p+2\alpha)^2
	\biint_{\Omega_T\cap Q_\rho}\chi\phi^2\zeta^2(\bH(Du))\bH (Du)^{p+2\alpha-2}|D^2u|^2\,\dx\\
	&\quad +
	\|D\phi\|_{L^\infty}^2\biint_{\Omega_T\cap Q_\rho}\bH (Du)^{p+2\alpha}
	\,\dx\dt\\
	&\le 
	\frac{c\,(p+2\alpha)}{\rho^2}\,\mathbf R,
\end{align*}
where the constant $c$ depends on $n,N,m,M,p$, and $\Theta$.
Joining this with \eqref{hi-int-1} yields
\begin{align}\label{hi-int-2}\nonumber
	\bigg(\biint_{\Omega_T\cap Q_\rho}&
	[\chi\phi^2]^{1+\frac{2}{n}}\bH(Du)^{p+2\alpha+2\delta}\dx\dt\bigg)^{\frac
          n{n+2}}\\\nonumber
        &\le
        1+\bigg(\biint_{\Omega_T\cap Q_\rho}
        [\chi\phi^2]^{1+\frac{2}{n}}
        \zeta^4(\bH(Du))\bH(Du)^{p+2\alpha+2\delta}\dx\dt\bigg)^{\frac
          n{n+2}}\\
	&\le c\,(p+2\alpha)\mathbf R.
\end{align}
Our next goal is to estimate $\mathbf{R}_2$. In
view of the integrability assumption \eqref{assumption:b}, i.e.~$b\in L^\sigma(\Omega_T,\R^N)$ with
$\sigma>n+2$, we can use H\"older's inequality to estimate 
\begin{align}\label{b-term}
	\mathbf{R}_2
  	&\le
	\Theta^\frac{2}{\sigma}[b]_{\sigma,\Omega_T\cap Q_\rho}^2
	\bigg(\biint_{\Omega_T \cap Q_\rho}\big[\chi\phi^2\bH
  	(Du)^{2\alpha+(2-p)_+}\big]^{\frac{\sigma}{\sigma-2}}
  	\,\dx\dt\bigg)^{\frac{\sigma-2}\sigma}
\end{align}
where we defined
\begin{equation*}
  [b]_{\sigma,\Omega_T\cap Q_\rho}:=
     \bigg[ \rho^{\sigma-n-2}\iint_{\Omega_T \cap Q_{\rho } }|b|^\sigma\dx\dt\bigg]^\frac{1}{\sigma}.
\end{equation*}
Here we used the fact that $\rho^{n+2}/|\Omega_T\cap Q_\rho|$ is bounded by
$\Theta$.
In order to estimate the integral on the right-hand side of \eqref{b-term} further, we
interpolate the $L^{\frac\sigma{\sigma-2}}$-norm between the
$L^1$-norm and the $L^{\frac{n+2}{n}}$-norm, which is possible since
$\sigma>n+2$. 
For every $\kappa>0$, this yields the bound
\begin{align}\label{interpolation}\nonumber
  	\mathbf{R}_2
  	&\le
  	\Theta^\frac{2}{\sigma}[b]_{\sigma,\Omega_T\cap Q_\rho}^2
  	\bigg[
  	\kappa\bigg(\biint_{\Omega_T \cap Q_\rho}
	\big[\chi\phi^2\bH(Du)^{2\alpha+(2-p)_+}\big]^{\frac{n+2}{n}}
	\,\dx\dt\bigg)^{\frac{n}{n+2}}\\\nonumber
  	&\qquad\qquad\qquad\qquad +
  	\kappa^{-\frac{n+2}{\sigma-n-2}}
  	\biint_{\Omega_T \cap Q_\rho}
	\chi\phi^2\bH(Du)^{2\alpha+(2-p)_+}\,\dx\dt 
  	\bigg]\\\nonumber
  	&\le
  	\Theta^\frac{2}{\sigma}[b]_{\sigma,\Omega_T\cap Q_\rho}^2
  	\bigg[
   	\kappa\bigg(\biint_{\Omega_T \cap Q_\rho}
	[\chi\phi^2]^{\frac{n+2}{n}}
  	\big[\bH(Du)^{p+2\alpha+2\delta}+1\big]\,\dx\dt\bigg)^{\frac{n}{n+2}}\\
  	&\qquad\qquad\qquad\qquad +
  	\kappa^{-\frac{n+2}{\sigma-n-2}}
  	\biint_{\Omega_T \cap Q_\rho}\big[\bH(Du)^{p+2\alpha}+1\big]\,\dx\dt
	\bigg].
\end{align}
In the last line, we used the fact that
\begin{align*}
  \big[2\alpha+(2-p)_+\big]\tfrac{n+2}n
  \le
  p+2\alpha+2\delta
  \qquad\mbox{and}\qquad
  2\alpha+(2-p)_+
  \le p+2\alpha.
\end{align*}
The latter  hold true for any $\alpha\ge0$ and $p\ge1$.
Joining estimates~\eqref{hi-int-2}
and~\eqref{interpolation}, we arrive at
\begin{align*}
 	\Big[1 -
	c\,(p&+2\alpha)^3\Theta^\frac{2}{\sigma}
	[b]_{\sigma,\Omega_T\cap Q_\rho}^2\kappa\Big]
	\bigg[\biint_{\Omega_T\cap Q_\rho}
	[\chi\phi^2]^{1+\frac{2}{n}}
	\bH(Du)^{p+2\alpha+2\delta}\,\dx\dt\bigg]^{\frac{n}{n+2}}\\\nonumber
	&\le 
    c\,(p+2\alpha) \mathbf{R}_1 +
	c\,\Theta^\frac{2}{\sigma}(p+2\alpha)^3
	[b]_{\sigma,\Omega\cap Q_\rho}^{2}\kappa + 1\\
    &\phantom{\le\,} +
    c\,\Theta^\frac{2}{\sigma}(p+2\alpha)^{3}
    [b]_{\sigma,\Omega_T\cap Q_\rho}^{2}\kappa^{-\frac{n+2}{\sigma-n-2}}
    \biint_{\Omega_T \cap Q_\rho}\big[\bH(Du)^{p+2\alpha}+1\big]\,\dx\dt.
\end{align*}
At this stage, we choose the parameter $\kappa>0$ so small that 
\begin{equation*}
	c\,(p+2\alpha)^3 \Theta^\frac{2}{\sigma}[b]_{\sigma,\Omega_T\cap
          Q_\rho}^{2}
          \kappa =\tfrac12.
\end{equation*}
This implies in particular that the second term on the right-hand side equals $\frac12 $. On the other hand the coefficient in front of the last term on the right equals
\begin{align*}
	\tfrac12\kappa^{-1-\frac{n+2}{\sigma-n-2}}
	&=
	\tfrac12 \kappa^{-\frac{\sigma}{\sigma-n-2}}
	=
	\Big[c\,(p+2\alpha)^3 \Theta^\frac{2}{\sigma}[b]_{\sigma,\Omega_T\cap Q_\rho}^{2}\Big]^\frac{\sigma}{\sigma-n-2}.
\end{align*}
This turns the preceding estimate into 
\begin{align*}
  	&\bigg[\biint_{\Omega_T\cap Q_\rho}
	[\chi\phi^2]^{1+\frac{2}{n}}\bH(Du)^{p+2\alpha+2\delta}\dx\dt\bigg]^{\frac{n}{n+2}}\\\nonumber
	&\qquad\le 
        	c\,(p+2\alpha)^{\frac{3\sigma}{\sigma-n-2}}
	\bigg[1+\mathbf{R}_1+ [b]_{\sigma,\Omega_T\cap Q_\rho}^{\frac{2\sigma}{\sigma-n-2}}
         \biint_{\Omega_T \cap Q_\rho}
         \big[\bH(Du)^{p+2\alpha}+1\big]\dx\dt\bigg]
\end{align*}
for a constant $c$ depending only on $n,N,m,M,p,\sigma$, and $\Theta$.
By an application of Young's inequality, we can rewrite the above inequality and obtain the {\it reverse
H\"older type estimate}
\begin{align}\label{reverse:start}\nonumber
  	\bigg[\biint_{\Omega_T\cap Q_\rho}
	&
	[\chi\phi^2]^{1+\frac{2}{n}}\bH(Du)^{p+2\alpha+2\delta}\,\dx\dt\bigg]^{\frac{n}{n+2}}\\
	&\le 
        c\,(p+2\alpha)^{\frac{3\sigma}{\sigma-n-2}}\boldsymbol\gamma_o\,
        \biint_{\Omega_T\cap Q_\rho}\big[\bH(Du)^{q+2\alpha}
        +1\big]\,\dx\dt,
\end{align}
where we defined  $q:=\max\{p,2\}$ and moreover abbreviated
\begin{equation*}
	\boldsymbol\gamma_o
	:=
	1+\rho^2\|D\phi\|_{L^\infty}^2
        +\rho^2\|\partial_t\chi\|_{L^\infty}+[b]_{\sigma,\Omega_T\cap
          Q_\rho}^{\frac{2\sigma}{\sigma-n-2}}.
\end{equation*}
For the constant $c$ above we have the dependencies $c(n,N,m,M,p,\sigma,\Theta)$.
For the Moser iteration scheme we need to compare the exponents on both sides of \eqref{reverse:start}.
We have
\begin{equation*}
  p+2\alpha+2\delta
  =
  q+2\alpha+\tfrac4n(1+\alpha)-(q-p)
  >
  q+2\alpha,
\end{equation*}
since $p>\frac{2n}{n+2}>2-\frac4n$.  

\subsection{The iteration scheme}
We fix radii $r,s$ with $\frac\rho2\le r<s\le \rho$ and define
\begin{equation*}
 \rho_k:=r+\tfrac1{2^k}(s-r)\quad\mbox{and}\quad Q_k:= Q_{\rho_k}(x_o,t_o)
\end{equation*}
for $k\in\N_0$.
We choose cut-off  functions $\phi_k\in C^\infty_0(B_{\rho_k}(x_o), [0,1])$ such that $\phi_k\equiv 1$ on
$B_{\rho_{k+1}}(x_o)$ and $|D\phi_k|\le\frac{2^{k+2}}{s-r}$ and
$\chi_k\in W^{1,\infty}((t_o-\rho_k^2,t_o),[0,1])$ such that $\chi_k(t_o-\rho_k^2)=0$, 
$\chi_k\equiv 1$ on $(t_o-\rho_{k+1}^2, t_o)$, and $0\le\partial_t\chi_k\le \frac{2^{2(k+2)}}{(s-r)^2}$.
With these specifications inequality \eqref{reverse:start} yields
\begin{align}\label{reverse:start-1}\nonumber
  	\bigg[\frac{1}{|\Omega_T\cap Q_{k-1}|}&
	\iint_{\Omega_T\cap Q_{k}}
	\bH(Du)^{q+2\alpha(1+\frac{2}{n}) -(q-p)+\frac{4}{n}}\,\dx\dt\bigg]^{\frac{n}{n+2}}\\
	&\le 
        \frac{C_o4^k\rho^2}{(s-r)^2}\,(p+2\alpha)^{\frac{3\sigma}{\sigma-n-2}}
        \biint_{\Omega_T\cap Q_{k-1}}\big[\bH(Du)^{q+2\alpha}
        +1\big]\,\dx\dt,
\end{align}
for a constant $C_o$ of the type
\begin{equation}\label{def:C_o}
	C_o:=C\Big( 1+[b]_{\sigma,\Omega_T\cap
          Q_\rho}^{\frac{2\sigma}{\sigma-n-2}}\Big).
\end{equation}
Here $C$ denotes a universal constant depending on $n,N,m,M,p,\sigma$, and $\Theta$.
To bound the left-hand side of \eqref{reverse:start-1} we use the fact
\begin{equation*}
        \frac{|\Omega_T \cap Q_{k}|}{|\Omega_T \cap Q_{k-1}|}
        \ge
        \frac{|\Omega_T\cap Q_{\rho/2}|}{|Q_\rho|}
        \ge
        \frac{|\Omega\cap B_{\rho/2}|}{4|B_1|\rho^n}
        \ge \frac{1}{c(n)\Theta}.
\end{equation*}
We use this in \eqref{reverse:start-1} and obtain
\begin{align}\label{reverse:start-2}\nonumber
  	\bigg[
	\biint_{\Omega_T\cap Q_{k}}&
	\bH(Du)^{q+2\alpha(1+\frac{2}{n}) -(q-p)+\frac{4}{n}}\dx\dt\bigg]^{\frac{n}{n+2}}\\
	&\le 
        \frac{C_o4^k\rho^2}{(s-r)^2}(p+2\alpha)^{\frac{3\sigma}{\sigma-n-2}}
        \biint_{\Omega_T\cap Q_{k-1}}\big[\bH(Du)^{q+2\alpha}
        +1\big]\dx\dt,
\end{align}
for some  constant $C_o$ with  the same structure as the one in~\eqref{def:C_o}. We now define
recursively a sequence $(\beta_k)_{k\in\N_o}$  by $\beta_o:=0$ and 
\begin{equation*}
	2\beta_k:=2\beta_{k-1}\big(1+\tfrac2n\big)-(q-p)+\tfrac4n.
\end{equation*}
Induction leads to
\begin{equation}\label{formula-beta-k}
  \beta_k= \frac{4-n(q-p)}{4}
  \Big[\Big( 1+\frac2n\Big)^k-1\Big].
\end{equation}
The choice $\alpha=\beta_{k-1}$ turns \eqref{reverse:start-2} into
\begin{align}\label{reverse:start-3}\nonumber
  	\bigg[
	\biint_{\Omega_T\cap Q_{k}}&
	\bH(Du)^{q+2\beta_k}\,\dx\dt\bigg]^{\frac{n}{n+2}}\\
	&\le 
        \frac{C_o4^k\rho^2}{(s-r)^2}(1+\beta_k)^{\frac{3\sigma}{\sigma-n-2}}\bigg[
        \biint_{\Omega_T\cap Q_{k-1}}\bH(Du)^{q+2\beta_{k-1}}
        \,\dx\dt+1\bigg].
\end{align}
In the last line we used $\beta_{k-1}<\beta_k$ to replace 
 $p+2\beta_{k-1}$ by $2p(1+\beta_{k})$.
 The constant $C_o$ in the above estimate is up to a multiplicative factor
the same as the one from~\eqref{def:C_o}. This, however,  does not change the dependencies
in $C_o$.  
To proceed further let
\begin{equation*}
	\boldsymbol A_k:=\biint_{\Omega_T\cap Q_{k}}\bH(Du)^{q+2\beta_k}\,\dx\dt.
\end{equation*}
In terms of $\boldsymbol A_k$ the reverse H\"older inequality \eqref{reverse:start-3}
leads to a recursion formula
\begin{equation*}
	\boldsymbol A_k
	\le
	\bigg[\frac{C_o4^k\rho^2}{(s-r)^2}(1+\beta_k)^{\frac{3\sigma}{\sigma-n-2}}\bigg]^{1+\frac2n}
	\big(\boldsymbol A_{k-1}+1\big)^{1+\frac2n}\qquad
	\forall\,k\in\N.
\end{equation*}
Iteration of this inequality gives
\begin{equation*}
	\boldsymbol A_k
	\le
	\prod_{j=1}^k
	\bigg[ \frac{C_o4^j\rho^2}{(s-r)^2}(1+\beta_j)^{\frac{3\sigma}{\sigma-n-2}}\bigg]^{(1+\frac2n)^{k-j+1}}
	\big(\boldsymbol A_{0}+1\big)^{(1+\frac2n)^k}
\end{equation*}
for any $k\in\N$. Here we enlarged $C_o$ by a factor 2.
We take this inequality to the power $\frac{1}{q+2\beta_k}$ and obtain
\begin{equation}\label{iteration:1}
	\boldsymbol A_k^\frac{1}{q+2\beta_k}
	\le
	\prod_{j=1}^k
	\bigg[ \frac{C_o4^j\rho^2}{(s-r)^2}(1+\beta_j)^{\frac{3\sigma}{\sigma-n-2}}
	\bigg]^\frac{(1+\frac2n)^{k-j+1}}{q+2\beta_k}
	\big(\boldsymbol A_{0}+1\big)^\frac{(1+\frac2n)^k}{q+2\beta_k}.
\end{equation}
Note that
\begin{equation*}
	\lim_{k\to\infty} \frac{(1+\frac2n)^k}{q+2\beta_k}= \frac{2}{4-n(q-p)}.
\end{equation*}
Therefore, we have
\begin{equation}\label{limit-rhs}
	\lim_{k\to\infty} 
	\big(\boldsymbol A_{0}+1\big)^\frac{(1+\frac2n)^k}{q+2\beta_k}
	=
	\bigg[
	\biint_{\Omega_T\cap Q_{s}}\bH(Du)^{q}\,\dx\dt+1
	\bigg]^\frac{2}{4-n(q-p)}.
\end{equation}
With the abbreviation 
\begin{equation*}
	\boldsymbol\gamma :=\frac{4-n(q-p)}{2}\in(0,2],
\end{equation*}
formula~\eqref{formula-beta-k} takes the form 
\begin{equation*}
  \beta_k=\frac{\boldsymbol\gamma}2\Big[\Big(1+\frac2n\Big)^k-1\Big]
         \le \Big(1+\frac2n\Big)^k-1.
\end{equation*}
Therefore, for any $j\in\N$ we have the estimate
  \begin{equation*}
     \frac{C_o4^j\rho^2}{(s-r)^2}(1+\beta_j)^{\frac{3\sigma}{\sigma-n-2}}
    \le
    \widetilde C_oK^j,
  \end{equation*}
with the abbreviations  
\begin{align*}
  \widetilde C_o&:=\frac{C_o\rho^2}{(s-r)^2}
  =\frac{C\rho^2}{(s-r)^2}\Big( 1+[b]_{\sigma,\Omega_T\cap
          Q_\rho}^{\frac{2\sigma}{\sigma-n-2}}\Big)
\end{align*}
and
\begin{align*}
    K&:=4\big(1+\tfrac2n\big)^{\frac{3\sigma}{\sigma-n-2}}\ge1.
\end{align*}
  We use this to bound the product appearing in \eqref{iteration:1},
  with the result
  \begin{align*}
   	 \prod_{j=1}^k
    	\bigg[
    	\frac{C_o4^{j}\rho^2}{(s-r)^2}(1+\beta_j)^{\frac{3\sigma}{\sigma-n-2}}
    	\bigg]^{\frac{(1+\frac2n)^{k-j+1}}{q+2\beta_k}}
 	& \le
  	\prod_{j=1}^k \widetilde C_o^{\frac{(1+\frac2n)^{k-j+1}}{q+2\beta_k}}
  	\prod_{j=1}^k K^{\frac{j(1+\frac2n)^{k-j+1}}{q+2\beta_k}} \\
 	& \le
  	\prod_{j=1}^k 
	\widetilde C_o^{\frac{(1+\frac2n)^{k-j+1}}{\boldsymbol\gamma[(1+\frac2n)^k-1]}}
  	\prod_{j=1}^k 
	K^{\frac{j(1+\frac2n)^{k-j+1}}{\boldsymbol\gamma[(1+\frac2n)^k-1]}}.
\end{align*}
The first product on the right-hand side can be computed with the help of \eqref{A1} from Lemma~\ref{lem:A} applied with $A=\widetilde C_o$ and $\theta=1+\frac2n$. We obtain
\begin{align*}
	\prod_{j=1}^k
	\widetilde C_o^{\frac{(1+\frac2n)^{k-j+1}}{\boldsymbol\gamma[(1+\frac2n)^k-1]}}
	=
	\widetilde C_o^\frac{n+2}{2\boldsymbol\gamma}
	=
	\widetilde C_o^{\frac{n+2}{4-n(q-p)} }.
\end{align*}	
Similarly, the second product can be bounded with the help of \eqref{A2} from Lemma~\ref{lem:A} applied with $A=K$ and $\theta=1+\frac2n$. This yields
\begin{align*}
  	\prod_{j=1}^k 
  	K^{\frac{j(1+\frac2n)^{k-j+1}}{\boldsymbol\gamma[(1+\frac2n)^k-1]}}
  	&\le
  	K^{\frac{(n+2)^2}{4\boldsymbol\gamma}}
  	=
  	K^{\frac{(n+2)^2}{2(4-n(q-p))}}.
\end{align*}
Inserting this above, we obtain
\begin{align*}
  	\prod_{j=1}^k\bigg[
	\frac{C_o4^{j}\rho^2}{(s-r)^2}
	(1+\beta_j)^{\frac{3\sigma}{\sigma-n-2}}
	\bigg]^{\frac{(1+\frac2n)^{k-j+1}}{q+2\beta_k}}
  	\le
  	\big(K^{\frac{n+2}2}\widetilde C_o\big)^{\frac{n+2}{4-n(q-p)} },
\end{align*}
where $K$ depends only on $n$ and $\sigma$. In particular, the
right-hand side is independent of $k\in\N$.
This allows us to pass to the limit $k\uparrow\infty$ in
\eqref{iteration:1}. In view of~\eqref{limit-rhs}, this yields
\begin{align*}
  \limsup_{k\to\infty}\mathbf A_k^{\frac{1}{ q+2\beta_k}}
  &\le
  \big(K^{\frac{n+2}2}\widetilde C_o\big)^{\frac{n+2}{4-n(q-p)} }\bigg[
  \biint_{\Omega_T\cap Q_{s}}\bH(Du)^{q}\,\dx\dt+1
  \bigg]^{\frac{2}{4-n(q-p)}}\\
  &\le
  C\bigg[\frac{\rho^{n+2}\big(1+ [ b]_{\sigma,\Omega_T\cap
      Q_{\rho}}^{\frac{(n+2)\sigma}{\sigma-n-2}}\big)}{(s-r)^{n+2}}
  \biint_{\Omega_T \cap Q_{s}}
	\big( 1+|Du|^2\big)^{\frac{ q }{2}}
	\,\dx\dt\bigg]^{\frac{2}{4-n(q-p)}}.
\end{align*}
In the last line we used  \eqref{def:C_o}, i.e.~the special form of $C_o$, and the fact that 
$\mu\le1$.
Since $\rho_k\downarrow r$ and $\beta_k\uparrow\infty$, the
last estimate implies the following {\it sup-estimate for the gradient}
\begin{align}\label{sup-Q-r-s}
	\sup_{\Omega_T\cap Q_{r}}|Du|
	&=
	\lim_{k\to\infty}
	\bigg[\biint_{\Omega \cap Q_{r}}
	|Du|^{q+2\beta_k}
	\,\dx\dt\bigg]^{\frac{1}{ q+2\beta_k}}\nonumber\\
        &\le
        \limsup_{k\to\infty}\mathbf A_k^\frac{1}{ q+2\beta_k}\nonumber\\
        &\le
        C\bigg[\frac{\rho^{n+2}\big(1+ [ b]_{\sigma,\Omega_T\cap
            Q_{\rho}}^{\frac{(n+2)\sigma}{\sigma-n-2}}\big)}{(s-r)^{n+2}}
        \biint_{\Omega_T \cap Q_{s}}
	\big( 1+|Du|^2\big)^{\frac{ q }{2}}
	\,\dx\dt\bigg]^{\frac1q\frac{2q}{4-n(q-p)}},
\end{align}
for a constant $C$ that depends on $n,m,M,p,\sigma$, and  $\Theta$.
In the case $p\ge 2$ we have $q=p$, and therefore \eqref{sup-Q-r-s} simplifies to
\begin{align*}
	\sup_{\Omega_T\cap Q_{r}}|Du|
        &\le
        C\bigg[\frac{\rho^{n+2}\Big(1+ [b]_{\sigma,\Omega_T\cap
          Q_{\rho}}^{\frac{(n+2)\sigma}{\sigma-n-2}}\Big)}{(s-r)^{n+2}} \biint_{\Omega_T \cap Q_{s}}
	\big( 1+|Du|^2\big)^{\frac{ p}{2}}
	\,\dx\dt\bigg]^{\frac{d}p},
\end{align*}
where $d=\tfrac12 p$ is  the scaling deficit  from \eqref{def:deficit}.
With the choice $r=\frac\rho2$ and $s=\frac{3\rho}4$, this yields the asserted
sup-estimate for the gradient~\eqref{Lipschitz-estimate} in the case $p\ge2$.
Note that this is in perfect accordance with the interior estimate
\cite[Chapter VIII, Theorem 5.1]{DB}.

\subsection{Interpolation in the case $\frac{2n}{n+2}<p<2$}
To reduce the integrability exponent
in the sup-estimate from $q=2$ to $p$ in the singular case we need an additional interpolation
argument. To this end, we 
apply \eqref{sup-Q-r-s} with arbitrary radii $r,s$ satisfying 
$\frac\rho2\le r<s\le \frac{3\rho}4$.
On the right-hand side of
the estimate, we bound a part of the integrand by its supremum and
then apply Young's inequality with exponents $\frac{4-n(2-p)}{2(2-p)}$ and
$\frac{4-n(2-p)}{p(n+2)-2n}$.
Note that this is possible if and only if $\frac{2n}{n+2}<p<2$.
This procedure leads us to 
\begin{align*}
	\sup_{\Omega_T\cap Q_{r}}& \big(1+|Du|^2\big)^{\frac12}\\
        &\le
        C\bigg[\frac{\rho^{n+2}\big(1+ [ b]_{\sigma,\Omega_T\cap
            Q_{\rho}}^{\frac{(n+2)\sigma}{\sigma-n-2}}\big)}{(s-r)^{n+2}}
          \biint_{\Omega_T \cap Q_{s}}
	\big( 1+|Du|^2\big)
	\,\dx\dt\bigg]^{\frac{2}{4-n(2-p)}}\\[1.2ex]
        &\le
        C\sup_{\Omega_T\cap Q_{s}}\big(1+|Du|^2\big)^{\frac{2-p}{4-n(2-p)}}\\
        &\quad
        \cdot
        \bigg[\frac{\rho^{n+2}\big(1+[b]_{\sigma,\Omega_T\cap
            Q_{\rho}}^{\frac{(n+2)\sigma}{\sigma-n-2}}\big)}{(s-r)^{n+2}}
          \biint_{\Omega_T \cap Q_{s}}
	\big( 1+|Du|^2\big)^{\frac p2}
	\,\dx\dt\bigg]^{\frac{2}{4-n(2-p)}}\\[1.2ex]
        &\le
        \tfrac12\sup_{\Omega_T\cap Q_{s}}\big(1+|Du|^2\big)^{\frac12}\\
        &\quad+
        C\bigg[\frac{\rho^{n+2}\big(1+ [ b]_{\sigma,\Omega_T\cap
            Q_{\rho}}^{\frac{(n+2)\sigma}{\sigma-n-2}}\big)}{(s-r)^{n+2}}
          \biint_{\Omega_T\cap Q_{s}}
	\big( 1+|Du|^2\big)^{\frac p2}
	\,\dx\dt\bigg]^{\frac{2}{p(n+2)-2n}}.
\end{align*}
By a standard iteration argument (cf. \cite[Chapter~V, Lemma~3.1]{Giaquinta-book}), this implies
\begin{align}\label{sup-est-rho}\nonumber
	\sup_{\Omega_T\cap Q_{\rho/2}}&(1+|Du|^2)^{\frac12}\\
        &\le
        C\bigg[\Big(1+ [ b]_{\sigma,\Omega_T\cap
            Q_{\rho}}^{\frac{(n+2)\sigma}{\sigma-n-2}}\Big)
          \biint_{\Omega_T\cap Q_{3\rho/4}}
	\big( 1+|Du|^2\big)^{\frac p2}
	\,\dx\dt\bigg]^\frac{d}{p},
\end{align}
where $d=\frac{2p}{p(n+2)-2n}$ is the scaling deficit, cf.~\eqref{def:deficit}.
This is exactly the claimed bound \eqref{Lipschitz-estimate} in the
singular range $\frac{2n}{n+2}<p<2$, and completes the proof of the sup-estimate from
Proposition~\ref{prop:apriori}. Note also that the sup-gradient estimate 
\eqref{sup-est-rho} is again in perfect accordance with the corresponding interior estimate 
\cite[Chapter VIII, Theorem 5.2']{DB}.

\section{Regularization}\label{sec:regularization}

In this section we describe the regularization procedure that will
allow us to extend the a priori estimate to the general
case. We consider the situation stated
in 
Theorem~\ref{thm:main}, i.e.~we let $\Omega\subset\R^n$ be
a   bounded convex domain, and suppose that $u\in
L^p(0,T;W^{1,p}(\Omega,\R^N))$ is a solution to \eqref{system}, where 
\eqref{assumption:a(0)} -- \eqref{assumption:b}
are in force. Moreover, we assume that for some $x_o\in\partial\Omega$
and $\rho>0$ we have $u\equiv0$ on $(\partial\Omega)_T\cap
Q_{2\rho}(z_o)$ in the sense of traces.

\subsection{Approximation of the domain}

For any $\eps\in(0,1]$ we consider the parallel set
$\widetilde\Omega_\eps:=\{x\in\R^n\colon \dist(x,\Omega)<\frac{3}2\eps\}$. Note that $\widetilde\Omega_\eps$ is convex as $\Omega$ is convex. By a well-known result from convex analysis (see e.g. \cite[\S XIII.2, Satz 2]{Marti:book}), the 
domains $\widetilde\Omega_\eps$ can be approximated in Hausdorff distance by 
smooth convex sets $\Omega_\eps$ with 
\begin{equation*}
  \dist_{\mathcal{H}}\big(\Omega_\eps,\widetilde\Omega_\eps\big)
  <
  \tfrac12\eps.
\end{equation*}
In particular, the regularized sets $\Omega_\eps$ satisfy
\begin{equation}\label{choice-Omega-delta}      
  \big\{x\in\R^n\colon \dist (x,\Omega)<\eps\big\}
  \subset\Omega_\eps\subset
  \big\{x\in\R^n\colon \dist
  (x,\Omega)<2\eps\big\}.
\end{equation}
Since the domains $\Omega_\eps$ approximate $\Omega$ from the outside, we obtain
\begin{equation}\label{measure-density-Omega-delta}
  \sup_{\eps\in(0,1]}\frac{\rho^n}{|\Omega_\eps\cap B_{\rho/2}(x_o)|}
  \le
  \frac{\rho^n}{|\Omega\cap B_{\rho/2}(x_o)|}
  =
  2^n\Theta_{\rho/2}(x_o)
\end{equation}
for every $x_o\in\partial\Omega$ and $\rho>0$,
with the constant $\Theta_{\rho/2}(x_o)$ introduced
in~\eqref{def-Theta}.
As a result, the constants in the a priori estimate will be
independent of $\eps\in(0,1]$.

\subsection{Regularization of the coefficients}\label{section:reg-coeff}
We regularize the coefficients by means of a
mollifyer $\phi\in C^\infty_0(\R,[0,\infty))$
with $\spt\phi\subset (-1,1)$ and $\int_\R\phi\,\d x=1$. For $\eps\in(0,1]$ we let
$\phi_\eps(x):=\eps^{-1}\phi(\frac x\eps)$ and
$$
	\mathbf c_\eps(s):=\big(\phi_\eps\ast \mathbf c\big)(s),
	\quad\mbox{where}\ \ 
	\mathbf c(s):= 
	\left\{
	\begin{array}{cl}
		\mathbf{a}(\mathrm{e}^{s}), 
		&\mbox{if $\mu>0$,} \\[5pt]
		\mathbf{a}\big(\sqrt{\eps^2 +\mathrm{e}^{2s}}\,\big)
		, &\mbox{if $\mu=0$,}
	\end{array}
	\right.
$$
for any $s\in\R$. 
The regularized coefficients $\mathbf a _{\eps}$ are defined by
$$
	\mathbf a_{\eps} (r):=  \mathbf c_\eps(\log r),
        \quad\mbox{for }r>0.
$$
Similarly as in \cite[Section~4.2]{BDMS} we obtain the following ellipticity
and growth conditions for $\mathbf a_{\eps}$; see also Appendix~\ref{app:reg} for the proof. For  any $r>0$ we have
\begin{equation}\label{growth-a-delta}
  \left\{
  \begin{array}{c}
    \tfrac{m}{c}(\lambda^2+r^2)^{\frac{p-2}2}
    \le
    \mathbf a_{\eps}(r)
    \le
    cM(\lambda^2+r^2)^{\frac{p-2}2},\\[5pt]
    \tfrac{m}{c}(\lambda^2+r^2)^{\frac{p-2}2}
    \le
    \mathbf a_{\eps}'(r)r+\mathbf a_{\eps}(r)
    \le
    cM(\lambda^2+r^2)^{\frac{p-2}2},\\[5pt]
    |\mathbf a_{\eps}''(r)r^2|
    \le
    \frac{cM}\eps(\lambda^2+r^2)^{\frac{p-2}2},
  \end{array}
  \right.
\end{equation}
with a constant $c=c(n,p)$ and 
\begin{equation}
	\lambda
	:=
	\left\{
	\begin{array}{cl}
		\mu ,&\mbox{if $\mu>0$,} \\[2pt]
		\epsilon, &\mbox{if $\mu=0$.}
	\end{array}
	\right.\label{lambda-def}
      \end{equation}
Moreover, we have 
\begin{align}\label{a-delta-convergence}
	| \mathbf a_{\eps} (r)-\mathbf{a}(r)|
	\le 
	2c(p)M \eps\max\big\{ 1,\mathrm{e}^{p-2}\big\}
        \big(\lambda^2 + r^2\big)^{\frac{p-2}{2}}
\end{align}
for any $r>0$; see also Appendix~\ref{app:reg} for the proof.
%
%
%
%
\subsection{Weak solutions to the regularized problems}

Here we assume that
\begin{equation*}
  \mbox{$u\equiv0\;$  on $(\partial\Omega)_T\cap Q_{2\rho}(x_o)$, }
\end{equation*}
where $Q_{2\varrho}(z_o)$ is a parabolic cylinder with 
$x_o\in \partial\Omega$ and $(t_o-4\rho^2,t_o)\subset(0,T)$. 
For a cut-off function
  $\widetilde\eta\in C^\infty(0,\infty;[0,1])$ with $\widetilde\eta\equiv 1$ on
  $[2,\infty)$ and $\widetilde\eta\equiv0$ on $(0,1)$ we
  consider the boundary values
  \begin{equation}\label{def:g_eps}
    g_\eps(x,t):=\eta_\eps(x) u(x,t)
    \quad\mbox{with}\quad
    \eta_\eps(x):=\widetilde\eta\Big(\frac{\dist(x,\partial\Omega)}\eps\Big)
    \mbox{ for $x\in\Omega$}.
  \end{equation}
  We extend this function to $\R^n\times(0,T)$ by letting
  $g_\eps\equiv 0$ on $(\R^n\setminus\Omega)_T$. 
  Note that the extension satisfies $g_\eps\in
  L^p(0,T;W^{1,p}_0(\Omega_\epsilon,\R^N))$.
For the inhomogeneity we consider the regularization 
$b_\eps:=\phi_{\epsilon/2}\ast b$, for a standard mollifier in space-time $\phi\in
C^\infty_0(Q_1,\R)$. 
Due to the construction of $\Omega_\eps$ (see \eqref{choice-Omega-delta}) we have  
\begin{equation}\label{b-compact-support}
  \spt b_\eps \Subset \Omega_\eps\times\R.
\end{equation}
We let $\mathbf{a}_\eps$ be the coefficients constructed in
Section~\ref{section:reg-coeff}.
By 
$$
	u_\epsilon
	\in 
	g_\eps+ L^p\big(t_o-\rho^2,t_o;W^{1,p}_0(\Omega_\epsilon\cap B_\rho(x_o),\R^N)\big)
$$ 
we denote the weak solution to the Cauchy-Dirichlet problem 
\begin{equation}\label{eq-ueps}
	\left\{
	\begin{array}{cl}
		\partial_t u_\epsilon - 
		\Div \big(\mathbf{a}_\eps(|Du_\epsilon|)Du_\epsilon\big)
		=
		b_\eps 
		& \quad\mbox{in $(\Omega_\epsilon)_T\cap Q_\rho(z_o)$} \\[5pt]
                u_\epsilon =
                   g_\eps
		& \quad\mbox{on $\partial_p((\Omega_\epsilon)_T\cap Q_\rho(z_o))$.}
	\end{array}
	\right.
\end{equation}
Note that $u_\eps=0$ on $(\partial\Omega_\eps)_T\cap Q_\rho(z_o)$ and
$u_\eps = \eta_\eps u$ on $(\Omega_\eps)_T\cap \partial_p Q_\rho (z_o)$. 
Using a reflection argument, interior regularity theory and
up-to-the-boundary Schauder estimates we can show that $u_\eps$ is
smooth up to the boundary component
$(\partial\Omega_\eps)_T\cap Q_\rho(z_o)$; see Appendix~\ref{app:smooth}.

\section{Proof of Theorem \ref{thm:main}}

The proof of the gradient estimate will be achieved in Section~\ref{sec:proof}. Prior to that, we shall prove an energy estimate for $u_\epsilon$. 

\subsection{An energy estimate for the approximating solutions}
Throughout this section we  omit the reference to the center $z_o$ in our notation. 
From \cite[Corollary~3.11]{Kinnunen-Martio} we recall the following result;
note that the constant $\gamma$ in \cite[Corollary~3.11]{Kinnunen-Martio} can be chosen as $\gamma=\frac12$ due to the convexity of $\Omega$; see \cite[Example~3.6\,(4)]{Kinnunen-Martio}.

\begin{lemma}[Hardy's inequality]\label{lem:Hardy}
Let $1<p<\infty$ and suppose that $\Omega\subset\R^n$ is a bounded open convex set. Then there is  a constant $c$ depending on $n$ and $p$ such that whenever $u\in W^{1,p}_0(\Omega)$ there holds
\begin{equation*}
	\int_\Omega \bigg(\frac{|u(x)|}{\dist(x,\partial\Omega)}\bigg)^p\,\dx
	\le
	c \int_\Omega|Du(x)|^p\,\dx.
\end{equation*}
\end{lemma}

In the following we let
\begin{equation}\label{def:V}
	\boldsymbol V_\epsilon
	:=
	L^\infty\big(t_o-\rho^2,t_o;L^2(\Omega_\eps\cap B_\rho,\R^N)\big)\cap 
	L^p\big(t_o-\rho^2,t_o;W_0^{1,p}(\Omega_\eps\cap B_\rho,\R^N)\big)
\end{equation}
with norm 
$$
	\|\varphi\|_{\boldsymbol V_\eps}
	:=
	\|\varphi\|_{L^\infty-L^2} +
	\|\varphi\|_{L^p-W^{1,p}}. 
$$
We start  with an estimate for the spatial gradient of the boundary values.

\begin{lemma}\label{lem:Dge}
Let $u$ be a weak solution to \eqref{system} in
$\Omega_T$ with $u\equiv0$ on $(\partial\Omega)_T\cap Q_{2\rho}(z_o)$
in the sense of traces, and $g_\eps=\eta_\eps u$ be constructed as
in~\eqref{def:g_eps}.
Then we have
\begin{align*}
  \iint_{\Omega_T\cap Q_\rho} |Dg_\eps|^p\dx\dt
  \le
    c \iint_{\Omega_T\cap Q_{2\rho}} \big[ |Du|^p+\rho^{-p}|u|^p\big]\dx\dt,
\end{align*}
with
a constant $c=c(n,p)$.
\end{lemma}

\begin{proof}
We choose a
standard cut-off function $\zeta\in C^\infty_0(B_{2\rho},[0,1])$ with
$\zeta\equiv1$ on $B_\rho$ and $|D\zeta|\le\frac2\rho$ on
$B_{2\rho}$. Then we apply
Hardy's inequality from Lemma~\ref{lem:Hardy} to the
function $\zeta u$ on the time-slices $\Omega\times\{t\}$
for a.e.~$t\in(t_o-\rho^2,t_o)$, with the result 
\begin{align*}
  \iint_{\Omega_T\cap Q_\rho} |Dg_\eps|^p\dx\dt
  &\le
  c \iint_{\Omega_T\cap Q_\rho} |Du|^p\dx\dt
  +
  \frac{c}{\eps^p}\iint_{\Omega_T\cap Q_\rho\cap\spt(D\eta_\eps)}
    |\zeta u|^p\dx\dt\\
  &\le
  c \iint_{\Omega_T\cap Q_\rho} |Du|^p\dx\dt
  +
  c \iint_{\Omega\times(t_o-\rho^2,t_o)}
    \bigg(\frac{|\zeta u|}{\dist(x,\partial\Omega)}\bigg)^p\dx\dt\\
  &\le
    c \iint_{\Omega_T\cap Q_\rho} |Du|^p\dx\dt
    +
    c \iint_{\Omega\times(t_o-\rho^2,t_o)} |D(\zeta u)|^p\dx\dt\\
  &\le
    c \iint_{\Omega_T\cap Q_{2\rho}} \big[ |Du|^p+\rho^{-p}|u|^p\big]\dx\dt.
\end{align*}
This proves the claimed estimate.
\end{proof}

In the next lemma we provide an estimate for the distributional time derivative of the boundary values.

\begin{lemma}\label{lem:time-u}
Let $u$ be a weak solution to \eqref{system} in
$\Omega_T$, and $g_\eps=\eta_\eps u$ be constructed as in~\eqref{def:g_eps}.
Then,  for any $\varphi\in C_0^\infty((\Omega_\eps)_T\cap Q_\rho,\R^N)$ we have
\begin{align}\label{time-deriv-g}
  \bigg|\iint_{\Omega_T} g_\eps\cdot\partial_t\varphi\,\dx\dt\bigg|
  &\le
    c \bigg[\iint_{\Omega_T\cap\spt\varphi}\big(\lambda^p+|Du|^p\big)\dx\dt\bigg]^{\frac{p-1}p}
    \|D\varphi\|_{L^p((\Omega_\eps)_T\cap Q_\rho)}\\[1ex]\nonumber
    &\phantom{\le\,}+
    \|b\|_{L^{\frac{p(n+2)}{p(n+2)-n}}(\Omega_T\cap\spt\varphi)}
    \|\varphi\|_{L^{\frac{p(n+2)}{n}}((\Omega_\eps)_T\cap Q_\rho)},
\end{align}
with
a constant $c=c(n,p,M)$ and the parameter $\lambda$ from \eqref{lambda-def}.
In particular,  $\partial_t g_\eps\in \boldsymbol V_\epsilon'$. 
\end{lemma}

 Note that
$b\in L^{\frac{p(n+2)}{p(n+2)-n}}(\Omega_T)$, since
$\sigma>n+2>\frac{p(n+2)}{p(n+2)-n}$. Therefore, the right-hand
side of \eqref{time-deriv-g} is finite.

\begin{proof}
  Let
  $\varphi\in C_0^\infty((\Omega_\eps)_T\cap Q_\rho,\R^N)$, and consider
  the cut-off function $\eta_\eps(x)$ from~\eqref{def:g_eps}.
  Testing the weak form of \eqref{system} with $\eta_\eps\varphi$ and recalling~\eqref{bounds-a}, we estimate
  \begin{align}\label{time-deriv-g-1}
        \bigg|\iint_{\Omega_T}& g_\eps\cdot \partial_t\varphi\,\dx\dt\bigg|\\\nonumber
        &=
	\bigg|\iint_{\Omega_T} u\cdot \partial_t(\eta_\eps\varphi)\,\dx\dt\bigg|\\\nonumber
	&=
	\bigg|\iint_{\Omega_T} 
	\Big[ \mathbf{a}(Du)Du\cdot D(\eta_\eps\varphi) - 
	\eta_\eps b\cdot\varphi\Big]\,\dx\dt\bigg| \\\nonumber
	&\le
	cM \iint_{\Omega_T} 
	\big(\mu^2+|Du|^2\big)^{\frac{p-2}{2}}|Du|\, |D(\eta_\eps\varphi)|
          \,\dx\dt 
          + 
	\iint_{\Omega_T} |b| |\varphi| \,\dx\dt \\\nonumber
	&\le
          cM
          \bigg[\iint_{\Omega_T\cap \spt\varphi}\big(\lambda^p+|Du|^p\big)\dx\dt\bigg]^{\frac{p-1}p}\|D(\eta_\eps\varphi)\|_{L^p(\Omega_T)}\\\nonumber
	&\phantom{\le\,}+
	\|b\|_{L^{\frac{p(n+2)}{p(n+2)-n}}(\Omega_T\cap\spt\varphi)} \|\varphi\|_{L^{\frac{p(n+2)}{n}}((\Omega_\eps)_T\cap Q_\rho)}.
  \end{align}
 For the norm in the second-to-last
term, we have 
\begin{equation}\label{bound-cut-off}
  \|D(\eta_\eps\varphi)\|_{L^p(\Omega_T)}^p
  \le
  c\iint_{\Omega_T\cap Q_\rho}|D\varphi|^p\dx\dt
  +
  \frac{c}{\eps^p}\iint_{\Omega_T\cap
    Q_\rho\cap\spt(D\eta_\eps)}|\varphi|^p\dx\dt.
\end{equation}
In order to bound the last integral, we observe that for points $x$ in the domain of
integration, we have
\begin{equation*}
  \dist(x,\partial\Omega_\eps)
  \le
  2\eps+\dist(x,\partial\Omega)
  \le
  4\eps,
\end{equation*}
by the construction of $\Omega_\eps$ and since 
$\spt(D\eta_\eps)$ is contained in the $2\eps$-neighborhood of
$\partial\Omega$.
Therefore, we can apply Hardy's inequality from
Lemma~\ref{lem:Hardy} on the time slices $\Omega_\eps\times\{t\}$ for
a.e. $t\in (t_o-\rho^2,t_o)$, after extending $\varphi$ by zero on
$(\Omega_\eps\times\{t\})\setminus B_\rho$. Note that the constant in
Hardy's inequality only depends on $n$ and $p$, but independent of
$\eps$. As a result, we obtain
\begin{align*}
\frac{1}{\eps^p}\iint_{\Omega_T\cap
  Q_\rho\cap\spt(D\eta_\eps)}|\varphi|^p\dx\dt
&\le
c\iint_{\Omega_\eps\times(t_o-\rho^2,t_o)}
  \bigg(\frac{|\varphi|}{\dist(x,\partial\Omega_\epsilon)}\bigg)^p\dx\dt\\
&\le
 c\iint_{\Omega_\eps\times(t_o-\rho^2,t_o)}|D\varphi|^p\dx\dt.
\end{align*}
Joining this bound with~\eqref{bound-cut-off}, we arrive at 
$
  \|D(\eta_\eps\varphi)\|_{L^p}
  \le
  c\|D\varphi\|_{L^p},
$
for a constant $c=c(n,p)$.
Using this  in \eqref{time-deriv-g-1}, we deduce the asserted
estimate~\eqref{time-deriv-g}. Finally, we note that 
Gagliardo-Nirenberg's inequality implies 
\begin{align*}
	&\|\varphi\|_{L^{\frac{p(n+2)}{n}}((\Omega_\eps)_T\cap Q_\rho)}\\
	&\qquad\le
	\bigg[\iint_{(\Omega_\eps)_T\cap Q_\rho}|D\varphi|^p \,\dx\dt 
	\bigg(\sup_{t\in(t_o-\rho^2,t_o)}\int_{\Omega_\eps\cap B_\rho} |\varphi(\cdot,t)|^2 \,\dx
	\bigg)^{\frac{p}{n}}  
	\bigg]^{\frac{n}{p(n+2)}} \\[1ex]
	&\qquad\le
	c \|\varphi\|_{\boldsymbol V_\epsilon} .
\end{align*}
Therefore, the estimate~\eqref{time-deriv-g} can be rewritten in the form
\begin{align*}
  \bigg|\iint_{\Omega_T}& g_\eps\cdot\partial_t\varphi\,\dx\dt\bigg|\\\nonumber
  &\le
    c \Bigg[\bigg(\iint_{\Omega_T\cap\spt\varphi}\big(\lambda^p+|Du|^p\big)\dx\dt\bigg)^{\frac{p-1}p}+
    \|b\|_{L^{\frac{p(n+2)}{p(n+2)-n}}(\Omega_T\cap\spt\varphi)}\Bigg]
    \|\varphi\|_{\boldsymbol V_\epsilon}
\end{align*}
for any $\varphi\in C_0^\infty\big((\Omega_\eps)_T\cap Q_\rho,\R^N\big)$.
This proves the assertion $\partial_tg_\eps\in\boldsymbol  V_\epsilon'$. 
\end{proof}

We use the preceding estimate of the distributional time derivative of
$g_\eps$ for the proof of the desired energy estimate. The difficulty comes from the fact that $u$ and $u_\epsilon$ are solutions on different domains $\Omega_T$ and $(\Omega_\epsilon)_T$. For ease of notation, we define
\begin{equation*}
  V_\lambda(A):=\big(\lambda^2+|A|^2\big)^{\frac{p-2}{4}}A,
  \qquad\mbox{for $A\in \R^{Nn}$}.
\end{equation*}

\begin{lemma}[Energy estimate]\label{lem:energy}
For any
$\eps>0$ and any weak solution $u_\eps$ to the Cauchy-Dirichlet
  problem~\eqref{eq-ueps} we have
  \begin{align*}
	&  \sup_{\tau\in(t_o-\rho^2,t_o)}\int_{(\Omega_\epsilon\cap B_\rho)\times\{\tau\}} 
	|u_\epsilon-g_\eps|^2 \,\dx +
	\iint_{(\Omega_\epsilon)_T\cap Q_\rho} 
	\big|V_\lambda(Du_\epsilon)\big|^2\,\dx\dt \\\nonumber
         &\qquad \le c
        \iint_{\Omega_T\cap Q_{2\rho}}\big[\lambda^p+|Du|^p+\rho^{-p}|u|^p\big]\dx\dt
          +
    \big\||b|+|b_\eps|\big\|_{L^{\frac{p(n+2)}{p(n+2)-n}}((\Omega_\eps)_T\cap
           Q_\rho)}^{\frac{p(n+2)}{p(n+2)-n-p}}
  \end{align*}
  with a constant $c=c(n,p,m,M)$.
\end{lemma}

\begin{proof}
For fixed $\tau\in (t_o-\rho^2,t_o)$ and $\delta\in(0,t_o-\tau)$ we let 
\begin{equation*}
	\zeta_\delta(t)
	:=
	\left\{\begin{array}{cl}
		1, & \mbox{for $t\in [t_o-\rho^2,\tau]$}, \\[3pt]
		\frac{\tau+\delta-t}{\delta}, & \mbox{for $t\in (\tau,\tau+\delta)$}, \\[3pt]
		0, & \mbox{for $t\in [\tau+\delta,t_o]$.}
	\end{array}\right.
\end{equation*}
As in Lemma~\ref{lem:time-u} one easily checks that solutions $u_\eps$ to the parabolic systems
\eqref{eq-ueps}  own a distributional time derivative $\partial_t
u_\epsilon\in \boldsymbol V_\epsilon'$. Therefore, the 
testing function $\zeta^2_\delta(u_\epsilon-g_\eps)\in V_\epsilon$ is admissible
in the weak form of \eqref{eq-ueps}, which implies
\begin{align}\label{weak-1}
	\big\langle \partial_tu_\eps, \zeta_\delta^2(u_\epsilon-g_\eps)\big\rangle
	&+ 
	\iint_{(\Omega_\epsilon)_T\cap Q_\rho} \zeta_\delta^2\mathbf{a}_\eps(Du_\epsilon) 
	Du_\epsilon\cdot (Du_\epsilon-Dg_\eps)
	\,\dx\dt \nonumber\\
	&=
	\iint_{(\Omega_\epsilon)_T\cap Q_\rho} 
	\zeta_\delta^2 b_\epsilon\cdot(u_\epsilon-g_\eps)\,\dx\dt.
\end{align}
Here, $\langle\cdot,\cdot\rangle$ denotes the duality pairing on 
$\boldsymbol V_\epsilon'\times \boldsymbol V_\epsilon$. 
We rewrite the first term on the left-hand side in the form
\begin{align*}
     \big\langle \partial_tu_\eps& ,
       \zeta_\delta^2(u_\epsilon-g_\eps)\big\rangle\\
     &=
       \big\langle \partial_t(u_\eps-g_\eps),
       \zeta_\delta^2(u_\epsilon-g_\eps)\big\rangle
       +
       \big\langle \partial_tg_\eps,
       \zeta_\delta^2(u_\epsilon-g_\eps)\big\rangle\\
     &=
         \big\langle \partial_t(\zeta_\delta (u_\eps-g_\eps)),
          \zeta_\delta(u_\epsilon-g_\eps)\big\rangle
          -
          \iint_{(\Omega_\epsilon)_T\cap Q_\rho} 
	  \zeta_\delta'\zeta_\delta |u_\epsilon-g_\eps|^2
       \,\dx\dt\\
       &\phantom{=\,}+
       \big\langle \partial_tg_\eps,
       \zeta_\delta^2(u_\epsilon-g_\eps)\big\rangle\\
	&=:
	\mathbf{I}(\delta) + \mathbf{II}(\delta) +\mathbf{III}(\delta),
\end{align*} 
with the obvious meaning of  $\mathbf{I}(\delta)$ to $\mathbf{III}(\delta)$.
For the first term we find 
\begin{align*}
  \mathbf{I}(\delta)
  =
  \tfrac12 \iint_{(\Omega_\eps)_T\cap Q_\rho} 
  \partial_t|\zeta_\delta (u_\epsilon-g_\eps)|^2 \,\dx\dt 
  =
  0,
\end{align*}
since $\zeta_\delta(t_o)=0$ and $u_\eps=g_\eps$ on the initial time
slice $(\Omega_\eps\cap B_\rho)\times\{t_o-\rho^2\}$.
By the mean value theorem we obtain  
\begin{align*}
	\lim_{\delta\downarrow 0} \mathbf{II}(\delta)
	&=
	\tfrac12 \int_{(\Omega_\epsilon\cap B_\rho)\times\{\tau\}} 
	|u_\epsilon-g_\eps|^2 \,\dx .
\end{align*} 
Finally, we estimate the third term by means of
Lemma~\ref{lem:time-u}, with the result
\begin{align*}
  \big|\mathbf{III}(\delta)\big|
  &\le
  c \bigg[\iint_{\Omega_T\cap Q_\rho}\big(\lambda^p+|Du|^p\big)\dx\dt\bigg]^{\frac{p-1}p}
    \|Du_\eps-Dg_\eps\|_{L^p((\Omega_\eps)_T\cap Q_\rho)}\\[1ex]
    &\quad+
    \|b\|_{L^{\frac{p(n+2)}{p(n+2)-n}}((\Omega_\eps)_T\cap Q_\rho)}
    \|u_\eps-g_\eps\|_{L^{\frac{p(n+2)}{n}}((\Omega_\eps)_T\cap Q_\rho)}.
\end{align*}
For the last term in~\eqref{weak-1}, a straightforward application of
H\"older's inequality yields
\begin{align*}
  \iint_{(\Omega_\epsilon)_T\cap Q_\rho} &
  \zeta_\delta^2 b_\epsilon\cdot(u_\epsilon-g_\eps)\,\dx\dt\\
  &\le
  \|b_\eps\|_{L^{\frac{p(n+2)}{p(n+2)-n}}((\Omega_\eps)_T\cap Q_\rho)}
    \|u_\eps-g_\eps\|_{L^{\frac{p(n+2)}{n}}((\Omega_\eps)_T\cap Q_\rho)}.
\end{align*}
The preceding considerations allow us to pass to the limit
$\delta\downarrow0$ in \eqref{weak-1}. 
In the term not yet considered, i.e.~the one containing the
coefficients $\mathbf a_\eps$, the passage to the limit under the
integral can be justified by dominated convergence. Overall we get 
\begin{align}\label{weak-2}\nonumber
	 \tfrac12 \int_{(\Omega_\epsilon\cap B_\rho)\times\{\tau\}}& 
	|u_\epsilon-g_\eps|^2 \,\dx +
	\iint_{(\Omega_\epsilon)_\tau\cap Q_\rho} 
	\mathbf{a}_\eps(Du_\epsilon) 
	Du_\epsilon\cdot (Du_\epsilon-Dg_\eps)
	\,\dx\dt \\\nonumber
	&\le
          c \bigg[\iint_{\Omega_T\cap Q_\rho}\big(\lambda^p+|Du|^p\big)\dx\dt\bigg]^{\frac{p-1}p}
    \|Du_\eps-Dg_\eps\|_{L^p((\Omega_\eps)_T\cap Q_\rho)}\\[1ex]
    &\phantom{\le\,}+
    2\big\||b|+|b_\eps|\big\|_{L^{\frac{p(n+2)}{p(n+2)-n}}((\Omega_\eps)_T\cap Q_\rho)}
    \|u_\eps-g_\eps\|_{L^{\frac{p(n+2)}{n}}((\Omega_\eps)_T\cap Q_\rho)}
\end{align}
for any $\tau\in(t_o-\rho^2,t_o)$. 
By the growth properties~\eqref{growth-a-delta} of $\mathbf{a}_\eps$ and
Young's inequality for the $V_\lambda$-function \cite[Lemma
2.3]{Acerbi-Mingione-electro} we obtain for the diffusion term 
\begin{align*}
	&\iint_{(\Omega_\epsilon)_\tau\cap Q_\rho} 
	\mathbf{a}_\eps(Du_\epsilon) Du_\epsilon\cdot (Du_\epsilon-Dg_\eps)
	\,\dx\dt \\
	&\qquad\ge
	\tfrac{m}{c} \iint_{(\Omega_\epsilon)_\tau\cap Q_\rho} 
	|V_\lambda(Du_\epsilon)|^2 
          \,\dx\dt\\
         &\qquad\qquad-
	cM \iint_{(\Omega_\epsilon)_\tau\cap Q_\rho} 
	\big(\lambda^2 + |Du_\epsilon|^2\big)^{\frac{p-2}{2}} 
	|Du_\epsilon||Dg_\eps| \,\dx\dt \\
	&\qquad\ge
	\tfrac{m}{2c} \iint_{(\Omega_\epsilon)_\tau\cap Q_\rho} 
	|V_\lambda(Du_\epsilon)|^2 
	\,\dx\dt -
	c \iint_{\Omega_\tau\cap Q_\rho} 
	|V_\lambda(Dg_\eps)|^2 \,\dx\dt.
\end{align*} 
We join the two preceding estimates, take the supremum over
$\tau\in(t_o-\rho^2,t_o)$ in the first term on the left-hand side and
let $\tau\uparrow t_o$ in the second one. This gives 
\begin{align}\label{weak-3}
	 \mathbf S&+
	\iint_{(\Omega_\epsilon)_T\cap Q_\rho} 
	|V_\lambda(Du_\epsilon)|^2\,\dx\dt \\\nonumber
	&\le
         c \bigg[\iint_{\Omega_T\cap Q_\rho}\big(\lambda^p+|Du|^p\big)\dx\dt\bigg]^{\frac{p-1}p}
    \|Du_\eps-Dg_\eps\|_{L^p((\Omega_\eps)_T\cap Q_\rho)}\\[1ex]\nonumber
    &\phantom{\le\,}
    +
    c\big\||b|+|b_\eps|\big\|_{L^{\frac{p(n+2)}{p(n+2)-n}}((\Omega_\eps)_T\cap Q_\rho)}
    \|u_\eps-g_\eps\|_{L^{\frac{p(n+2)}{n}}((\Omega_\eps)_T\cap Q_\rho)}\\
    &\phantom{\le\,}+
         c \iint_{\Omega_T\cap Q_\rho} 
	|V_\lambda(Dg_\eps)|^2 \,\dx\dt \nonumber
\end{align} 
with the abbreviation
\begin{equation*}
	\mathbf S:=\sup_{\tau\in(t_o-\rho^2,t_o)}\int_{(\Omega_\epsilon\cap B_\rho)\times\{\tau\}} 
	|u_\epsilon-g_\eps|^2 \,\dx.
\end{equation*}
The inequalities of Gagliardo-Nirenberg and Young provide us with the estimate
\begin{align*}
  \|u_\eps-g_\eps\|_{L^{\frac{p(n+2)}{n}}((\Omega_\eps)_T\cap Q_\rho)}
  &\le
  c\bigg[\mathbf S^{\frac{p}{n}}
    \iint_{(\Omega_\eps)_T\cap Q_\rho}|Du_\eps-Dg_\eps|^p \,\dx\dt 	
    \bigg]^{\frac{n}{p(n+2)}} \\
  &\le
    c\bigg[
   \mathbf S
    +
    \iint_{(\Omega_\eps)_T\cap Q_\rho}|Du_\eps-Dg_\eps|^p
    \,\dx\dt
    \bigg]^{\frac{n+p}{p(n+2)}}.
\end{align*}
We use this bound in~\eqref{weak-3} and apply Young's inequality,
with the result
\begin{align*}
	 \mathbf S &+
	\iint_{(\Omega_\epsilon)_T\cap Q_\rho} 
	|V_\lambda(Du_\epsilon)|^2\,\dx\dt \\\nonumber
	&\le
          \tfrac12\mathbf S
	 +
	\tfrac12\iint_{(\Omega_\epsilon)_T\cap Q_\rho} 
          |V_\lambda(Du_\epsilon)|^2\,\dx\dt
          +c
           \iint_{\Omega_T\cap Q_\rho}\big(\lambda^p+|Du|^p\big)\dx\dt\\
         & \phantom{\le\,}\nonumber+
         c \iint_{\Omega_T\cap Q_\rho} 
	|V_\lambda(Dg_\eps)|^2 \,\dx\dt
	+
	c
    \big\||b|+|b_\eps|\big\|_{L^{\frac{p(n+2)}{p(n+2)-n}}((\Omega_\eps)_T\cap Q_\rho)}^{\frac{p(n+2)}{p(n+2)-n-p}}.
\end{align*}
In turn we used the elementary estimate
$
	|A|^p \le |V_\lambda (A)|^2 +\lambda^p
$.
The first two terms on the right-hand side can be absorbed into the
left. Finally, we use Lemma~\ref{lem:Dge} to estimate the term involving $Dg_\eps$ by 
\begin{align*}
  \iint_{\Omega_T\cap Q_\rho} |V_\lambda(Dg_\eps)|^2 \,\dx\dt
  \le
    c \iint_{\Omega_T\cap Q_{2\rho}} \big[ \lambda^p+|Du|^p+\rho^{-p}|u|^p\big]\dx\dt.
\end{align*}
Inserting this above, we
arrive at the desired estimate.
\end{proof}

\begin{remark}\label{local-energy} \upshape
The same arguments yield the following local (in time) energy 
estimate
 \begin{align*}
	&  \sup_{\tau\in(t_o-\rho^2,s)}\int_{(\Omega_\epsilon\cap B_\rho)\times\{\tau\}} 
	|u_\epsilon-g_\eps|^2 \,\dx +
	\iint_{(\Omega_\epsilon)_s\cap Q_\rho} 
	\big|V_\lambda(Du_\epsilon)\big|^2\,\dx\dt \\\nonumber
         &\qquad \le c
        \iint_{\Omega_s\cap Q_{2\rho}}\big[\lambda^p+|Du|^p+\rho^{-p}|u|^p\big]\dx\dt
          +
          c
    \big\||b|+|b_\eps|\big\|_{L^{\frac{p(n+2)}{p(n+2)-n}}((\Omega_\eps)_s\cap
           Q_\rho)}^{\frac{p(n+2)}{p(n+2)-n-p}}
\end{align*}
for any $s\in (t_o-\rho^2, t_o]$. \hfill$\Box$
\end{remark} 

\subsection{Proof of the gradient estimate}\label{sec:proof}

We recall \eqref{measure-density-Omega-delta}, \eqref{growth-a-delta}
and \eqref{b-compact-support} and the fact $u_\eps\in
C^3\big((\overline\Omega_\epsilon)_T\cap Q_\rho(z_o)\big)$
(see Appendix~\ref{app:smooth}). Therefore, 
Proposition~\ref{prop:apriori} is applicable with $u,\mathbf{a},b,\Omega,\Theta$ replaced
by $u_\eps,\mathbf{a}_{\eps},b_\eps,\Omega_\eps,\Theta_{\rho/2}(x_o)$. We thus
obtain the gradient estimate
\begin{align}\label{sup-est-w-delta}
	&\sup_{(\Omega_\eps)_T\cap Q_{\rho/2}}
	\big( 1+|Du_\eps|^2\big)^{\frac{1}{2}}\nonumber\\
	&\qquad \le
    C\bigg[\Big(1+ \rho^{n+2}
    \| b_\epsilon\|_{L^\sigma((\Omega_\epsilon)_T\cap Q_\rho)}^{\frac{(n+2)\sigma}{\sigma-n-2}}\Big) 
    \biint_{(\Omega_\epsilon)_T \cap Q_{3\rho/4}}
	\big( 1+|Du_\epsilon|^p\big)
	\,\dx\dt\bigg]^{\frac{d}p}.
\end{align}
In view of~\eqref{measure-density-Omega-delta} and
\eqref{growth-a-delta}, the constant $C$ in the preceding inequality
depends only on $n,N,p,m,M,$ and $\Theta_{\rho/2}(x_o)$, but is independent of $\eps$. 
The energy estimate from Lemma~\ref{lem:energy} implies
\begin{align}\label{energy-bounded}
  \sup_{\tau\in(t_o-\rho^2,t_o)}\int_{(\Omega_\epsilon\cap B_\rho)\times\{\tau\}} 
  |u_\epsilon-g_\eps|^2 \,\dx
  +
  \iint_{(\Omega_\eps)_T\cap Q_\rho}|Du_\eps|^p\dx\dt
  \le
  C,
\end{align}
with a constant $C$ independent of $\epsilon\in(0,1)$; note in particular that
$\|b_\eps\|_{L^\sigma}$ is bounded independently of
$\epsilon$. We combine this with the gradient sup-estimate from
\eqref{sup-est-w-delta} replacing $(\frac\rho2,\frac{3\rho}4)$ by
$(\frac{3\rho}4,\rho)$. This yields the uniform bound 
\begin{align}\label{sup-unif}
  \sup_{(\Omega_\eps)_T\cap Q_{3\rho/4}} |Du_\eps|\le C,
\end{align}
with a constant $C$ independent of $\eps$. 
From the construction of the boundary values $g_\eps$ it is clear that $g_\eps\to u$ in
$L^p(\Omega_T\cap Q_\rho)$ as $\eps\downarrow0$.
Moreover, Lemma~\ref{lem:Dge} ensures that 
\begin{align}\label{boundary-energy-est}
  \iint_{\Omega_T\cap Q_\rho}&\big[|Dg_\eps|^p+|g_\eps|^p\big]\dx\dt 
  \le
  c\iint_{\Omega_T\cap Q_{2\rho}}\big[|Du|^p+\rho^{-p}|u|^p\big]\,\dx\dt,
\end{align}
for every $\eps\in(0,1)$. 
We therefore deduce
\begin{equation}\label{geps-weak-conv}
	\mbox{$g_\eps\wto u$ weakly in $L^p\big(t_o-\rho^2,t_o;W^{1,p}(\Omega\cap B_\rho,\R^N)\big)$ as $\eps\downarrow0$.}
\end{equation}
Moreover, Poincar\'e's inequality and \eqref{boundary-energy-est} yield the bound
\begin{align}\label{ueps-Lp-bound}
  &\iint_{(\Omega_\eps)_T\cap Q_\rho} |u_\eps|^p\,\dx\dt\\\nonumber
  &\qquad\le
  c \iint_{(\Omega_\eps)_T\cap Q_\rho}|u_\eps-g_\eps|^p\,\dx\dt
  +c\iint_{(\Omega_\eps)_T\cap Q_\rho}|g_\eps|^p\,\dx\dt\\\nonumber
  &\qquad\le
  c\rho^p\iint_{(\Omega_\eps)_T\cap Q_\rho}|Du_\eps|^p\,\dx\dt
    +c\iint_{\Omega_T\cap Q_{2\rho}}\big[\rho^p|Du|^p+|u|^p\big]\,\dx\dt.
\end{align}
We extend $u_\eps$ by zero on $Q_\rho\setminus(\Omega_\eps)_T$. Since
$u_\eps=0$ on $(\partial\Omega_\eps)_T\cap Q_\rho$ in the sense of traces,
the extended maps satisfy $u_\eps\in
L^p\big(t_o-\rho^2,t_o;W^{1,p}(B_\rho,\R^N)\big)$ for every
$\eps\in(0,1)$. Moreover, estimates~\eqref{energy-bounded}
and~\eqref{ueps-Lp-bound} imply that the family
$(u_\eps)_{\eps\in(0,1)}$ is bounded in the latter space.
Therefore, we find $\eps_i\downarrow0$ and a limit map $\tilde u\in
L^p(t_o-\rho^2,t_o;W^{1,p}(B_\rho,\R^N))$ such that
\begin{equation}\label{ueps-weak-conv}
	\mbox{$u_{\eps_i}\wto \tilde u$ weakly in $L^p\big(t_o-\rho^2,t_o;W^{1,p}(B_\rho,\R^N)\big)$
	as
    $i\to\infty$.}
\end{equation}
In view of the  uniform $L^\infty{-}L^2$ bound~\eqref{energy-bounded}
we can pass to a non-relabelled subsequence to deduce that
$\tilde u\in L^\infty(t_o-\rho^2,t_o;L^2(B_\rho,\R^N))$ and 
\begin{equation*}
	\mbox{$u_{\eps_i}-g_{\eps_i}\wsto \tilde u-u$
	weakly$^\ast$ in $L^\infty\big(t_o-\rho^2,t_o;L^2(B_\rho,\R^N)\big)$ as
    	$i\to\infty$.}
\end{equation*}
By construction, the maps $u_\eps$ agree with $g_\eps$ on the lateral
boundary in the sense that
\begin{equation*}
u_\eps-g_\eps\in L^p\big(t_o-\rho^2,t_o;W^{1,p}_0(B_\rho,\R^N)\big).
\end{equation*}
Because of the weak convergences~\eqref{geps-weak-conv}
and~\eqref{ueps-weak-conv}, this boundary condition is preserved in
the limit $\eps_i\downarrow0$, from which we deduce 
\begin{equation}\label{boundary-condition-1}
\tilde u-u\in L^p\big(t_o-\rho^2,t_o;W^{1,p}_0(B_\rho,\R^N)\big).
\end{equation}
Now, let  $\epsilon_o>0$ and consider the outer parallel set
$O_{2\eps_o}:=\big\{x\in B_\rho\colon \dist(x,\Omega)<2\eps_o\big\}$.
Since $\Omega_\eps\subset O_{2\eps_o}$ for every $\eps\in(0,\eps_o)$,
we have
\begin{equation*}
	\mbox{$u_\eps-g_\eps=0\;$ a.e.~on $Q_\rho\setminus (O_{2\eps_o})_T$, for every $\eps\in(0,\eps_o)$.}
\end{equation*}
Also this property is preserved in the limit $\eps_i\downarrow0$,
which implies that $\tilde u=u$ a.e.~on $Q_\rho\setminus
(O_{2\eps_o})_T$ for every $\eps_o>0$. In turn, we conclude
\begin{equation}\label{boundary-condition-2}
	\mbox{$\tilde u=u\;$   a.e.~on $Q_\rho\setminus\Omega_T$.  }
\end{equation}
Combining the properties~\eqref{boundary-condition-1}
and~\eqref{boundary-condition-2}, we infer the desired boundary condition
\begin{equation*}
\tilde u\in u + L^p\big(t_o-\rho^2,t_o;W^{1,p}_0(\Omega\cap B_\rho,\R^N)\big)
\end{equation*}
for the limit map $\tilde u$. 
Our next goal is to show that the limit map $\tilde u$ attains the expected initial
values at the initial time $t_o-\rho^2$. To this end, we exploit the
lower semicontinuity of the $L^2$-norm with respect to weak
convergence and the local (in time) energy estimate from Remark \ref{local-energy} to estimate
\begin{align*}
  \tfrac1h\int_{t_o-\rho^2}^{t_o-\rho^2+h} &
  \|\tilde u(t)-u(t)\|^2_{L^2(\Omega\cap B_\rho)}\dt\\
  &\le
  \liminf_{\eps\downarrow0}
  \tfrac1h\int_{t_o-\rho^2}^{t_o-\rho^2+h}
  \|u_\eps(t)-g_\eps(t)\|^2_{L^2(\Omega\cap B_\rho)} \dt\\
  &\le
  c\iint_{(\Omega\cap
    B_{2\rho})\times (t_o-\rho^2 , t_o-\rho^2+h)}\big[\lambda^p+|Du|^p+\rho^{-p}|u|^p\big]\dx\dt\\
  &\phantom{\le\,} +
  c\big\|b\big\|_{L^{\frac{p(n+2)}{p(n+2)-n}}((\Omega\cap B_\rho)\times(t_o-\rho^2,t_o-\rho^2+h))}^{\frac{p(n+2)}{p(n+2)-n-p}},
\end{align*}
for every $h\in(0,\rho^2)$. Since the right-hand side of the last inequality converges to 0 as  $h\downarrow 0$ we infer
$$
	\lim_{h\downarrow 0} \tfrac1h
	\int_{t_o-\rho^2}^{t_o-\rho^2+h}
	\|\tilde u(t)-u(t)\|^2_{L^2(\Omega\cap B_\rho)}\dt =0.
$$
Since $u\in C^0([0,T];L^2(\Omega,\R^N))$
by assumption, this implies that $\tilde u=u$ on 
$(\Omega\cap B_\rho)\times \{t_o-\rho^2\}$ in the usual $L^2$-sense.
At this stage it remains to verify the differential equation for the limit map
$\tilde u$. 
For a fixed compact set $K\Subset \Omega_T\cap Q_{\rho}$, the interior
$C^{1,\alpha}$-estimates from \cite[Chapter IX, Theorem~1.1, Chapter~VIII, Theorems~5.1 and~5.2']{DB} and the uniform energy bound~\eqref{energy-bounded} imply
\begin{equation}\label{C1alpha-bound}
  \|Du_\epsilon\|_{C^{0,\alpha}(K)}\le C
\end{equation}
for every $\eps\in(0,1)$, for some H\"older exponent $\alpha\in(0,1)$ and some constant $C>0$, 
both independent of $\eps$. This allows us to apply Ascoli-Arz\'ela's theorem
to conclude that $Du_\epsilon$ converges uniformly to $D\tilde u$ on
compact subsets of $\Omega_T\cap Q_{\rho}$. In particular, we have
$Du_\epsilon\to D\tilde u$ pointwise in $\Omega_T\cap Q_{\rho}$.
In view of the uniform gradient bound on compact subsets contained in
\eqref{C1alpha-bound}
and the property~\eqref{a-delta-convergence} of the regularized coefficients,
we can use dominated convergence to pass to the limit in the weak
formulation of the system~\eqref{eq-ueps}.
We conclude
that the limit map $\tilde u$ is a weak solution to the system
\eqref{system} on $\Omega_T\cap Q_{\rho}$. Moreover, we know that $\tilde u=u$ on $\partial_{p}(\Omega_T\cap Q_{\rho})$. By uniqueness of solutions this shows that $\tilde u\equiv u$ in $\Omega_T\cap Q_{\rho}$.

Moreover, due
to the sup-bound for the spatial gradient \eqref{sup-unif} we may
apply the dominated convergence theorem to get
\begin{equation}
  \label{ueps-strong-conv}
  \mbox{$Du_{\eps_i}\to D\tilde u=Du\;$ strongly in $L^p(Q_{3\rho/4},\R^N)$ in the limit
    $\eps_i\downarrow0$,} 
\end{equation}
where we extended $u_{\eps_i}$ by zero on
$Q_{3\rho/4}\setminus(\Omega_{\eps_i})_T$.
This strong convergence enables us to 
pass to the limit $\eps_i\downarrow 0$ on the right-hand side of
\eqref{sup-est-w-delta}. Note that the construction of $b_\eps$ ensures the convergence $\|
b_\epsilon\|_{L^\sigma((\Omega_\epsilon)_T\cap Q_\rho)}\to \|
b\|_{L^\sigma(\Omega_T\cap Q_\rho)}$. On
the left-hand side of \eqref{sup-est-w-delta} we may pass to the limit due to the pointwise convergence. In this way, we obtain
\begin{align}\label{sup-est-tilde-u}
	&\sup_{\Omega_T\cap Q_{\rho/2}}
	\big( 1+|D u|^2\big)^{\frac{1}{2}}\nonumber\\
	&\qquad \le
    C\bigg[\Big(1+ \rho^{n+2}
    \| b\|_{L^\sigma(\Omega_T\cap Q_\rho)}^{\frac{(n+2)\sigma}{\sigma-n-2}}\Big) 
    \biint_{\Omega_T \cap Q_{3\rho/4}}
	\big( 1+|D u|^p\big)
	\,\dx\dt\bigg]^{\frac{d}p}.
\end{align}
This yields the asserted
sup-estimate for the gradient of $u$, and  
completes the proof of Theorem~\ref{thm:main}.
\hfill\qed
\appendix

\section{Properties of the regularized coefficients}\label{app:reg}

Here, we provide proofs for the properties of the regularized coefficients $\mathbf a_{\eps}$ stated in Subsection~\ref{section:reg-coeff}. 
The first line in \eqref{growth-a-delta} follows directly from the
definition of~$\mathbf c_\eps$ and the growth
condition~\eqref{bounds-a} for $\mathbf{a}$. The constant $c$ can be
chosen in the form $c(p)\max\{ 1,\mathrm{e}^{p-2}\}$ with the constant $c(p)$ from 
\eqref{bounds-a}. Concerning the ellipticity condition, we observe that
\begin{equation}\label{adelta}
  	\mathbf a_{\eps}'(r)r+\mathbf a_{\eps}(r)
  	=
  	\mathbf c_\eps'(\log r)+\mathbf c_\eps(\log r)
  	=
  	\big(\phi_\eps\ast(\mathbf c'+\mathbf c)\big)(\log r)
	\quad\mbox{for any $r>0$. }
\end{equation}
For the function $\mathbf c'+\mathbf c$ appearing on
the right-hand side, we have in the case $\mu=0$ that 
\begin{align*}
	&\mathbf c'(s)+\mathbf c(s)\\
         &\quad=
         \frac{1}{\eps^2+\mathrm{e}^{2s}}
         \bigg[\mathrm{e}^{2s}
         \Big[
         \sqrt{\eps^2+\mathrm{e}^{2s}}\,\mathbf{a}'\big(\sqrt{\eps^2+\mathrm{e}^{2s}}\big)
         +
         \mathbf{a}\big(\sqrt{\eps^2 +\mathrm{e}^{2s}}\big)
         \Big]
         +
         \eps^2\mathbf{a}\big(\sqrt{\eps^2+\mathrm{e}^{2s}}\big)\bigg]
\end{align*}
for any $s\in\R$.
Using the lower bounds from~\eqref{assumption:a'} and
\eqref{bounds-a}, we deduce 
\begin{align}\label{cdelta}
	\mathbf c'(s)+\mathbf c(s)
         &\ge
         c(p)^{-1}m(\eps^2 +\mathrm{e}^{2s})^{\frac{p-2}2}.
\end{align}
On the other hand, in the case $\mu>0$ we have
\begin{align}\label{cdelta_}
	\mathbf c'(s)+\mathbf c(s)
	&=
	 \mathbf{a}'(\mathrm{e}^{s})\mathrm{e}^{s}
       +\mathbf{a}(\mathrm{e}^{s})
       \ge
       m(\mu^2 +\mathrm{e}^{2s})^{\frac{p-2}2}
\end{align}
for any $s\in\R$.
Similarly as above, we infer from
\eqref{adelta}, \eqref{cdelta}, \eqref{cdelta_} and the definition of $\lambda$ that 
\begin{equation*}
  \mathbf a_{\eps}'(r)r+\mathbf a_{\eps}(r)
  \ge
   c(p)^{-1}m\min \big\{ 1,\mathrm{e}^{-(p-2)}\big\}\big(\lambda^2
   +r^2\big)^\frac{p-2}{2}
   \quad\mbox{for any }r>0.
 \end{equation*}
 This yields the lower bound in~\eqref{growth-a-delta}$_2$. 
 Similarly, by applying the upper bound from \eqref{assumption:a'} (taking also into account the fact
 that $\mathbf a$ is non-negative, cf.~\eqref{bounds-a}), we obtain 
 \begin{align*}
   (\phi_\eps\ast\mathbf{c}')(s)
   \le 
   c(p) M\max\{1,\mathrm{e}^{p-2}\} \big(\lambda^2+\mathrm
    e^{2s}\big)^{\frac{p-2}{2}}
    \quad\mbox{for any }s\in\R. 
  \end{align*}
  From this we deduce  
 \begin{align*}
   \mathbf{a}_{\eps}'(r)r
   &=
   (\phi_\eps\ast\mathbf{c}')(\log r)
   \le
   c(p)M\max\{1,\mathrm{e}^{p-2}\} \big(\lambda^2+r^2\big)^{\frac{p-2}{2}}
   \quad\mbox{for any }r>0,
 \end{align*}
 which implies the asserted upper bound in~\eqref{growth-a-delta}$_2$. 
 At this stage it remains to derive the estimate for the second derivative $\mathbf a_\eps''$. To this
 end, we compute
 \begin{equation*}
   \mathbf a_{\eps}''(r)r^2
   =
   \mathbf{c}_\eps''(\log r)-\mathbf{c}_\eps'(\log r)
   =
   \big((\phi_\eps'-\phi_\eps)\ast \mathbf c'\big)(\log r).
 \end{equation*}
 Then we use~\eqref{assumption:a'} and \eqref{bounds-a} to derive in the case $\mu=0$ the bound
\begin{align*}
 	|\mathbf c'(s)|
	&=
	\bigg|\mathbf{a}'\big(\sqrt{\eps^2 +\mathrm{e}^{2s}}\big)\frac{\mathrm{e}^{2s}}{\sqrt{\eps^2 +\mathrm{e}^{2s}}}\bigg|\\\nonumber
	&\le
        \Big|\sqrt{\eps^2+\mathrm{e}^{2s}}\,\mathbf{a}'\big(\sqrt{\eps^2
          +\mathrm{e}^{2s}}\big)+\mathbf{a}\big(\sqrt{\eps^2+\mathrm{e}^{2s}}\big)\Big|
        +\Big|\mathbf{a}\big(\sqrt{\eps^2
          +\mathrm{e}^{2s}}\big)\Big|\\
        &\le
	\big( 1+c(p)\big)M \big(\eps^2 +\mathrm e^{2s}\big)^\frac{p-2}{2}
	\le
	2c(p)M \big(\eps^2 +\mathrm e^{2s}\big)^\frac{p-2}{2},
 \end{align*}
while in the case $\mu>0$ we obtain
\begin{align*}
 	|\mathbf c'(s)|
	&=
	\big|\mathbf{a}'(\mathrm{e}^{s})\mathrm{e}^{s}\big|
	\le
    \big|\mathrm{e}^{s}\,\mathbf{a}'(\mathrm{e}^{s}) +
    \mathbf{a}(\mathrm{e}^{s})\big| +
    \big|\mathbf{a}(\mathrm{e}^{s})\big|\\
        &\le
	\big( 1+c(p)\big)M \big(\mu^2 +\mathrm e^{2s}\big)^\frac{p-2}{2}
	\le
	2c(p)M \big(\mu^2 +\mathrm e^{2s}\big)^\frac{p-2}{2}.
 \end{align*}
Hence, in both cases we have
\begin{align}\label{cprime-bound}
 	|\mathbf c'(s)|
	\le
	2c(p)M \big(\lambda^2 +\mathrm e^{2s}\big)^\frac{p-2}{2}.
\end{align}
From this we deduce, similarly as above, that
\begin{align*}
  |\mathbf a_{\eps}''(r)|r^2
  &=
  |((\phi_\eps'-\phi_\eps)\ast\mathbf c')(s)|\\
  &\le
  2c(p)M\max\{1,\mathrm{e}^{p-2}\} \big(\lambda^2+\mathrm
  e^{2s}\big)^{\frac{p-2}{2}}\int_\R|\phi_\eps'-\phi_\eps|\,\d s\\
  &\le 
  2c(p)M
  \Big( \tfrac{2}{\eps}\|\phi'\|_{L^\infty}+1\Big)
\max\big\{1,\mathrm{e}^{p-2}\big\} \big(\lambda^2+\mathrm
  e^{2s}\big)^{\frac{p-2}{2}}.
\end{align*}
The proof of the claim~\eqref{growth-a-delta} is thus complete.
Finally, we analyze the convergence of $\mathbf a_{\eps} (r)$ in the
limit $\eps\downarrow 0$ and thereby prove
\eqref{a-delta-convergence}.
For any $s\in\R$ we estimate
\begin{align*}
	\big| \mathbf c_\eps (s)- \mathbf c (s)\big|
	&\le 
	\sup_{s-\eps<r<s+\eps} 
	\,\big|\mathbf{c}(r)-\mathbf{c}(s)\big|
	=\sup_{s-\eps<r<s+\eps} 
	\bigg| \int_s^r \mathbf{c}'(\tau)\,\dtau\bigg|
	\\[1ex]
	&\le
	2c(p)M
	\sup_{s-\eps<r<s+\eps}\bigg| \int_s^r \big(\eps^2 +\mathrm e^{2\tau}\big)^\frac{p-2}{2}\dtau\bigg|
	\\
	&\le
	2c(p)M \eps\max\big\{ 1,\mathrm{e}^{p-2}\big\}
        \big(\eps^2 + \mathrm e^{2s}\big)^{\frac{p-2}{2}},
\end{align*}
where the second-to-last step follows from~\eqref{cprime-bound}.
This gives the desired estimate
\eqref{a-delta-convergence}.


\section{Regularity up to the boundary}\label{app:smooth}

Here, we show that solutions to the regularized
problem \eqref{eq-ueps} are smooth up to the boundary as claimed at the end of Section~\ref{sec:regularization}.
To this end, we follow the strategy
of  Banerjee \& Lewis \cite[Appendix.~Proof of
(2.7)]{BanerjeeLewis:2014} to flatten the boundary and then to reduce
the problem of boundary regularity to the interior case by a
reflection argument. 

\subsection{Schauder estimates for linear parabolic systems}
In this section, we explain Schauder estimates for linear parabolic systems of the type
\begin{equation}\label{A:linear-system}
	\partial_t u^i -\sum_{\al,\be=1}^n\sum_{j=1}^N\big[A^{ij}_{\alpha\beta}u^j_{x_\beta}\big]_{x_\alpha}
	=
	\sum_{\al=1}^n\sum_{j=1}^Nb^{ij}_{\alpha}u^j_{x_\alpha}+\sum_{\al=1}^n(f_\alpha^i)_{x_\alpha}+c^i
	\qquad\mbox{in $\Om_T$,}
\end{equation}
for $i=1,2,\dots, N$, 
where the coefficients $A^{ij}_{\alpha\beta}\colon\Om_T\to\R$ satisfy for some $0<\nu\le L$ the ellipticity and boundedness condition
\begin{equation}\label{parabolicity}
	\nu |\xi|^2\le \sum_{\al,\be=1}^n\sum_{i,j=1}^N A^{ij}_{\alpha\be}\xi^i_\alpha\xi^j_\be\le L|\xi|^2\quad\mbox{for all $\xi\in\rr^{Nn}$.}
\end{equation}
We will assume that the functions $c^i\colon\Omega_T\to\R$ belong to a parabolic Campanato-Morrey space, which is defined as follows. 
\begin{definition}\label{def-Morrey}
With $q \geq 1$, $\theta \ge 0$, a measurable map $w\colon \Omega_T\to \R^k$, $k\ge 1$, belongs
to the (parabolic) Morrey space $L^{q,\theta}(\Omega_T,\R^k)$ if and only if
\begin{equation*}
	\|w\|_{L^{q,\theta}(\Omega_T,\R^k)}^q
	:=
	\sup_{z_o\in\Omega_T,\, 0<\rho<\diam(\Omega_T)}
	\varrho^{-\theta} \iint_{\Omega_T\cap Q_\rho(z_o)} |w|^q \,\dx\dt
	<
	\infty.
\end{equation*}
\end{definition}

By $C^{0,\mu}$ we mean H\"older continuity with respect to the parabolic metric, i.e.~with H\"older exponent $\mu$ in space and $\frac{\mu}{2}$ in time.  With these notions at hand we state the following parabolic Schauder estimates, which can be proved by standard comparison and freezing techniques, cf. \cite{Campanato, Misawa, Schlag}.

\begin{theorem}\label{Thm:Schauder}
Suppose $A^{ij}_{\alpha\be}$ and $f_\alpha^i$ are in $C^{0,\mu}(\Om_T)$ for some $\mu\in(0,1)$,
whereas $b^{ij}_{\alpha}\in L^{\infty}(\Om_T)$ and $c^i\in L^{2,\theta}(\Om_T)$ for $\theta:=n+2\mu$. 
Let $u$ be a weak solution to \eqref{A:linear-system} under the assumption \eqref{parabolicity}.
Then $Du\in C^{0,\mu}_{\loc}(\Om_T)$ and moreover for any compact set $K\Subset\Om_T$ we have
\begin{equation*}
	[Du]_{\mu,K}\le C\big[\|Dv\|_{L^2(\Om_T)}+M\big],
\end {equation*}
where $C$ depends on $n$, $\nu$, $L$ and $\dist(K,\Om_T)$,
and $M$ depends on the norms of $A^{ij}_{\alpha\be}$, $b^{ij}_{\alpha}$, $f_\be^i$, and $c^i$
in their corresponding spaces.
\end{theorem}

\subsection{Flattening of the boundary}
Before we start with the actual construction of local boundary coordinates, we introduce a few abbreviations. 
By $B_\delta^{(n-1)}\equiv B_\delta^{(n-1)}(0)$ we denote the ball of radius $\delta>0$ centered at the origin in $\R^{n-1}$. 
Then, for $\eta>0$ we define $\bC_{\delta,\eta}:=B_\delta^{(n-1)}\times (-\eta,\eta)$, and similarly $\bC_{\delta,\eta}^{+}:= B_\delta^{(n-1)}\times (0,\eta)$
and $\bC_{\delta,\eta}^{-}:= B_\delta^{(n-1)}\times (-\eta,0)$. Cylinders in $\R^{n+1}$ of height $T>0$
with base $\bC_{\delta,\eta},\,\bC_{\delta,\eta}^\pm $ are denoted by $\bQ_{\delta,\eta}$, $\bQ_{\delta,\eta}^\pm$, so that 
$ \bQ_{\delta,\eta}:= \bC_{\delta,\eta}\times (0,T)$.

Since $\partial\Omega_\eps$ is a smooth closed $(n-1)$-dimensional
submanifold of $\R^n$, it can locally be written as graph of a smooth function $\phi\in C^\infty(B_\delta^{n-1})$ after a suitable rigid motion.
More precisely, for any point
$y_o\in \partial \Omega_\eps\cap B_\rho(x_o)$,
there is a neighboorhood $N_o$ of $y_o$ so that
$\Omega_\eps\cap N_o=\Phi(\bC_{\delta,\eta}^-)$ with the
parametrization 
$\Phi\colon \bC_{\delta,\eta}\to N_o\subset\R^n$
defined by
\begin{equation}\label{def:Phi}
  \Phi(y',y_n):= \big(y',\phi(y')\big)+ \nu\big( y',\phi(y')\big) y_n,\quad \mbox{for $y'\in B_\delta ^{n-1}$ and $y_n\in(-\eta,\eta)$,}
\end{equation}
where $\nu\colon \partial\Omega_\eps\to\R^n$ denotes the outward unit normal on
$\partial\Omega_\eps$. By another rigid motion we can achieve that $y_o=0$ and
$\nu(0)=e_n$. 
The inverse mapping 
$
	\Psi:=\Phi^{-1}\colon N_o
	\to \bC_{\delta,\eta}
$
is given by
\begin{equation}\label{eq:Psi}
	\Psi (x)= \Big(x_1-\boldsymbol d_{x_1}(x)\boldsymbol d(x),\dots,
	x_{n-1}-\boldsymbol d_{x_{n-1}}(x)\boldsymbol d(x),
        \boldsymbol d(x)\Big)
        \qquad\mbox{for }x\in N_o,
\end{equation}
where $\boldsymbol d$ denotes the signed distance to
$\partial\Omega_\eps$.       
A straightforward computation yields 
\begin{equation}\label{DPsi(0)}
  D\Psi(0)=\mathrm{id}_{\R^{n}},
\end{equation}
and
\begin{equation}\label{def:Q}
	\mathbf Q(x)
	:=
	D\Psi(x)^t\cdot D\Psi(x)
	 =
	\left[
	\begin{matrix}
	\big( D\Psi_\alpha(x)\cdot D\Psi_\beta(x)\big)_{1\le \alpha,\beta\le n-1}& 0 \\[5pt]
	0 & 1  \\
	\end{matrix}
	\right].
\end{equation}
For a more detailed derivation of these properties, we refer to
\cite[Section~5.1]{BDMS}.
In what follows, we use the short-hand notations
\begin{equation}\label{def:Q2}
  \mathbf{Q}_x(\xi,\zeta):=\sum_{i=1}^N\sum_{\alpha,\beta=1}^n
  \mathbf{Q}_{\alpha\beta}(x)\xi_\alpha^i\zeta_\beta^i
  \qquad\mbox{and}\qquad
  |\xi|_{\mathbf{Q}_x}
  :=\sqrt{ \mathbf Q(x)(\xi,\xi)},
\end{equation}
for matrices $\xi,\zeta\in\R^{Nn}$.
Now we define
\begin{equation}
	\w(y,t):=u_\eps \big(\Phi (y),t\big)
	\quad\iff\quad 
	u_\eps(x)=\w\big(\Psi (x),t\big),
\end{equation}
for $y\in \bC_{\delta,\eta}^-$ and $t\in[0,T]$, and analogously
\begin{equation*}
  \widehat\varphi(y,t):=\varphi\big(\Phi (y),t\big)
  \quad\iff\quad \varphi(x,t)=\widehat\varphi\big(\Psi (x),t\big)
\end{equation*}
for any $\varphi\in L^p(0,T;W^{1,p}_0(\Omega_\eps,\R^N))$.
Then, 
$\w\in L^p(0,T;W^{1,p}( \bC_{\delta,\eta}^-))$ and
$\w=0$  in the sense of traces on $B_\delta^{(n-1)}\times \{0\}\times(0,T)$.
For the
derivatives in spatial directions, we have 
\begin{align*}
	Du_\eps(x,t)\cdot D\varphi(x,t)
	&=
	\mathbf Q_x\big( D\w \big(\Psi (x),t\big),D\widehat \varphi \big(\Psi (x),t\big)\big).
\end{align*}
Moreover, for a.e.~$x\in N_o$ and $t\in(0,T)$ we have
\begin{align*}
  u_\eps(x,t)\cdot \partial_t\varphi(x,t)
  =
  \w(\Psi(x),t)\cdot \partial_t\widehat\varphi(\Psi(x),t).
\end{align*}
Using the two preceding formulae and applying the transformation
$x=\Phi(y)$ on a fixed time slice, we infer 
\begin{align*}\nonumber
  &\int_{\Phi (\bC_{\delta,\eta}^- )\times\{t\}}\big[u_\eps\cdot\partial_t\varphi
  -
  \mathbf a_\eps\big( |Du_\eps|\big) Du_\eps\cdot D\varphi\big] \,\dx\\\nonumber
  &\qquad=
   \int_{\bC_{\delta,\eta}^-\times\{t\}}
  \big[\w\cdot\partial_t\widehat\varphi
  -
  \mathbf a_\eps\big( |D\w|_{\mathbf{Q}_\Phi}\big)
	\mathbf Q_\Phi\big( D\w ,D\widehat \varphi \big)\big] J_n\Phi\, \dy,
\end{align*}
for a.e.~$t\in(0,T)$, where $J_n\Phi:=|\det D\Phi|$ denotes the Jacobian of $\Phi$. Integrating this identity with respect to  $t\in(0,T)$,
we obtain the left-hand side of~\eqref{eq-ueps}. 
Diminishing $\eta>0$ if necessary, we can achieve that the right
hand side $b_\eps$ in \eqref{eq-ueps} vanishes in a tubular neighborhood of $\partial\Omega_\eps\times (0,T)$ by construction, cf.~\eqref{b-compact-support}.
Consequently,~\eqref{eq-ueps} turns into
\begin{equation}\label{equation:u-hat}
  \iint_{\bQ_{\delta,\eta}^-}\Big[\w\cdot\partial_t\widehat\varphi
  -
  \mathbf a_\eps\big( |D\w|_{\mathbf{Q}_\Phi}\big)
  \mathbf Q_\Phi\big( D\w ,D\widehat \varphi \big)\Big] J_n\Phi\, \dy\dt
  =
  0.
\end{equation}
In this equation, the testing function $\widehat\varphi$ can be chosen
as an arbitrary smooth function with compact support in
$\bQ_{\delta,\eta}^-$. By an approximation argument, we
can also verify it for every $\widehat\varphi\in L^p\big(0,T;W^{1,p}_0(
\bC_{\delta,\eta}^-)\big)$ with $\widehat\varphi_t\in
L^2(\bQ_{\delta,\mu}^-)$ and $\widehat\varphi(0)=0=\widehat\varphi(T)$.
Next, for an arbitrary testing function $\psi\in L^p\big(0,T;W^{1,p}_0(
\bC_{\delta,\eta}^-)\big)$ with $\psi_t\in L^2(\bQ_{\delta,\eta}^-)$ and $\psi(0)=0=\psi(T)$, we test  \eqref{equation:u-hat}
with
$\widehat\varphi := (J_n\Phi)^{-1}\psi$, which is admissible
since $J_n\Phi$ is a positive Lipschitz function. This leads to 
\begin{align}\label{transformed-Euler}\nonumber
  &\iint_{\bQ_{\delta,\eta}^-}\Big[\w\cdot\partial_t\psi
  -
  \mathbf a_\eps\big( |D\w|_{\mathbf{Q}_\Phi}\big)
  \mathbf Q_\Phi\big( D\w ,D\psi \big) \Big]\, \dy\dt\\
  &\qquad=
  -\iint_{\bQ_{\delta,\eta}^-}\mathbf a_\eps\big( |D\w|_{\mathbf{Q}_\Phi}\big)
  \mathbf Q_\Phi\big(D\w ,\psi\otimes D[J_n\Phi]\big)(J_n\Phi)^{-1}\dy\dt,
\end{align}
for every 
$\psi\in L^p\big(0,T;W^{1,p}_0(\bC_{\delta,\eta}^-)\big)$ with $\psi_t\in
L^2(\bQ_{\delta,\eta}^-)$ and $\psi(0)=0=\psi(T)$.

\subsection{Reflection and reduction to the interior}
Next, we extend $\mathbf Q_\Phi$ and $J_n\Phi$ to
$\bC_{\delta,\eta}^+$ by an even reflection across
$\Gamma_\delta:=B_\delta^{(n-1)}\times\{0\}$. To this aim we define
$$
	\mbox{$\mathbf Q_{\Phi (y',y_n)}:= \mathbf Q_{\Phi (y',-y_n)}\,$
	and
	$J_n\Phi (y',y_n):=J_n\Phi (y',-y_n)\,$ 
	for any $(y',y_n)\in \bC_{\delta,\eta}^+$.}
$$
Note that the functions $\mathbf{Q}_\Phi$ and $J_n\Phi$
are smooth on $B_\delta^{(n-1)}\times (-\eta,0]$, and therefore their
extensions are also smooth on $\Gamma_\delta$.
However, the extensions are in general only Lipschitz continuous on
$\bC_{\delta,\eta}$.
Only the horizontal derivatives
of the extended Jacobian are continuous across $\Gamma_\delta$,
since they are even functions as the Jacobian itself.
Next, we extend the solution $\w$ by an odd reflection
across the boundary $\Gamma_\delta$ on each time-slice. More
precisely, we let
$$
	\mbox{$\w (y',y_n,t):=-\w (y',-y_n,t)\, $ for 
$(y',y_n)\in \bC_{\delta,\eta}^+$.}
$$
Now we consider testing functions $\psi\in
L^p\big(0,T;W^{1,p}_0(\bC_{\delta,\eta})\big)$ with $\partial_t\psi\in
L^2(\bQ_{\delta,\eta})$ and $\psi(0)=0=\psi(T)$. We decompose $\psi=\psi_{\rm e}+\psi_{\rm o}$
into its even part $\psi_{\rm e}$ and odd part $\psi_{\rm o}$
with respect to reflection across
$\Gamma_\delta$. 
According to  $\bQ_{\delta,\eta}= \bQ_{\delta,\eta}^+\cup \bQ_{\delta,\eta}^-$ we write
\begin{align*}
	\mathbf I:=
  	\iint_{\bQ_{\delta,\eta}} 
	\Big[\w\cdot\partial_t\psi - 
	\mathbf a_\eps\big( |D\w|_{\mathbf Q_\Phi}\big) 
	\mathbf Q_\Phi\big( D\w ,D\psi \big) \Big]\, \dy\dt
	=
	\mathbf I_{\rm e}^+ + \mathbf I_{\rm e}^- +
	\mathbf I_{\rm o}^+ + \mathbf I_{\rm o}^-.
\end{align*}
The right-hand side integrals are defined as follows: For any sign $\{ +,-\}$ and any symmetry type $\{{\rm e},{\rm o}\}$
one has to replace $\bQ_{\delta,\eta}, \psi$ in $\mathbf I$ by the corresponding half cylinder $\{ \bQ_{\delta,\eta}^+, \bQ_{\delta,\eta}^-\}$
and the corresponding even or odd part $\{\psi_{\rm e},\psi_{\rm o}\}$ of $\psi$. 
In the last two terms, we observe that 
$\mathbf Q_\Phi( D\w ,D\psi_{\rm o})$ is an even function with respect to $y_n$ 
because the derivatives of $\w$ and $\psi_{\rm o}$ in direction of $y_i$ with $i\in\{1,\dots ,n-1\}$ are
odd and the derivatives in the direction
of $y_n$ are even. Furthermore, the structure of $\mathbf Q$ from \eqref{def:Q}
does not lead to mixed terms with both types of derivatives. 
For the same reason, $|D\w|_{\mathbf Q_\Phi}$ is an even function, and
by definition we have that 
$\w\cdot\partial_t\psi_{\mathrm{o}}$ is even as well. Consequently,
the integrands of the last two integrals are even, which implies
 $\mathbf I_{\rm o}^-=\mathbf I_{\rm o}^+$.
Similarly, using the facts that $\w$ is odd and $\psi_{\rm e}$ is
even, we deduce that $\mathbf I_{\rm e}^-=-\mathbf
I_{\rm e}^-$.
Therefore, we obtain 
\begin{align}\label{Euler-LHS}
	\iint_{\bQ_{\delta,\eta}}&\Big[\w\cdot\partial_t\psi-
        \mathbf a_\eps\big(|D\w|_{ \mathbf Q_\Phi}\big)
	\mathbf Q_\Phi\big( D\w ,D\psi \big) \Big] \dy\dt
        \\
        \nonumber
	&
	=2\iint_{\bQ_{\delta,\eta}^-}\Big(\w\cdot\partial_t\psi_o-
        \mathbf a_\eps\big(|D\w|_{ \mathbf Q_\Phi}\big)
	\mathbf Q_\Phi\big( D\w ,D\psi_{\rm o} \big)
        \Big) \dy\dt.
\end{align}
Note that the right-hand side coincides with the left-hand side
of~\eqref{transformed-Euler} with $\psi_{\rm o}$ in place of $\psi$.
Analogously to the decomposition of $\mathbf I$, we write
\begin{align*}
  \iint_{\bQ_{\delta,\eta}}&
  \mathbf a_\eps\big( |D\w|_{\mathbf Q_\Phi}\big)
	\mathbf Q_\Phi\big( D\w ,\psi\otimes
        D[J_n\Phi]\big)(J_n\Phi)^{-1}\dy\dt
       =\mathbf{II}_{\rm e}^+ + \mathbf{II}_{\rm e}^- +
       \mathbf{II}_{\rm o}^+ + \mathbf{II}_{\rm o}^-.
\end{align*}
For these integrals, we can use the similar symmetry considerations as
above. Since $\psi_{\mathrm{o}}\otimes  D[J_n\Phi]$ enjoys the same
symmetry properties as $D\psi_{\mathrm{o}}$, we infer
$\mathbf{II}_{\mathbf{o}}^+=\mathbf{II}_{\mathrm{o}}^-$. Similarly, we
deduce $\mathbf{II}_{\mathrm{e}}^+=-\mathbf{II}_{\mathrm{e}}^-$.  This
implies
\begin{align}\label{Euler-RHS}\nonumber
  \iint_{\bQ_{\delta,\eta}}&
  \mathbf a_\eps\big( |D\w|_{\mathbf Q_\Phi}\big)
	\mathbf Q_\Phi\big( D\w ,\psi\otimes
        D[J_n\Phi]\big)(J_n\Phi)^{-1}\dy\dt  \\
  &=
        2\iint_{\bQ_{\delta,\eta}^-}
  \mathbf a_\eps\big( |D\w|_{\mathbf Q_\Phi}\big)
	\mathbf Q_\Phi\big( D\w ,\psi_{\mathrm{o}}\otimes
        D[J_n\Phi]\big)(J_n\Phi)^{-1}\dy\dt.
\end{align}
Note  that $\psi_{\rm o}=0$  on
$\Gamma_\delta$, which makes $\psi_{\rm o}$
admissible in the transformed parabolic system
\eqref{transformed-Euler}. This means that the right-hand sides
of~\eqref{Euler-LHS} and \eqref{Euler-RHS} coincide. Thus, we conclude
that the extended map $\w$ satisfies 
\begin{align}\label{Euler-extension}\nonumber
  \iint_{\bQ_{\delta,\eta}}&
    \Big[\w\cdot\partial_t\psi-\mathbf a_\eps\big(|D\w|_{ \mathbf Q_\Phi}\big)
	\mathbf Q_\Phi\big( D\w ,D\psi \big) \Big]\dy\dt\\
  &=
  \iint_{\bQ_{\delta,\eta}}\mathbf a_\eps\big( |D\w|_{\mathbf Q_\Phi}\big)
   \mathbf Q_\Phi\big( D\w ,\psi\otimes D[J_n\Phi]\big)(J_n\Phi)^{-1}\dy\dt,
 \end{align}
 for every $\psi\in L^p\big(0,T;W^{1,p}_0(\bC_{\delta,\eta})\big)$
 with $\partial_t\psi\in L^2(\bQ_{\delta,\eta})$ and $\psi(0)=0=\psi(T)$.
Dropping the $\Phi$ on $\mathbf{Q}$ for ease of notation,
\eqref{Euler-extension} is the weak form of  the  parabolic system
\begin{equation}\label{Euler-pointwise}
\begin{aligned}
\pl_t \w^i - &\sum_{\al,\be=1}^n\big[\mathbf a_\eps(|D\w|_{{\bf Q}}){\bf Q}_{\al\be}\w^i_{y_\al}\big]_{y_\be}\\
&=\sum_{\al,\be=1}^n\mathbf a_\eps(|D\w|_{\bf Q}){\bf Q}_{\al\be}\w^i_{y_\al} \frac{[J_n\Phi]_{y_\be}}{(J_n\Phi)}
\quad\mbox{ in $\mathcal{Q}_{\dl,\eta}$.}
\end{aligned}
\end{equation}

 \subsection{Smoothness of $u_\eps$ up to the lateral boundary}
%
%
%

We first observe that $\mathbf Q_{\Phi(0)}(\xi,\xi)=|\xi|^2$, since $D\Psi (0)={\rm id}_{\rr^n}$ by~\eqref{DPsi(0)}. By shrinking $\delta$ and $\eta$ if necessary, we can  achieve 
\begin{equation}\label{bounds-norm}
	\mbox{$\tfrac12|\xi|\le  |\xi|_{\mathbf{Q}_{\Phi(y)}}\le 2|\xi|$ for any $\xi\in\rr^{Nn}$
	  and $y\in \bC_{\delta,\eta}$, and 
	 $\mathrm{Lip}\Big(\mathbf{Q}_\Phi\big|_{\bC_{\delta,\eta}}\Big)\le\Lambda$}
\end{equation}
for some universal constant $\Lambda<\infty$. 
This implies that 
assumptions (1.7) -- (1.9) from  \cite{Tolksdorf:1983} are
fulfilled if we replace the functions $b$, $q(\xi)$ used by Tolksdorf 
by the functions $\mathbf{Q}$, $|\xi|_{\mathbf{Q}_x}^2$ defined in
\eqref{def:Q2}. Similarly, we have
\begin{equation}\label{bounds-Jacobian}
  	\tfrac12\le J_n\Phi(y)\le 2
	\mbox{ for  any $y\in \bC_{\delta,\eta}$,} 
	\quad \mbox{and}\quad 
	\mathrm{Lip}\Big( J_n\Phi \big|_{\bC_{\delta,\eta}}\Big)\le\Lambda
\end{equation}
for some universal constant $\Lambda<\infty$. 
Furthermore, the estimates \eqref{growth-a-delta} for the coefficients
$\mathbf a_\eps (t)$ imply that assumptions (1.10)--(1.12) from
Tolksdorf \cite{Tolksdorf:1983} hold true. For the inhomogeneous term, we observe that
\begin{equation*}
	a^i(x,\xi)=\sum_{\al,\be=1}^n\mathbf{a}_\eps(|\xi|_{\mathbf{Q}})\mathbf{Q}_{\al\be}\xi^i_{\al}\frac{[J_n\Phi]_{y_\be}}{J_n\Phi}.
\end{equation*}
Again, by \eqref{bounds-norm} and \eqref{bounds-Jacobian}, we will find the desired positive constant
in order to verify (1.13) from  \cite{Tolksdorf:1983}. Having arrived at this stage we can apply the $C^{1,\alpha}$-regularity results from
\cite{DB}. Indeed, as pointed out by DiBenedetto in the monograph \cite[Chapter~VIII.7]{DB}, the statement of \cite[Chapter~IX, Theorem~1.1]{DB} continues to hold under these assumptions. The application of the theorem yields $D_y\w\in C^{0,\al}_{\loc}(\mathcal{Q}_{\dl,\eta})$.
Hence $u_\varep$ enjoys the same degree of regularity in the vicinity of $(\pl\Om_\varep\cap B_\rho(x_o))\times(0,T)$.
A further application of the interior regularity from
\cite[Chapter~IX]{DB} directly to $u_\varep$ gives 
$Du_\varep\in C^{0,\al}(\overline{\Om_\varep \cap B_\rho(x_o)}\times [\tau,T])$ for some $\tau>0$.

Up to now, all the above regularity results also hold for the
degenerate or singular case, 
and solutions cannot be expected to be more regular in this case.
However, since the regularized problem is non-degenerate, we can show
higher regularity of solutions. We begin by noting that a standard
application of the difference quotient technique yields the weak
differentiability of $V_\lambda(D\w)=(\lambda^2+|D\w|^2)^{\frac{p-2}{4}}D\w$
with $D[V_\lambda(D\w)]\in L^2_{\loc}(\mathcal{Q}_{\dl,\eta},\R^{Nn})$;
see for instance \cite[Lemma 5.1]{DMS} and
\cite[Thm. 1.1]{Scheven:2010} for the cases $p\ge2$ and
$\frac{2n}{n+2}<p<2$, respectively.
By using the fact $\lambda>0$ in the case $p>2$ and the local
boundedness of $|D\w|$ if $p<2$, we deduce that the second spatial
derivatives of the solution satisfies
$D^2_y\w\in L^2_{\loc}(\mathcal{Q}_{\dl,\eta},\R^{Nn})$.

Having second spatial derivates in $L^2_{\loc}$ and first spatial derivatives locally bounded, we are allowed to perform an integration by parts in \eqref{Euler-extension} in the diffusion term. After that we shift all terms except the one containing the time derivative to the right-hand side. In this way we obtain an estimate of the form 
\begin{equation}\label{time-deriv}
  	\bigg|\iint_{\bQ_{\delta,\eta}} \w\cdot\partial_t\psi \,\dy\dt\bigg|
    \le 
    C \|\psi\|_{L^2(\bQ_{\delta,\eta})}
\end{equation}
for any $\psi\in C_0^\infty (\bQ_{\delta,\eta},\R^N)$. This implies that $\partial_t\w \in L^2_{\loc}(\mathcal{Q}_{\dl,\eta})$.

The main ingredient for the higher regularity are the Schauder
estimates for
linear parabolic systems stated in Theorem~\ref{Thm:Schauder}. 
We begin by differentiating \eqref{Euler-extension} in tangential
directions, i.e.~with respect to $y_\ell$ for $\ell=1,2,\dots, n-1$.
As before we omit the $\Phi$ on $\mathbf{Q}$.
Since $D^2_y \w\in L^2_{\loc}(\mathcal{Q}_{\dl,\eta})$, we infer that
$v:=\w_{y_\ell}$ is a weak solution to the following parabolic system:
\begin{equation}\label{Schauder-system}
	\partial_t v^i -
	\sum_{\al,\be=1}^n\sum_{j=1}^N\big[A^{ij}_{\alpha\be}v^j_{y_\alpha}\big]_{y_\be}
	=
	\sum_{\al=1}^n\sum_{j=1}^Nb^{ij}_{\alpha}v^j_{y_\alpha} +
	\sum_{\al=1}^n(f^i_\alpha)_{y_\alpha} + c^i
\end{equation}
in $\mathcal{Q}_{\dl,\eta}$ and for $i=1,\dots, N$, where the coefficients are given by
\begin{align}\label{def_A}
	A^{ij}_{\alpha\be}
	&:=
	\mathbf{a}_\eps(|D\w|_{\mathbf{Q}})\mathbf{Q}_{\alpha\be}\delta^{ij} +
	\frac{\mathbf{a}_\eps^\prime(|D\w|_{\mathbf{Q}})}{|D\w|_{\mathbf{Q}}}
	\sum_{\gamma,\delta=1}^n
	\mathbf{Q}_{\alpha\gm} \mathbf{Q}_{\be\delta} \w^i_{y_\gm}\w^j_{y_\delta} , 
\end{align}
and
\begin{align}\label{def_b}
	b^{ij}_{\alpha}
	&:=
	\sum_{\beta=1}^n 
	\frac{[J_n\Phi]_{y_\be}}{J_n\Phi}
	\bigg[\mathbf{a}_\eps(|D\w|_{\mathbf{Q}})\mathbf{Q}_{\alpha\be}\dl^{ij} +
	\sum_{\gamma, \delta=1}^n 
	\frac{\mathbf{a}_\eps^\prime(|D\w|_{\mathbf{Q}})}{|D\w|_{\mathbf{Q}}}
	\mathbf{Q}_{\alpha\delta} \mathbf{Q}_{\be\gamma}
	\w^i_{y_\gamma}\w^j_{y_\delta}\bigg].
\end{align}
The inhomogeneities are defined by 
\begin{align*}
	f_i^\alpha
	&:=
	\sum_{\beta,\gamma,\delta=1}^n \sum_{k=1}^N 
	\frac{\mathbf{a}_\eps^{\prime}(|D\w|_{\mathbf{Q}})}{2|D\w|_{\mathbf{Q}}}
	[\mathbf{Q}_{\gm\dl}]_{y_\ell}\mathbf{Q}_{\al\be}
	\w^i_{y_\beta} \w^k_{y_\gm}\w^k_{y_\dl} +
	\sum_{\beta=1}^n 
	\mathbf{a}_\eps(|D\w|_{\mathbf{Q}})
	[\mathbf{Q}_{\al\be}]_{y_\ell}\w^i_{y_\beta}
\end{align*}
and
\begin{align*}
	c^i
	&:=
	\sum_{\alpha, \beta=1}^n\mathbf{a}_\eps(|D\w|_{\mathbf{Q}})
	\bigg[
	\frac{\mathbf{Q}_{\al\be}[J_n\Phi]_{y_\be}}{J_n\Phi}\bigg]_{y_\ell}
	\w^i_{y_\al} \\
	&\quad\ \, +
	\frac{[J_n\Phi]_{y_\be}}{J_n\Phi}
	\frac{\mathbf{a}_\eps^\prime(|D\w|_{\mathbf{Q}})}{2|D\w|_{\mathbf{Q}}}
	\sum_{\alpha, \beta, \gamma, \delta=1}^n
	[\mathbf{Q}_{\gm\dl}]_{y_\ell} \mathbf{Q}_{\al\be}
	\w^i_{y_\al}\w^j_{y_\gm}\w^j_{y_\dl} .
\end{align*}
Note that the derivatives $(J_n\Phi)_{y_\ell}$ and $\mathbf{Q}_{y_\ell}$
are Lipschitz continuous on the whole domain $B_{\dl}\times(-\eta,\eta)$ for any $\ell=1,2,\dots,n-1$.
According to the $C^{1,\al}$-regularity of $\w$, the coefficients
$A^{ij}_{\alpha\be}$ 
and the term $f_i^\alpha$ appearing in \eqref{Schauder-system}
are H\"older continuous, while the coefficients $b^{ij}_{\alpha}$
and the inhomogeneities $c^i$ are bounded.
Moreover, for any $\xi\in\rr^{Nn}$, by \eqref{growth-a-delta}$_{1,2}$ we have that 
\[
\sum_{\al,\be=1}^n\sum_{i,j=1}^N A^{ij}_{\alpha\be}\xi^i_\alpha\xi^j_\be\ge\tfrac{m}{c}\big(\varep+|D\w|^2_{\mathbf{Q}}\big)^{\frac{p-2}2}|\xi|^2.
\]
Consequently the interior Schauder estimates from Theorem~\ref{Thm:Schauder} yield the H\"older continuity of the spatial gradient 
$Dv$ for some proper H\"older exponent. In particular, $\w_{y_\alpha y_\beta}$
is locally H\"older continuous on $\mathcal{Q}_{\dl,\eta}$, provided $\alpha+\beta<2n$.

Likewise, we may differentiate \eqref{Euler-pointwise} with respect to $t$.
This procedure becomes legitimate if we can show $D_{y}\partial_t \w\in L^2_{\loc}(\mathcal{Q}_{\dl,\mu})$.
Thanks to \eqref{time-deriv}, this can be done  by working with the difference quotient of $w$ in the time variable.
Indeed, let $h>0$ and define the finite difference in time by
\[
	\tau_h \w(t):=\w(t+h)-\w(t).
\]
Here and in the sequel we keep silent of the dependence of $\w$ on $x$.
Taking finite differences in the time variable of \eqref{Euler-pointwise} we obtain
that the parabolic system
\begin{equation}\label{diff-quo}
\begin{aligned}
	\tau_h\pl_t \w^i -
	&
	\sum_{\al,\be=1}^n \Big[\tau_h \big[{\bf a}_\eps(|D\w|_{{\bf Q}}){\bf Q}_{\al\be}\w^i_{y_\al}\big]\Big]_{y_\be}\\
	&=
	\sum_{\al,\be=1}^n \tau_h\big[{\bf a}_\eps(|D\w|_{\bf Q}){\bf Q}_{\al\be}\w^i_{y_\al}\big] 
	\frac{[J_n\Phi]_{y_\be}}{J_n\Phi},
\end{aligned}
\end{equation}
is satisfied weakly in  $\mathcal{Q}_{\dl,\eta}$.
Next, for fixed $t$ and $h$, we introduce the quantity
\[
\Dl(s):=s D\w(t+h)+(1-s)D\w(t)\in\R^{Nn}
\]
whose entries are $
\Dl_\al^i(s):=s\w_{y_\al}^i(t+h)+(1-s)\w_{y_\al}^i(t),
$
and calculate
\begin{align*}
	\tau_h\big[ {\bf a}_\eps&(|D\w|_{{\bf Q}}){\bf Q}_{\al\be}\w^i_{y_\al}\big]\\
	&= 
	\tau_h \w^{j}_{y_\al}
	\int_0^1\left[{\bf a}_\eps\big(|\Dl(s)|_{\bf Q}\big)
	{\bf Q}_{\al\be}\dl^{ij} +
	\frac{{\bf a}_\eps^{\prime}(|\Dl(s)|_{\bf Q})}{|\Dl(s)|_{\bf Q}}
	{\bf Q}_{\al\gm}\Dl^j_\gm(s){\bf Q}_{\be\dl}\Dl^i_\dl(s)\right]\d s\\
	&=:
	\tau_h \w^{j}_{y_\al} \mathcal{A}_{\al\be}^{ij},
\end{align*}
for $\beta=1,\ldots,n$ and $i=1,\ldots,N$. 
It is not hard to verify that the matrix $\mathcal{A}_{\al\be}^{ij}$
satisfies
\begin{align*}
&\sum_{\al,\be=1}^n\sum_{i,j=1}^N\mathcal{A}^{ij}_{\alpha\be}\xi^j_\alpha\xi^i_\be\ge\tfrac{m}{c}|\xi|^2\int_0^1\big(\varep^2+|\Dl(s)|^2_{\bf Q}\big)^{\frac{p-2}2}\,\d s
\ge C_o|\xi|^2
\end{align*}
and
\begin{align*}
&|\mathcal{A}^{ij}_{\alpha\be}|\le cM \int_0^1\big(\varep^2+|\Dl(s)|^2_{\bf Q}\big)^{\frac{p-2}2}\,\d s\le C_1
\end{align*}
for some  positive constants $C_o$ and $C_1$ depending on $p$, $m$, $M$, $c$, $\varep$, and $\|D\w\|_{L^\infty}$.

We may test  \eqref{diff-quo} by $\tau_h \w^i \z^2$ with $\z\in C^1_0(\mathcal{Q}_{\dl,\eta})$.
Employing the above growth conditions on $\mathcal{A}^{ij}_{\alpha\be}$ and the fact that $\partial_t \w \in L^2_{\loc}(\mathcal{Q}_{\dl,\eta})$, a standard calculation gives
\begin{align*}
	\iint_{\mathcal{Q}_{\dl,\eta}} \z^2|\tau_h D\w|^2\,\d y\d t
	\le 
	C h^2
\end{align*}
for some constant $C$ with dependence only on $C_o$, $C_1$, $\Lm$, $\|\z\|_{L^\infty}$, $\|D\z\|_{L^\infty}$, $\|\partial_t\z\|_{L^\infty}$ and $\|\partial_t \w\|_{L^2(\spt\zeta)}$
but independent of $h$.
Passing to the limit in the above estimate as $h\downarrow0$,
we conclude that $D\partial_t \w\in L^2_{\loc}(\mathcal{Q}_{\dl,\eta})$ as promised. Therefore, we may differentiate \eqref{Euler-extension} with respect to $t$ and obtain, denoting $\tilde v:=\partial_t \w$, that
\begin{equation*}
	\partial_t \tilde v^i -
	\sum_{\al,\be=1}^n\sum_{j=1}^N\big[A^{ij}_{\al\be}\tilde v^j_{y_\al}\big]_{y_\be}
	=
	\sum_{\al=1}^n\sum_{j=1}^N b^{ij}_\alpha \tilde v^j_{y_\alpha}
	\qquad\text{for $i=1,2,\dots, N$}
\end{equation*}
in $\mathcal{Q}_{\dl,\mu}$, where $A^{ij}_{\alpha\be}$ and
$b^{ij}_\alpha$ are defined in \eqref{def_A} and \eqref{def_b}, respectively.
Then the interior Schauder estimates from Theorem~\ref{Thm:Schauder} yield the local  H\"older
continuity of 
$\partial_t D\w$ on
$\mathcal{Q}_{\dl,\mu}$.


To obtain H\"older regularity for $\w_{y_n y_n}$, we turn back to
\eqref{transformed-Euler} in $\mathcal{Q}_{\dl,\mu}^-$.
Let us write it in non-divergence form and keep the terms with $\w^i_{y_ny_n}$
on the left-hand side, while we put all other terms on the right-hand side.
As usual, we will omit $\Phi$ on $\mathbf{Q}$.
In this way, we may obtain an algebraic, linear system
\begin{equation}\label{linear-system}
\sum_{j=1}^N B^{ij}\hat u^j_{y_ny_n}=g^i\quad\text{ in }\mathcal{Q}_{\dl,\mu}^-
\mbox{ for $i=1,2,\dots, N$,}
\end{equation}
where
\[
	B^{ij}
	=
	\mathbf{a}_\eps(|D\w|_{\mathbf{Q}})\mathbf{Q}_{nn}\dl^{ij} +
	\frac{\mathbf{a}_\eps^\prime(|D\w|_{\mathbf{Q}})}{|D\w|_{\mathbf{Q}}}
	\mathbf{Q}_{n\gm} \w^j_{y_\gm}\mathbf{Q}_{\al n}\w^i_{y_\al}
\]
and the right-hand side $g^i$ is a combination of first derivatives, second derivatives of $\w$
excluding $\w_{y_ny_n}$, together with $\mathbf{Q}$, $J_n\Phi$ and their first derivatives.
As a result, $g^i$ is  H\"older continuous for all $i=1,2,\dots, N$.
On the other hand, we observe that the matrix $(B^{ij})$ is positive definite and H\"older continuous in
the closure of $\mathcal{Q}_{\dl,\mu}^-$, provided we choose $\dl$ and $\mu$
sufficiently small. As a result, $\w^i_{y_ny_n}$ can be solved from the algebraic, linear system \eqref{linear-system},
and is also H\"older continuous in the closure of $\mathcal{Q}_{\dl,\mu}^-$. 
Hence we have shown that $\w_{y_iy_j}$ is H\"older continuous in the closure of $\mathcal{Q}_{\dl,\mu}^-$
for all $i,j=1,2,\dots,n$. Consequently, the same fact holds for $\partial_t \w$ due to the system \eqref{transformed-Euler}.
Transforming back to $u_\varep$ we obtain that $\pl_t u_\varep$ and $D^2 u_\varep$
are H\"older continuous up to the lateral boundary $\{\pl\Om_\varep\cap N_o\}\times[\tau,T]$
for some $\tau>0$.

The sketched procedure can be iterated to give even higher regularity. To this
end, we successively differentiate the linear
system~\eqref{Schauder-system} in tangential directions and with
respect to time and apply the Schauder estimate from
Theorem~\ref{Thm:Schauder}. This yields the H\"older continuity for all
derivatives expect from the ones in normal directions. The H\"older
regularity of the remaining derivatives can then be deduced from the
system~\eqref{transformed-Euler} on $\bQ_{\delta,\eta}^-$. In this
way, we inductively deduce $\w\in C^{k,\alpha}_{\loc}(\bQ_{\delta,\eta})$ for any $k\in\N$, which yields
the desired smoothness of the approximating solutions $u_\eps$.

\end{document}